\documentclass[12pt]{article}
\usepackage{amsmath,amssymb,amsthm,esint,dsfont,tikz}
\usepackage[english]{babel}
\usepackage[margin=3cm]{geometry}
\usepackage{csquotes}
\usepackage{hyperref}
\usepackage{mathtools}
\usepackage{graphicx}
\usepackage{pdfsync}

\newcommand{\eps}{\varepsilon}

\newcommand{\R}{\mathbb R}
\newcommand{\C}{\mathbb{C}}
\newcommand{\NN}{\mathbb{N}}
\newcommand{\ZZ}{\mathbb{Z}}

\newcommand{\dvr}{\mathrm{div}}
\newcommand{\tr}{\mathrm{tr}\,}

\newcommand{\beq}{\begin{equation}}
\newcommand{\eeq}{\end{equation}}
\newcommand{\bal}{\begin{align*}}
\newcommand{\eal}{\end{align*}}

\newcommand{\grad}{\nabla}

\newcommand{\wto}{\rightharpoonup}

\def\leq{\leqslant}
\def\div{{\rm div}\,}
\newcommand{\Om}{\Omega}
\newcommand{\Ipm}{\mathcal{I}}

\newcommand{\DDD}{\mathcal{D}}
\newcommand{\np}{n_i^+}
\newcommand{\nm}{n_i^-}

\newcommand{\no}{n_j^0}

\newcommand{\be}{\begin{equation}}
\newcommand{\ee}{\end{equation}}
\newcommand{\nnn}{\nonumber}

\newtheorem{theorem}{Theorem}[section]
\newtheorem{proposition}[theorem]{Proposition}
\newtheorem{lemma}[theorem]{Lemma}
\newtheorem{corollary}[theorem]{Corollary}

\newtheorem{remark}[theorem]{Remark}

\numberwithin{equation}{section}

\begin{document}

\markboth{S. Alama \& L. Bronsard \& D. Golovaty} {Title TBA}

\title{Thin Film Liquid Crystals \\ 
with Oblique Anchoring and Boojums}

\author{Stan Alama\footnote{Department of Mathematics and Statistics, McMaster University,
Hamilton, ON, L8S 4K1, Canada, \texttt{alama@mcmaster.ca}},  Lia Bronsard\footnote{Department of Mathematics and Statistics, McMaster University,
Hamilton, ON, L8S 4K1, Canada, \texttt{bronsard@mcmaster.ca}},
 Dmitry Golovaty\footnote{Department of Mathematics, University of Akron, \texttt{dmitry@uakron.edu}}}

\maketitle


\vskip 0.25truein

\begin{abstract}
We study a two-dimensional variational problem which arises as a thin-film limit of the Landau-de Gennes energy of nematic liquid crystals.  We impose an oblique angle condition for the nematic director on the boundary, via boundary penalization (weak anchoring.)  We show that for strong anchoring strength (relative to the usual Ginzburg-Landau length scale parameter,) defects will occur in the interior, as in the case of strong (Dirichlet) anchoring, but for weaker anchoring strength all defects will occur on the boundary.  These defects will each carry a fractional winding number; such boundary defects are known as ``boojums''.  The boojums will occur in ordered pairs along the boundary; for angle $\alpha\in (0,\frac{\pi}{2})$, they serve to reduce the winding of the phase by steps of $2\alpha$ and $(2\pi-2\alpha)$ in order to avoid the formation of interior defects.  We determine the number and location of the defects via a Renormalized Energy and numerical simulations.
\end{abstract}

%
%
%

\section{Introduction}

In this paper we study minimizers of a variational problem motivated by the study of defects in a nematic liquid crystal.  We consider a two-dimensional setting, arising in a thin-film reduction of the three dimensional Landau--de Gennes model to two dimensions.  The special feature of our problem is in the boundary condition imposed, in which energy minimization prefers that the nematic director be oblique to the normal to the boundary with a prescribed angle.  In three dimensions, such a fixed angle condition constrains the nematic director to lie on a cone coaxial with the boundary normal; in the plane, this reduces to demanding that the director make an angle of $\pm\alpha$ with respect to the normal vector at each boundary point.  In our model this will be accomplished by imposing {\em weak anchoring conditions} on the domain boundary, that is, by adding a penalization term to the energy which favors oblique director configurations. We refer the reader to \cite{VL83,Sluckin,TAS12} for the detailed discussion of anchoring within the context of the Landau-de Gennes theory and relevant physical observations. 

We begin by describing the variational problem in mathematical terms, and stating our main result in Theorem~\ref{main}.  Let $\Om\subset\R^2$ be a bounded, simply connected domain with $C^2$ boundary $\Gamma:=\partial\Om$, carrying unit exterior normal vector $\nu$.  Our energy functional is the classical Ginzburg-Landau functional, modified by the addition of a surface energy term which enforces the desired weak anchoring.  Let $\alpha\in (0,{\pi\over 2})$ be fixed throughout the paper.  As usual, we associate $\C\simeq \R^2$, with scalar product $(u,v)=\Re[u\, \bar v]$ and wedge product $u\wedge v=(iu,v)$ for $u,v\in\C$.  Let $g:\ \Gamma\to \mathbb{S}^1\subset\C$ be a given $C^1$ smooth function on the boundary.  We assume that 
$$  \mathcal{D}=\deg(g;\Gamma)>0,  $$
and take a smooth lifting $\gamma: \ \Gamma\to \R$, $g=e^{i\gamma}$.  In the physical context, $g$ would represent the unit normal vector field on $\Gamma$; in the orientable GL or Ericksen models, it would then have degree $\mathcal{D}=1$.  In the reduction from the 3D Landau-de Gennes model the complex order parameter doubles the phase of the director, and so we would have $\mathcal{D}=2$.  (See the discussion below.)  However, we may take $g$ to be any smooth $\mathcal{S}^1$-valued map in our analysis.

Our energy then takes the form:
\begin{equation}
\label{eq:E}
E_{\eps}^{g,\alpha}(u) 
	:= \frac{1}{2} \int_{\Omega} \left( |\grad u|^2 + \frac{1}{2\eps^2} \big( |u|^2-1\big)^2\right)  \, dx 
	  + \frac{\Upsilon}{ 2} \int_{\Gamma}
	   W(u,g)\, ds,
\end{equation}
with boundary anchoring energy density $W$ given by:
\begin{equation}
\label{eq:W}
W(u,g):= {\frac12} (|u|^2-1)^2 + [(u,g)-\cos\alpha]^2 .
\end{equation}
The weak anchoring strength $\Upsilon$ is assumed to depend on the length scale parameter $\eps$,  
$$\Upsilon=\Upsilon(\eps)=\eps^{-s} \quad\text{ 
for $s\in (0,1]$.}
$$

The effect of the weak anchoring may be inferred from the form of $W$.  As $\eps\to 0$ we expect that $W(u_\eps,g)\to 0$ almost everywhere on $\Gamma$.  At points $y\in\Gamma$ at which $W(u_\eps(y),g(y))\to 0$, we would have $|u_\eps|\to 1$ and $(u,g)\to \cos \alpha$, that is, $u_\eps\simeq g\, \exp(\pm i\alpha)$.  
If $g$ represents the unit normal vector field, this is the desired cone condition for a nematic.  If there are no defects on $\Gamma$ then the phase shift $\pm\alpha$ is uniformly chosen on $\Gamma$, and $u_\eps$ will effectively satisfy a Dirichlet boundary condition with degree $\mathcal{D}$, for which there must be interior defects, which will be vortices.
However, energy minimization may prefer to accept defects on $\Gamma$ in order to avoid the energy cost of interior vortices.  In this case, the phase of $u_\eps$ must jump at defect points in order to ``unwind'' its phase so as to have degree zero on $\Gamma$.  The form of $W$ allows the phase to unwind by steps of $2\alpha$, $(2\pi-2\alpha)$ or $2\pi$.  This suggests that there are {\em three} distinct types of boundary defects. The first two are {\em boojums}---defects with fractional degree. In correspondence with the size of the jump in angle, we call these a ``light'' boojum and a ``heavy'' boojum, respectively.  The last type is a boundary vortex, of the sort studied in \cite{ABGS}, with integer degree.
Our result states that for very strong anchoring (larger $s$,) minimizers prefer interior vortices, while for milder anchoring (smaller $s$,) we will obtain light-heavy boojum pairs on $\Gamma$ and no interior vortices.  The threshold value for $s$ will depend on the angle $\alpha$.  In no case are boundary vortices (of integer degree) preferred.  

In order to state our result, we define
\beq\label{cal}
C_{\alpha}:= \left\{\left(\frac{\alpha}{\pi}\right)^2+\left(1-\frac{\alpha}{\pi}\right)^2\right\},
\eeq
a constant which will appear often in our calculations of the energy of boundary defects of solutions.  Note that $\frac12< C_{\alpha}<1$ for all $\alpha\in (0,\pi/2)$.

\begin{theorem}\label{main}
\begin{itemize}
\item[(a)] If $1\ge s>{1\over 2C_\alpha}$, then $\exists \, \DDD $ points, $p_1, \dots, p_{\DDD}\in \Omega$ and a subsequence $\eps_n\to 0$ such that the minimizers $u_{\eps_n}$ of $E_{\eps_n}^{g,\alpha}$ satisfy
$$
u_{\eps_n}\to u_* \text{ in } H^1_{\text{loc}}\cap C^{1,\alpha}_{\text{loc}}(\bar \Omega\setminus\{p_1,\dots,p_{\DDD}\}),
$$
with $u_*$ an $S^1$-valued harmonic map with $W(u,g)=0$ on $\Gamma$, and $p_i$ is a vortex of degree $1, \forall i$.

\item[(b)] If $s<{1\over 2C_\alpha}$, there $\exists \, (2\DDD)$ points $y_1, \tilde y_1, y_2, \tilde y_2, \dots, y_{\DDD}, \tilde y_{\DDD} \in \Gamma$, ordered along the boundary curve, and a subsequence $\eps_n\to 0$ such that the minimizers $u_{\eps_n}$ of $E_{\eps_n}^{g,\alpha}$ satisfy
$$
u_{\eps_n}\to u_* \text{ in } H^1_{\text{loc}}\cap C^{1,\alpha}_{\text{loc}}(\bar \Omega\setminus\{y_1,\tilde y_1, \dots,y_{\DDD},\tilde y_{\DDD} \}),
$$
with $u_*$ an $S^1$-valued harmonic map with $W(u,g)=0$ on $\Gamma\setminus \{y_1,\tilde y_1, \dots,y_{\DDD},\tilde y_{\DDD} \}$, and $y_j, \tilde y_j$ is a boojum pair of total degree $-1$.
\end{itemize}
\end{theorem}
In the critical case $s={1\over 2C_\alpha}$ the situation is more delicate, as interior and boundary defects will have the same energy to highest order $O(|\ln\eps|)$, and one may have coexistence of the two species of defect depending on the geometry of the domain and choice of boundary map $g$.  As in \cite{ABGS} we expect that by introducing a coefficient $\Upsilon=K\eps^{-s}$ in the weak anchoring strength, the cross-over between boundary and interior vortices may be observed by varying $K$ when $s={1\over 2C_\alpha}$, but we do not pursue this direction in the present paper.

As in the classical work \cite{BBH2} on the Ginzburg--Landau model with Dirichlet boundary conditions, the location of the defects may be determined by minimizing a finite dimensional Renormalized Energy.  This will be briefly discussed in Section~\ref{RNE}, after the proof of Theorem~\ref{main}.

A related model is that of a thin ferromagnetic film as obtained in appropriate limiting regime by DeSimone, Kohn, Muller and Otto (\cite{DKMO}). This limiting ferromagnetic thin film was studied by Moser (\cite{Moser}), and by Kurzke (\cite{Ku}) in certain settings. In those problems, they impose  tangential weak anchoring conditions (i.e. $\alpha=0$) and find critical anchoring strength (though with a different critical exponent,) at which boundary vortices are favored over interior vortices. In our case we impose oblique anchoring conditions which reveal boojums defects. 

In the context of nematic liquid crystals, boojums were also observed in \cite{GNSV,GKLNS} at interfaces between the nematic and the isotropic phases. In the limit of the small nematic correlation length, the assumption of a large splay elastic constant in \cite{GNSV,GKLNS} led to the tangency condition of the director on the interface and appearance of boojums. In this setting, the boojums are also associated with interface singularities because the interface location is one of the unknowns of the problem.

The rest of the paper is organized as follows: in Section 2 we describe how to obtain the above variational problem as a thin-film limit of the Landau-de Gennes energy of nematic liquid crystals. In Section 3, we present an upper bound on the energy of minimizers, as well as {\it a priori} pointwise bounds for all solutions of the Euler-Lagrange equations. In Section 4, we present an $\eta$ compactness result adapted to handle boundary defects and use it to define the ``bad balls" and show that they are contained in a finite number of very small balls. Next in Section 5, we classify the ``bad balls" as interior vortex, boundary vortex, light and heavy boojums. In Section 6, we obtain an energy lower bound for each type of defects and prove an important new ``degree Lemma" (Lemma~\ref{newdeglem}) which will be essential in proving the lower bound on the energy of boundary defects in terms of the degree of the boundary data.
In section 7, we put everything together and prove our main theorem, modifying the technique of vortex ball analysis introduced by Jerrard \cite{Jerrard} and Sandier \cite{Sandier}. In Section 8, we formally derive the associated Renormalized Energy and, finally, in Section 9 we present numerical examples of possible defect configurations.

\section{Modeling nematic thin films}
\label{sec:nema}

In this section we motivate our variational problem via the Landau-de Gennes theory of nematic liquid crystals, in a limiting thin-film regime.

\subsection{The $Q$-tensor}
\label{s:qtens}
A nematic liquid crystal occupying a region $\Omega\in\mathbb R^3$ can be described by a $2$-tensor-valued field which can be thought of as the field $Q:\mathbb R^3\to M_{sym}^{3\times3}$ of $3\times  3$ symmetric, traceless matrices  \cite{MN}. It immediately follows that $Q$ has a mutually orthonormal eigenframe $\left\{\mathbf{e}_1,\mathbf{e}_2,\mathbf{e}_3\right\}$ and three real eigenvalues satisfying $\lambda_1+\lambda_2+\lambda_3=0$. The tensor $Q(x)$ represents the second moment of the orientational distribution in $\mathbb S^2$ of the nematic molecules near $x\in\mathbb R^3$, hence its eigenvalues must satisfy the constraints
\begin{equation}
\label{eq:bnds}
\lambda_i\in[-1/3,2/3],\ \mathrm{for}\ i=1,2,3.
\end{equation}

Suppose that $\lambda_1=\lambda_2=-\lambda_3/2.$ Then the liquid crystal is in a {\em uniaxial nematic} state and \begin{equation}Q=-\frac{\lambda_3}{2}\mathbf{e}_1\otimes\mathbf{e}_1-\frac{\lambda_3}{2}\mathbf{e}_2\otimes\mathbf{e}_2+
\lambda_3\mathbf{e}_3\otimes\mathbf{e}_3=S\left(\mathbf{n}\otimes\mathbf{n}-\frac{1}{3}\mathbf{I}\right),\label{uniaxial}
 \end{equation}
 where $S:=\frac{3\lambda_3}{2}$ is the uniaxial nematic order parameter and $\mathbf{n}=\mathbf{e}_3\in\mathbb{S}^2$ is the nematic director. If there are no repeated eigenvalues, the liquid crystal is said to be in a {\em biaxial nematic} state and
\begin{multline}
Q=\lambda_1\mathbf{l}\otimes\mathbf{l}+\lambda_3\mathbf{n}\otimes\mathbf{n}-\left(\lambda_1+\lambda_3\right)\left(\mathbf{I}-\mathbf{l}\otimes\mathbf{l}-\mathbf{n}\otimes\mathbf{n}\right)\\=S_1\left(\mathbf{l}\otimes\mathbf{l}-\frac{1}{3}\mathbf{I}\right)+S_2\left(\mathbf{n}\otimes\mathbf{n}-\frac{1}{3}\mathbf{I}\right),
\label{biaxial}\end{multline}
where $S_1:=2\lambda_1+\lambda_3$ and $S_2=\lambda_1+2\lambda_3$ are biaxial order parameters. 

For so-called {\em thermotropic} liquid crystals nematic states are typically observed at low temperatures. On the contrary, at high temperatures, these materials loose orientational order and become {\em isotropic}. The corresponding state is represented by $Q=0$ so that $\lambda_1=\lambda_2=\lambda_3=0$.

\subsection{Landau-de Gennes model}
Within the $Q$-tensor theory, the bulk elastic energy density of a nematic liquid crystal is given by
\begin{equation}
\label{elastic}
f_e(\nabla Q):=\sum_{j=1}^3\left\{\frac{L_1}{2}{|\nabla Q_j|}^2+\frac{L_2}{2}\left(\dvr{Q_j}\right)^2+\frac{L_3}{2}\nabla Q_j\cdot \nabla Q_j^T\right\},
\end{equation}
while the bulk Landau-de Gennes energy density is
\begin{equation}
\label{eq:LdG}
f_{LdG}(Q):=a\,\mathrm{tr}\left(Q^2\right)+\frac{2b}{3}\mathrm{tr}\left(Q^3\right)+\frac{c}{2}\left(\mathrm{tr}\left(Q^2\right)\right)^2,
\end{equation}
cf. \cite{MN}. Here $Q_j,\,j=1,2,3$ is the $j$-th column of the matrix $Q$ and $A\cdot B=\tr\left(B^TA\right)$ is the dot product of two matrices $A,B\in M^{3\times3}.$ The coefficient $a=a_0\left(T-T_*\right)$ in \eqref{eq:LdG} is temperature-dependent and negative for sufficiently low temperatures, while $c>0$.
The potential \eqref{eq:LdG} is designed to depend only on the eigenvalues of $Q$ and its form guarantees that the isotropic state $Q\equiv 0$ yields the global minimum of $f_{LdG}$ at high temperatures while a uniaxial state of the form \eqref{uniaxial} gives the global minimum when temperature is sufficiently low, cf. \cite{apala_zarnescu_01,MN}. In what follows we set the temperature to be low enough so that the minimizers of $f_{LdG}$ are uniaxial. Note that by adding an appropriate constant to $f_{LdG}$ we can assume the global minimum value of zero for $f_{LdG}.$

Now consider a nematic sample occupying a thin domain $\Omega_h:=\Omega\times(-h,h)\subset\mathbb R^3$, where $\Omega\subset\mathbb R^2$ and $h\ll1$. The equilibrium nematic configuration should minimize the bulk energy subject to the appropriate boundary conditions on $\partial\Omega_h$.  There are two possible alternatives. The first option is to impose Dirichlet boundary conditions on $Q$---also known as strong anchoring conditions---that fix the alignment of nematic molecules on $\partial\Omega_h$. The second option is to consider {\em weak anchoring}, that is, to specify the surface energy on the boundary of the nematic sample. The molecular orientations on the boundary are then determined as a part of the minimization procedure. 

In this paper we consider a two-dimensional variational problem for $Q$ that can be obtained---following \cite{GSM2}---via a rigorous dimension reduction procedure by taking the limit $h\to0$. Briefly, as in \cite{GSM2}, suppose that weak anchoring conditions are specified on the top and the bottom surfaces $\Omega\times\{-h,h\}$ of the nematic film $\Omega_h$. The anchoring energy density has the form
\begin{equation}
\label{fs}
f^{(1)}_s(Q,\hat{z})=\alpha\left(\left[(Q\hat{z}\cdot\hat{z})-\beta\right]^2+{\left|\left(\mathbf{I}-\hat{z}\otimes\hat{z}\right)Q\hat{z}\right|}^2\right),
\end{equation}
for any $Q\in\mathcal A$, where $\alpha>0$, $\beta\in\mathbb{R}$,
\begin{equation}
\label{eq:cala}
\mathcal A:=\left\{Q\in M^{3\times3}_{sym}:\mathrm{tr}\,{Q}=0\right\},
\end{equation}
and $\hat{z}$ is normal to the surface of the film. This form of the anchoring energy requires that a minimizer of $f^{(1)}_s$ has $\hat{z}$ as an eigenvector with corresponding eigenvalue equal to $\beta$. 

On the remaining part $\Gamma\times(-h,h)$ of $\partial\Omega_h$ we impose different weak anchoring conditions with the nonnegative surface energy density $f_s^{(2)}(Q,g),$ where $\Gamma=\partial\Omega,$ the uniaxial data $g\in H^{1/2}(\Gamma\times(-h,h);\mathcal{A})$ does not vary in the direction normal to $\Gamma,$ and $f_s^{(2)}$ is a smooth function of its arguments.

The Landau-de Gennes energy can then be obtained by combining together \eqref{elastic}, \eqref{eq:LdG}, and \eqref{fs} so that
\begin{multline}
\label{energy}
E_h(Q):=\int_{\Omega_h}\left\{f_e(\nabla Q)+f_{LdG}(Q)\right\}\,dV\\+\int_{\Omega\times\{-h,h\}}f_s^{(1)}(Q,\hat{z})\,dS+\int_{\Gamma\times(-h,h)}\tilde f_s^{(2)}(Q,g)\,dS.
\end{multline}

In what follows we will assume that the elastic constants $L_2=L_3=0$ and $L_1=L>0$; this corresponds to the so called equal elastic constants case where the equality of the constants refers to those in the Oseen-Frank model. The elastic energy density we consider is thus given by
\begin{equation}
\label{elastic_eq}
f_e(\nabla Q):=\frac{L}{2}{|\nabla Q|}^2.
\end{equation}

The problem can be nondimensionalized by scaling the spatial coordinates
\[\tilde{x}=\frac{x}{D},\ \tilde{y}=\frac{y}{D},\ \tilde{z}=\frac{z}{h},\ \]
where $D:=\mathrm{diam}(\omega)$. Set $\xi=\frac{L}{2D^2},$ $\delta=\frac{h}{D}$ and introduce $\tilde{f}_e(\nabla Q):=\frac{1}{\xi}f_e(\nabla Q).$ Dropping tildes, we obtain
\[
f_e(\nabla Q): =Q_{im,j}Q_{im,j}+\frac{1}{\delta^2}Q_{im,3}Q_{im,3},
 \]
where the indices $i,m=1,2,3,$ and $j=1,2.$ Rescaling the Landau-de Gennes potential $\tilde{f}_{LdG}(Q):=\frac{\varepsilon^2}{\xi}f_{LdG}(Q)$ and ignoring tildes again gives
\begin{equation}
f_{LdG}(Q)=2A\,\mathrm{tr}\left(Q^2\right)+\frac{4}{3}B\,\mathrm{tr}\left(Q^3\right)+{\left(\mathrm{tr}\left(Q^2\right)\right)}^2,
\end{equation}
where $A=\frac{a}{c},$ $B=\frac{b}{c},$ and $\varepsilon=\sqrt{\frac{2\xi}{c}}.$ We also let $\tilde\alpha=\frac{\alpha}{\xi D}$ and set
\[\tilde{f}_s^{(1)}(Q,\hat z):=\frac{1}{\xi D}f_s^{(1)}(Q,\hat z),\qquad \tilde{f}_s^{(2)}(Q,g):=\frac{1}{\xi D}f_s^{(2)}(Q,g)\]
to obtain the expressions for the nondimensionalized surface energies. 

Finally, introducing the non-dimensional energy $F_\delta[Q]:=\frac{2}{Lh}E[Q]$ and dropping all tildes, we find that
\begin{multline}
\label{nden}
F_\delta(Q)=\int_{\Omega\times(-1,1)}\left(f_e(\nabla Q)+\frac{1}{\varepsilon^2}f_{LdG}(Q)\right)\,dV\\+\frac{1}{\delta}\int_{\Omega\times\{-1,1\}}f^{(1)}_s(Q,\hat z)\,dA+\int_{\Gamma\times\{-1,1\}}f^{(2)}_s(Q,g)\,dA.
\end{multline}

We now define the space 
\begin{equation}
\mathcal{H}:=\left\{Q\in H^1(\Omega\times (-1,1);\mathcal{A}):\; \frac{\partial Q}{\partial z}\equiv 0\;\mbox{a.e}.,\;f_s(Q(x),\hat{z})=0\mbox{ a.e. in }\Omega\right\}
\label{Hg}
\end{equation}
and let $F_0:H^1(\Omega\times (-1,1);\mathcal{A})\to\R$ be given by

 \begin{equation}
   \label{eq:F0}
 F_0(Q):=\left\{
     \begin{array}{ll}
       2\int_\Omega\left\{{|\nabla_{xy} Q|}^2+\frac{1}{\varepsilon^2}f_{LdG}(Q)\right\}\,dx+ 2\int_\Gamma f^{(2)}_s(Q,g)\,dA& \mbox{ if }Q\in \mathcal{H}, \\
       +\infty & \mbox{ otherwise. }
     \end{array}
 \right.
 \end{equation}
 
The following theorem can be proved in the same way as its analog in \cite{GSM2}. 

\begin{theorem}
\label{t1}
Fix $g\in H^{1/2}\left(\partial\Omega;\mathcal{A}\right)$ that does not vary in the direction normal to $\Gamma$. Let $F_\delta$ be given by \eqref{nden}. Then $\Gamma$-$\lim_\delta{F_\delta}=F_0$ in the weak  $H^1$  topology.
Furthermore, if a sequence $\left\{Q_\delta\right\}_{\delta>0}\subset H^1(\Omega\times (-1,1);\mathcal{A})$ satisfies a uniform energy bound $F_\delta[Q_\delta]<C_0$ then there is a subsequence $\{{Q}_{\delta_j}\}$ such that ${Q}_{\delta_j}\rightharpoonup Q$ as $\delta_j\to 0$
for some $Q\in \mathcal{H}.$
\end{theorem}

From now on we use the following representation of $Q\in H$ invoked, for example, in \cite{GSM} and \cite{BPP}:
\begin{equation}
  \label{eq:pr}
  Q=\left(
    \begin{array}{ccc}
      p_1-\frac{\beta}{2} & p_2 & 0 \\
      p_2 & -p_1-\frac{\beta}{2} & 0 \\
      0 & 0 & \beta
    \end{array}
\right).
\end{equation}
It is a convenient change of variables in the setting when one eigenvector of the $Q$-tensor is parallel to the $z$-axis. For simplicity, we also assume that $\beta=-1/3$ and that a uniaxial tensor minimizing $W$ has eigenvalues $-1/3,\ -1/3,$ and $2/3$. Then 
\begin{equation}
  \label{eq:pr1}
  Q(p_1,p_2)=\left(
    \begin{array}{ccc}
      \frac{1}{6}+p_1 & p_2 & 0 \\
      p_2 & \frac{1}{6}-p_1 & 0 \\
      0 & 0 & -\frac{1}{3}
    \end{array}
\right).
\end{equation}
and
\[f_{LdG}(Q)={\left(\tr Q^2\right)}^2-\frac{4}{3}\tr Q^3-\frac{2}{3}\tr Q^2+\frac{8}{27},\]
where the constant $8/27$ was added to ensure that the minimum value of $W$ is equal 0. The potential function can now be written as 
\[\tilde f_{LdG}(p):=W(Q(p))=\frac{1}{4}{\left(4{|p|}^2-1\right)}^2\]
in terms of $p=(p_1,p_2)$. Dropping the subscript $xy$ in \eqref{eq:F0}, we also have that 
\[
{|\nabla Q|}^2=2{|\nabla p|}^2,
\]
so that the bulk contribution to \eqref{eq:F0} takes the form
\begin{equation}
\label{eq:11}
\int_\Omega\left\{{4|\nabla p|}^2+\frac{1}{2\varepsilon^2}{\left(4{|p|}^2-1\right)}^2\right\}\,dx.
\end{equation}

With a slight abuse of notation, we set $f_s^{(2)}(p,g):=f_s^{(2)}(Q(p_1,p_2),g),$ where $g:\partial\Omega\to\mathbb{S}^1$ is fixed. In order to establish the form of $f_s^{(2)}$, we appeal to the Rapini-Papoular form of the surface energy \cite{Sluckin} that in the Oseen-Frank director description can be written as 
\begin{equation}
\label{eq:RP}
\sigma{\left((n\cdot g)^2-\cos^2{\left(\frac{\alpha}{2}\right)}\right)}^2.
\end{equation}
Here $n$ is the nematic director and $\frac{\alpha}2$ is a preferred angle between the director $n$ and the uniaxial data $g$ on the boundary with $\alpha\in (0,\pi)$. We now recall the relationship between the director $n$ and the uniaxial tensor $Q$. Because we assumed that $Q$ is given by \eqref{eq:pr1}, the largest eigenvalue minimizing the potential energy is $2/3$. If the director $n=(n_1,n_2, n_3)$ lies in the $xy$-plane, we have that $n_3=0$ and
\[Q=n\otimes n-\frac{1}{3}I,\]
so that 
\begin{equation}
\label{eq:pn}
n\otimes n=\left(
    \begin{array}{ccc}
      n_1^2 & n_1n_2 & 0 \\
      n_1n_2 & n_2^2 & 0 \\
      0 & 0 & 0
    \end{array}
\right)=\left(
    \begin{array}{ccc}
      \frac{1}{2}+p_1 & p_2 & 0 \\
      p_2 & \frac{1}{2}-p_1 & 0 \\
      0 & 0 & 0
    \end{array}
\right).
\end{equation}
Then
\begin{multline}{(n\cdot g)}^2=(n\otimes n)\cdot (g\otimes g)=\left(\frac{1}{2}+p_1\right)g_1^2+2p_2g_1g_2+\left(\frac{1}{2}-p_1\right)g_2^2\\=\frac{1}{2}+\left(g_1^2-g_2^2\right)p_1+2g_1g_2p_2=\frac{1}{2}+p\cdot \hat g,\end{multline}
where $\hat{g}=\left(g_1^2-g_2^2,2g_1g_2\right)$. It follows that we can write \eqref{eq:RP} as
\begin{equation}
\label{eq:RP2}
f_s^{(2)}(p,\hat g)=\sigma{\left(p\cdot \hat g-\frac{1}{2}\cos{\alpha}\right)}^2.
\end{equation}
This choice is well-motivated physically and clearly favors the desired cone condition for the angle of the director for nematic directors of fixed length, $2|p|=1$. However, when relaxing this constraint via the Ginzburg-Landau functional this boundary energy effectively does not enforce the angle condition when $2|p|<1$.  Indeed, in order to obtain the desired behavior at boundary defects it is necessary to add a term for which the boundary energy is minimized when $|p|=1/2$ lies on the cone of aperture $\alpha$ with the axis that coincides with the normal to $\Gamma.$ With this observation in mind, we replace \eqref{eq:RP2} with
\begin{equation}
\label{eq:RP3}
f_s^{(2)}(p,\hat g)=\frac{\sigma}{4}\left[{\frac12}(|2p|^2-1)^2+{\left(2p\cdot \hat g-\cos{\alpha}\right)}^2\right].
\end{equation}

Defining $u:=2p$ and $E_{\varepsilon}^{g,\alpha}[u]:=\frac{1}{2}F_0[Q(u/2)]$, dropping the hat in $\hat g$, and denoting $\Upsilon:=\sigma/2$ we arrive at the expression \eqref{eq:E} for the energy $E_{\varepsilon}^{g,\alpha}$.

Note that the weak anchoring condition is now coercive: using complex notation, the condition \eqref{eq:RP2} when $u\in \mathbb{S}^1\subset \C$ says that the angle between $\hat g=g^2$ and $u$ is twice that of the angle between the director $n$ and $g$, which is consistent since the phase of $u$ is doubled compared to that of $n$. 

We also note that the main theorem is stated for $\alpha \in (0,\frac{\pi}2)$. The case $\alpha\in (\frac{\pi}2, \pi)$ follows directly from this case, as will be mentioned in the course of the proof.

\section{Upper bounds}
In this section we prove two fundamental estimates:  a rough upper bound on the energy of minimizers, and {\it a priori} pointwise bounds for all solutions of the Euler-Lagrange equations,
\begin{equation}\label{EL}
\left.
\begin{gathered}
-\Delta u + \frac{1}{\eps^2}(|u|^2-1)u =0, \quad\text{in $\Omega$},\\
\frac{\partial u}{\partial \nu} + \Upsilon_{\eps} \Big( (|u|^2-1)u + [(u,g)-\cos\alpha] g\Big)=0, \quad\text{on $\Gamma$},
\end{gathered}\right\}
\end{equation}
where $\Upsilon_{\eps}=\eps^{-s}$, with $0<s\le 1$.

\begin{lemma}\label{upperbound}
Let $g=e^{i\gamma}:\Gamma\to S^1$ be a $C^1$ smooth map, with
$$\DDD=\deg(g;\Gamma)>0,\text{ and } E_\eps^{g,\alpha}:=\min_{H^1(\Omega)}E_\eps^{g,\alpha}.
$$

Recall $C_\alpha=(\frac{\alpha}{\pi})^2+(1-\frac{\alpha}{\pi})^2$, $\alpha\in (0,\frac{\pi}2)$,
\begin{enumerate}
\item[(i)]
If $sC_\alpha\ge \frac12$
then
\beq E_\eps^{g,\alpha}\le \pi\DDD|\ln\eps|+C_1\label{ub1}\eeq

\item[(ii)] 
If $s C_\alpha<\frac12$, then
\beq E_\eps^{g,\alpha}\le 2\pi \DDD s C_\alpha |\ln\eps| +C_1 \label{ub2}\eeq
\end{enumerate}

\end{lemma}

\begin{proof}

In case (i)  we let $v_\eps$ be the energy minimizer of the standard Ginzburg-Landau functional with Dirichlet boundary conditions $u\big|_\Gamma = g e^{i\alpha}$ (so as to make the boundary anchoring energy vanish.) The bound \eqref{ub1}then follows from the work of Bethuel-Brezis-H\'elein \cite{BBH2} since $E_\eps^{g,\alpha}$ becomes the Ginzburg-Landau functional.

In case (ii), we will construct an $S^1$-valued test function with boundary defects. Note that in that case, the upper bound in \eqref{ub2} is smaller than the bound \eqref{ub1}. The construction follows Kurzke \cite{Ku} (see also [\cite{ABGS}, Lemma 3.1]) although our boundary condition is quite different. {This construction is also very helpful in understanding what minimizers should look like.}

Choose $2\DDD$ points $q_1,\dots, q_{2\DDD}\in \Gamma$, well separated, and fix $R>0$ with $R<\frac12|q_i-q_j|\ \forall i\neq j$. We order the points along $\Gamma$ so that the index of $q_i$ increases as $\Gamma$ is traced out counterclockwise. For each $i$ we define $v_\eps^i$ in $\omega_R(q_i)=B_R(q_i)\cap\Omega$ via polar coordinates $(\rho,\theta)$ centered at $q_i$ with $\theta$ measured from the oriented unit tangent $\tau$ to $\Gamma$ at $q_i$. Since $\Gamma$ is smooth, by reducing $R$ (if necessary) we may ensure that $\omega_R(q_i)$ is a polar rectangle, and nearly a half-disk:
$$\omega_R(q_i)=\{(\rho,\theta)|\ \theta_+(\rho)<\theta<\theta_-(\rho), 0<\rho<R\}$$
with $\theta_{\pm}\in C^1$ and $|\theta_+(\rho)|\le c\rho, \ |\pi-\theta_-(\rho)|\le c\rho$.


The odd $q_{2j-1}$ will be ``light" boojums, with phase decreasing by $2\alpha$, while the even $q_{2j}$ will be ``heavy" boojums, with phase decreasing by $2\pi-2\alpha$. Consider the ``light" case; the ``heavy" case will be essentially the same, except for the coefficient.

With $g=e^{i\gamma}$, define phases for $v_\eps^i$ on the two components of $\Gamma\cap B_R(q_i)\setminus\{q_i\}$ around a point $q_i$ (see Figure~\ref{fig:Ar},)
parametrized by
$$\Gamma^{\pm}=\{(\rho,\theta_{\pm}(\rho)):\ 0<\rho\le R\},$$
as follows: let
$$h_{\pm}(\rho) =\gamma(\rho,\theta_{\pm}(\rho))\mp\alpha$$
and
$$\psi(\rho,\theta)=h_+(\rho)\frac{\theta-\theta_-(\rho)}{\theta_+(\rho)-\theta_-(\rho)} + h_-(\rho)\frac{\theta-\theta_+(\rho)}{\theta_-(\rho)-\theta_+(\rho)},$$
which linearly interpolates between them. We introduce a cut-off near $q_i$: $\chi_\eps(\rho) \in C^\infty, $
$$0\le \chi_\eps(\rho)\le 1, \chi_\eps(\rho)=0 \text{ for } 0\le \rho<\eps^s, \chi_\eps(\rho)=1 \text{ for } \rho\ge 2\eps^s, |\nabla\chi_\eps(\rho)|\le c\eps^{-s}.$$

\begin{figure}[!htbp]
\begin{center}
\includegraphics*[width=200pt]{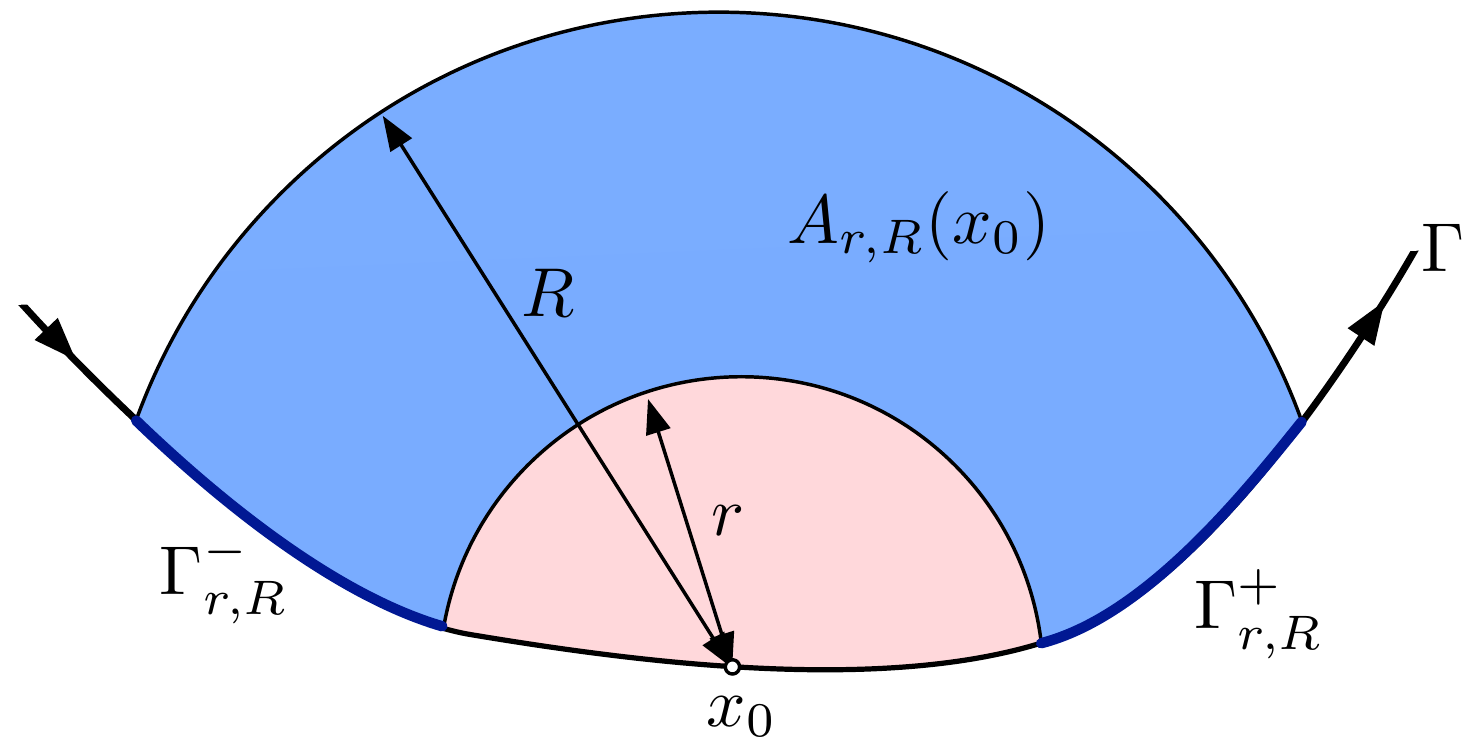}
\end{center}
\caption{Near a boundary point $x_0\in\Gamma$, the disk $\omega_R(x_0)$ and annulus $A_{r,R}$, which separates the boundary $\Gamma\cap A_{r,R}$ into two arcs, $\Gamma^\pm_{r,R}$.}\label{fig:Ar}
\end{figure}

Then we define $v_\eps^i$ in $\omega_R(q_i)$ with $i$ odd by:
$$
v_\eps^i(\rho,\theta)=\exp\left\{ i[\chi_\eps(\rho)\psi(\rho,\theta)+
(1-\chi_\eps(\rho))\gamma(q_i)]\right\}
$$
Note that on $\Gamma^+\setminus B_{2\eps^s}(q_i)$, $v_\eps^i=g\, e^{-i\alpha}$ and on $\Gamma^-\setminus B_{2\eps^s}(q_i)$, $v_\eps^i=g\, e^{i\alpha}$, so that $W(v_\eps^i, g)=0$ on $\Gamma^{\pm}\setminus B_{2\eps^s}(q_i)$.
In particular,
$$\Upsilon_\eps\int_{\Gamma\cap B_R(q_i)} W(v_\eps^i, g)\, ds\le O(1).$$

Also, $|v_\eps^i|=1$ in $\omega_R(q_i)$, so
$$\eps^{-2}\int_{B_R(q_i)} (|v_\eps^i|^2-1)^2\, dx =0.$$

Furthermore, $v_\eps^i$ is smooth and $\eps$-independant in $\omega_R(q_i)\setminus\omega_{\eps^s}(q_i)$
with
$$\int_{\omega_R(q_i)}|\partial_\rho v_\eps^i|^2\le C \ \text{and} \ \quad\int_{\omega_{\eps^s}(q_i)}|\partial_\theta v_\eps^i|^2\le C
$$
so the main contribution to the gradient energy is via $|\partial_\theta v_\eps^i|$ in $\omega_R(q_i)\setminus\omega_{\eps^s}(q_i)$:

\begin{align*}
\frac12\int_{\omega_R(q_i)\setminus\omega_{\eps^s}(q_i)} \frac1{\rho^2}
|\partial_\theta v_\eps^i|^2\, dx &=
\frac12 \int_{\omega_R(q_i)\setminus\omega_{\eps^s}(q_i)} \chi_\eps^2(r)|\partial_\theta \psi(\rho,\theta)|^2\, \frac1{\rho^2}\, dx\\
&\le \frac12\int_{\omega_R(q_i)\setminus\omega_{\eps^s}(q_i)}
\frac{(h_+(\rho)-h_-(\rho))^2}{(\theta_+(\rho)-\theta_-(\rho))^2} \frac1{\rho^2}\, dx\\
&\le\frac12 \int_{\eps^s}^R \frac{(h_+(\rho)-h_-(\rho))^2}{(\theta_+(\rho)-\theta_-(\rho))} \frac1{\rho}\, d\rho\\
&\le \frac12 \int_{\eps^s}^R \frac{((2\alpha)+c\rho)^2}{(\pi-c\rho)}\frac1{\rho}\, d\rho\\
&\le 2\pi\left(\frac{\alpha}{\pi}\right)^2\ln(\frac{R}{\eps^s}) + O(1),
\end{align*}
for $i$ odd. Thus we have
$$E_\eps(v_\eps^i;\omega_R(q_i))\le 2\pi\left(\frac{\alpha}{\pi}\right)^2 s |\ln({\eps})| + C,$$
for odd $i$. When $i$ is even , we modify $h^{\pm}$ to
$$\tilde h^-(\rho)=\gamma(\rho, \theta_-(\rho))-\alpha;\ \tilde h^+(\rho)=\gamma(\rho, \theta_+(\rho))+\alpha-2\pi,$$
and follow the same estimates to arrive at:
$$E_\eps(v_\eps^i;\omega_R(q_i))\le 2\pi\left(\frac{\pi-\alpha}{\pi}\right)^2 s |\ln({\eps})| + C.$$

Now, consider the ($\eps$-independent) domain,
$\tilde \Omega = \Omega \setminus \bigcup_{i=1}^{2\DDD}B_R(q_i).$
Define $\tilde g:\partial\tilde\Omega\to S^1$ by:

$$\tilde g=\left\{\begin{aligned}
g\, \text{ on } \Gamma\setminus \bigcup_{i=1}^{2\DDD}B_{{2R}}(q_i)\\
v_\eps^i \text{ on } \partial B_{2R}(q_i)\cap \Omega.
\end{aligned}
\right.
$$

By the construction of $v_\eps^i$, the function $\tilde g$ is $\eps$-independent, piecewise $C^1$, continuous, and deg$(\tilde g;\partial\tilde\Omega)=0$. 

Therefore we can find $\tilde v\in H^1_{\tilde g}(\tilde\Omega)$ with
$$E_\eps(\tilde v;\tilde \Omega)\le C.$$
Hence, setting
$$v_\eps=\left\{\begin{aligned}
&\tilde v\, \text{ in } \tilde \Omega\\
&v_\eps^i \text{ in } \omega_R(q_i),
\end{aligned}
\right.
$$
we have $v_\eps\in H^1(\Omega)$ with
$$E_\eps(v_\eps)\le 2\pi \DDD s \left( \left(\frac{\alpha}{\pi}\right)^2+\left(1-\frac{\alpha}{\pi}\right)^2\right)|\ln\eps| + C,$$
which is the desired upper bound \eqref{ub2} and this ends the proof of the Lemma.


\end{proof}

Next we prove the following pointwise upper bounds on solutions to \eqref{EL}.

\begin{lemma}\label{apriori}
Let $u_\eps$ be any solution of \eqref{EL}.  
\begin{enumerate}
\item[(i)] Suppose that $\eps \Upsilon_\eps\le C$.
Then $||u_\eps||_\infty\le 2$ and there exists  a constant $C_1=C_1(\Omega)>0$ so that $|\nabla u_\eps|\le C_1/\eps$, for all $x\in\Omega$. 
\item[(ii)] If we further assume that $\eps\Upsilon_\eps\to 0$ as $\eps\to 0$, then $\limsup_{\eps\to 0} ||u_\eps||_\infty\le 1$.
\end{enumerate}
\end{lemma}

\begin{proof}
Let $u_\eps$ solve \eqref{EL}, and  suppose by contradiction that $u_\eps$ is not uniformly bounded, and that there is a point $p_\eps\in \Omega$ such that $|u_\eps(p_\eps)|=||u_\eps||_\infty>{ 2}$. Set $V_\eps=|u_\eps|^2$ and (using $u,V$ rather than $u_\eps,V_\eps$,) we obtain
$\nabla V=2 u\cdot\nabla u$ and 
$$\frac12\Delta V =|\nabla u|^2+u\cdot \Delta u= |\nabla u|^2 + \frac{1}{\eps^2}V(V-1).$$

If $p_\eps \in \Omega$, $V$ attains an interior maximum at $p_\eps$ with $V(p_\eps)>1$ so that $\frac12\Delta V (p_\eps) >0$, which is a contradiction. 

If instead $p_\eps \in \partial \Omega$, note that 
$${\frac{\partial V}{\partial \nu}}=2 u\cdot {\frac{\partial u}{\partial \nu}}=
-2\Upsilon_\eps\Big[V(V-1) + [u\cdot g -\cos\alpha](u\cdot g)\Big]$$ 
and hence denoting by $m_\eps:=||u_\eps||_\infty$ and still assuming that it is not uniformly bounded we obtain:
\bal
{\frac{\partial V}{\partial \nu}}&\le -2\Upsilon_\eps\left[m_\eps^2(m_\eps^2-1) - \left|u\cdot g -\cos\alpha||u\cdot g\right|\right]\\
&\le -2\Upsilon_\eps\left[ m_\eps^4 -m_\eps^2- [(m_\eps+|\cos\alpha|)m_\eps]\right]\\
&\le -2\Upsilon_\eps\left[ \frac12 m_\eps^4 -2m_\eps^2\right]\\
&\le -\Upsilon_\eps\left[ m_\eps^4-4m_\eps^2\right] \\
&<0,
\end{align*}
in case $m_\eps>2$.
But if $V$ attains its maximum at $p_\eps \in \partial\Omega$, then 
${\frac{\partial V}{\partial \nu}}(p_\eps)\ge 0$, which is a contradiction.
Therefore $m_\eps=\|u_\eps\|_\infty$ is bounded by $2$.

Next we show that if $\eps\Upsilon_\eps\to 0$ as $\eps\to 0$, then $\limsup_{\eps\to 0} ||u_\eps||_\infty\le 1$. Indeed, assume for a contradiction that (along some subsequence) $m_\eps\to m_0>1$. Note that if $\|u_\eps\|_\infty$ is attained inside $\Omega$, by the above $|u_\eps(x)|\le 1,\forall x \in \bar\Omega$. So suppose that $p_\eps\in\partial\Omega$. Without loss of generality, we rotate the domain such that the normal $\nu(p_\eps) =\vec e_2$. 
Blowing up at scale $\eps$ around the points $p_\eps$, define for $y\in \Omega_\eps:=\eps[\Omega-p_\eps]$ the function
$v_\eps(y)=u_\eps(p_\eps + \eps y)$ and let $\tilde g(y)$ denote the boundary values $g$ for $y\in \partial \Omega_\eps$. As $\eps\to 0$, the set $\Omega_\eps$ becomes the upper half plane $\R^2_+$ and we have:

$$\Delta v_\eps = \eps^2\Delta u_\eps =(|v_\eps|^2-1)v_\eps,$$
$$\frac{\partial v_\eps}{\partial \nu} = \eps \frac{\partial u_\eps}{\partial \nu} 
=-\Upsilon_\eps \eps \Big[(|v_\eps|^2-1) v_\eps + [v_\eps\cdot \tilde g -\cos\alpha]\tilde g\Big]$$

    Note that $|v_\eps|\le m_\eps\le 2$  by the above, so that in $B_R^+(0)$ we have $v_\eps\to v_0$ in $C^2_{loc}$ (along a subsequence.) Therefore, using that $\eps\Upsilon_\eps\to 0$, we obtain
    
\begin{align*}
\Delta v_0&=(|v_0|^2-1)v_0,\\
\frac{\partial v_0}{\partial \nu} &=0, \text{ and }|v_0(0)|=m_0>1.
\end{align*}

As usual, we define $V_0(y):= |v_0(y)|^2$, and we obtain for $V_0$:
$$\frac12\Delta V_0 =|\nabla v_0|^2+v_0\cdot \Delta v_0\ge \frac{1}{\eps^2}V_0(V_0-1)>0,$$
while having a maximum at $y=0$ with $\frac{\partial V_0}{\partial \nu} =0$: this contradicts the Hopf Lemma. We conclude that $\limsup_{\eps\to 0} ||u_\eps||_\infty\le 1$.

To establish the gradient bound, we argue by contradiction:  suppose there exist sequences $\eps_k\to 0$, $x_k\in \overline{\Omega}$ for which
$t_k:=|\nabla u_k(x_k)| = \|\nabla u_k\|_\infty$ satisfies $t_k\eps_k\to\infty$.  Blowing up at scale $t_k$ around the points $x_k$, define
$v_k(x):= u_k\left(x_k + \frac{x}{t_k}\right)$.  By our choice of scaling, $\|v_k\|_\infty <C$, and $v_k$ solves 
$$  -\Delta v_k = \frac{1}{(t_k\eps_k)^2} (|v_k|^2-1)v_k \to 0, $$
uniformly on $\Omega$ (since $\|u_k\|_\infty=\|v_k\|_\infty<C$, by the first part of the lemma.)  If, for some subsequence, $t_k\text{dist}\,(x_k,\partial\Omega)\to \infty$, then the domain $t_k[\Omega-x_k]$ of $v_k$ converges to all $\R^2$, and $v_k\to v$ in $C^k_{loc}$.  Moreover, the limit
 $v$ is a bounded harmonic function on $\R^2$, and hence constant:  $\nabla v(x)\equiv 0$.  However, by construction, $|\nabla v_k(0)|=1$ for all $k$, and hence $|\nabla v(0)|=1$, a contradiction.
 
On the other hand, if $t_k\text{dist}\,(x_k, \partial\Omega)$ is uniformly bounded, then the domains $t_k[\Omega-x_k]$ of $v_k$ converge to a half-space $\R^2_+$, with boundary condition 
$$  \frac{\partial v_k}{\partial\nu}= -\frac{\Upsilon_{\eps_k}}{t_k}\left[v_k(|v_k|^2 -1)-\left[\left(v_k, g\left(x_k + \frac{x}{t_k}\right)\right)-\cos \alpha\right]g\left(x_k + \frac{x}{t_k}\right)\right] \to 0,  $$
since $\frac{\Upsilon_{\eps_k}}{t_k}\to 0$ and $v_k$ is uniformly bounded by the a priori bound on $u_\eps$ proven above.
That is, $v_k\to v$ which is  bounded and harmonic in $\R^2_+$, and with a Neumann condition $\partial_\nu v=0$ on the boundary.  By the reflection principle and Liouville's theorem we again conclude that $v$ is constant, which leads to the same contradiction as in the previous case.  Thus, the desired gradient bound must hold. 
 \end{proof}

\section{Isolating the defects}

We begin by proving an $\eta$-compactness (also called $\eta$-ellipticity) result (see \cite{Struwe}, \cite{Riviere}). We then define vortex balls, of radius of order $\eps$ in the interior and of radius of order $\eps^s$ on the boundary of the domain, and following Struwe (\cite{Struwe}) we show that they form a uniformly bounded family.

\subsection*{$\eta$-compactness}

 Basically, if the energy contained in a ball of radius $\eps^\beta$ is too small, there can be no vortex in a slightly smaller ball, $B_{\eps^\gamma}(x_0)$.  To this end, we recall that $\Upsilon=\Upsilon(\eps)=\eps^{-s}$ for $s\in (0,1]$,  and fix $\beta,\gamma$ such that $\frac34 s\le\beta<\gamma<s$. We also let $a\in (0,\frac12)$ to be chosen later.

\begin{proposition}[$\eta$-compactness]\label{eta}
There exist constants $\eta, C, \eps_0>0$ such that for any solution $u_\eps$ of \eqref{EL} with $\eps\in (0,\eps_0)$, if $x_0\in\overline{\Omega}$,  $a\in (0,\frac12)$ and
\be\label{eta0}
E_\eps\left(u_\eps;B_{\eps^\beta}(x_0)\right) 
   \le \eta\, |\ln\eps|,
\ee
then
\begin{gather}
\label{eta1}  |u_\eps|^2\ge 1-\sqrt2 a\quad\text{ in } \quad B_{\eps^\gamma}(x_0),  \\
\label{eta2}  W(u,g):=  {\frac12}(|u|^2-1)^2 + [(u,g)-\cos\alpha]^2 \le a^2\quad\text{ on $\Gamma\cap B_{\eps^\gamma}(x_0)$,} \\
\label{eta3}  
   \frac{1}{4\eps^2}\int_{B_{\eps^\gamma}(x_0)} \left(|u_\eps|^2-1\right)^2\, dx + 
      \frac{\Upsilon}{2}\int_{\Gamma\cap B_{\eps^\gamma}(x_0)} W(u_\eps, g) \, ds \le C\eta.
\end{gather}
\end{proposition}

We note that in case $\Gamma\cap B_{\eps^\beta}(x_0)=\emptyset$, this has been proven in Lemma~2.3 of \cite{Struwe}, and hence it suffices to consider $x_0\in \Gamma\subset\partial\Omega$ when proving Proposition~\ref{eta}.

Define $\Gamma_r(x_0)= \partial\Omega\cap B_r(x_0)$, and following Struwe \cite{Struwe}, define
\begin{multline}\label{Fdef}
F(r)=F(r; x_0, u,\eps)=\\ r\left[
  \int_{\partial B_r(x_0)\cap\Omega} \frac12\left\{ |\nabla u|^2
      + \frac{1}{2\eps^2}(|u|^2-1)^2\right\} ds  
       + \frac{\Upsilon(\eps)}{2}\sum_{x\in\partial \Gamma_r(x_0)}  W(u,g)
  \right].
\end{multline}
Note that if $\partial\Gamma_r(x_0)\neq\emptyset$, then for $r>0$ sufficiently small it consists of two points.

The proof of Proposition~\ref{eta} relies on the following estimate.  For any $x_0\in\overline{\Omega}$ and $R>0$, we define (as in the proof of Lemma~\ref{upperbound}) 
\be\label{omega} 
{\displaystyle\qquad \qquad\omega_R(x_0)=B_R(x_0)\cap\Omega . }
\ee
 Then, we first prove:
\begin{lemma}\label{poho}
There exist $C>0$ and $r_0>0$ such that for $\eps \in(0,1)$, $x_0\in\Gamma$, and $r \in (0,r_0)$, we have that
\begin{equation*}
	\frac{1}{2\eps^2} \int_{\omega_r(x_0)} \big( |u_\eps|^2-1\big)^2 \, dx \ + \, \Upsilon \int_{\Gamma_r(x_0)} W(u,g) \,ds 
	\leq C  \left\{
		r\int_{\omega_r(x_0)} |\grad u_\eps|^2 \, dx 
		+ F(r) + r^2\Upsilon
	\right\}.
\end{equation*}
where $F(r)$ is as in \ref{Fdef}.

\end{lemma}

\begin{proof}[ Proof of Lemma~\ref{poho}]
We denote $u=u_\eps$, $\omega_r=\omega_r(x_0)$, and $\Gamma_r=\Gamma_r(x_0)$ for convenience, as $x_0\in\Gamma$ and $\eps>0$ are fixed.

Let $\psi \in C^\infty(\Omega;\R^2)$ be a vector field, to be determined later.  Taking the complex scalar product of the equation \eqref{EL} with $\psi\cdot\nabla u$ and integrating over $\omega_r$, we obtain the Pohozaev-type equality,
\begin{multline}\label{pohoz}
\int_{\partial\omega_r} \left\{
-(\partial_\nu u, \psi\cdot\nabla u) + \frac12 |\nabla u|^2 (\psi\cdot \nu)
  +\frac{1}{4\eps^2} (|u|^2-1)^2(\psi\cdot \nu)
\right\} ds  \\
= \int_{\omega_r} \left\{ \frac{1}{4\eps^2} (|u|^2-1)^2 \div\psi 
             + \frac12 |\nabla u|^2\div \psi - \sum_{i,j}\partial_i\psi_j (\partial_iu,\partial_ju)
             \right\}dx.
\end{multline}

We choose $r_0>0$ sufficiently small so that $\Gamma\cap B_r(x_0)$ consists of a single smooth arc, and $\omega_r$ is strictly starshaped with respect to some $x_1\in\omega_r$, for all $0<r\le r_0$.  

Let $\mathcal{N}$ be a $2r_0$-neighborhood of $\Gamma$.  We claim that, by taking $r_0$ smaller if necessary, there exists a vector field $X\in C^2(\mathcal{N};\R^2)$ with the following properties (see \cite{Ku}, \cite{Moser}):
\begin{gather}
X\cdot\nu = 0, \quad\text{for all $x\in \Gamma_r$}, \label{X1} \\
|X-(x-x_0)|\le C |x-x_0|^2, \quad |DX - Id| \le C |x-x_0|,\quad
\text{for all $x\in \omega_r$}, \label {X2}
\end{gather}
for a constant $C>0$, for any $x_0\in\Gamma$.
The existence of such a vector field in a disk $B_r(x_0)$ follows from the smoothness of $\Gamma$; to obtain the uniform global estimates \eqref{X1}, \eqref{X2} we use the compactness of $\Gamma$ and a partition of unity.
In particular, note that $X= (X\cdot\tau)\tau\simeq (x-x_0)\tau$ lies along the tangent vector on $\Gamma_r$.  

We now take $\psi=X$ in \eqref{pohoz} and estimate each term in \eqref{pohoz}, separating the $\partial\omega_r$ terms into the pieces along $\Gamma_r$ and along $\partial B_r(x_0)\cap\Omega$.  First,
on $\Gamma_r$ we have $X\cdot\nu=0$, and the only contribution to the left hand side of \eqref{pohoz} is:
\begin{align*} \nnn
-\int_{\Gamma_r} (\partial_\nu u, X\cdot\nabla u)\, ds
  &= \Upsilon \int_{\Gamma_r} \left\{(|u|^2-1)(u,X\cdot\nabla u)+[(u\cdot g)-\cos \alpha]\,  \, (g,X\cdot\nabla u) \right\}\, ds \\
  &= \Upsilon \int_{\Gamma_r} 
     \left\{(|u|^2-1)\left(u,  \frac{\partial u}{\partial \tau}\right)+\left[ (u\cdot g)-\cos \alpha\right] \, \, \left(g, \frac{\partial u}{\partial \tau}\right)\right\} (X\cdot\tau)\, ds \\
   &=  \Upsilon \int_{\Gamma_r}\left[ \partial_\tau \left(\frac12[(u,g)-\cos\alpha]^2\right) -[(u,g)-\cos\alpha](\partial_\tau g, u)\right.\\
   &\left.\qquad\qquad+\partial_\tau(\frac14(|u|^2-1)^2))
   \right]\, X\cdot \tau\, ds\\
&=I_1-I_2 + I_3
\end{align*}
To estimate $I_1$, we  use integration by part and \eqref{X2} as follows:
\begin{align}
I_1&=\Upsilon \int_{\Gamma_r} \frac{\partial}{\partial \tau}\left\{ \frac12 [u\cdot g-\cos\alpha]^2\right\} X\cdot \tau\, ds\\
&=\frac{\Upsilon}{2}\left\{ [u\cdot g-\cos\alpha]^2 (X\cdot \tau)\Bigg|_{\partial \Gamma_r}-\int_{\Gamma_r} \left\{  [(u\cdot g)-\cos\alpha]^2\right\} \partial_\tau (X\cdot \tau)\, ds\right\}\\
&=\frac{\Upsilon}{2}\left\{ r \sum_{\partial \Gamma_r} [(u\cdot g)-\cos\alpha]^2 -\int_{\Gamma_r} \left\{  [(u\cdot g)-\cos\alpha]^2\right\}\, dx \right\} + O(\Upsilon r^2),
\label{p1}
\end{align}
using \eqref{X2} in the last line. Indeed, on the endpoints of $\Gamma_r$, $|X\cdot\tau \mp r|\le Cr^2$ and on $\Gamma_r$ itself, $\partial_\tau(X\cdot\tau)= 1 + O(|x-x_0|)$.  

For $I_2$, we have the rough estimate:

$$|I_2|\le \Upsilon |\Gamma_r|\left( (\|u\|_\infty + 1)^2\left|\left|\frac{\partial g}{\partial \tau}\right|\right|_\infty \|X\cdot \tau\|_\infty\right)\le C\Upsilon r^2,$$
using again \eqref{X2}.
Finally, $I_3$ is estimated in the same way as $I_1$:

\begin{align}
I_3&=\Upsilon \int_{\Gamma_r}\frac14\partial_\tau((|u|^2-1)^2)
   \, X\cdot \tau\, ds\\
&=\Upsilon\left\{\frac14(|u|^2-1)^2(X\cdot\tau)\Bigg|_{\partial \Gamma_r}-\frac14\int_{\Gamma_r} \left(|u|^2-1\right)^2 \partial_\tau (X\cdot \tau)\, ds\right\}\\
&=\frac{\Upsilon}{2}\left\{ r\sum_{\partial\Gamma_r}\frac12\left(|u|^2-1\right)^2 -
\frac12\int_{\Gamma_r} \left(|u|^2-1\right)^2\, ds\right\} + O(\Upsilon r^2)
\end{align}

The remaining terms on the left-hand side of \eqref{pohoz}  may also be estimated in a simple way, using $|X\cdot\nu|, |X\cdot\tau|\le Cr$ and \eqref{X1}:

\be\label{p2}
\left| \int_{\partial\omega_r\cap\Omega}\left\{-(\partial_\nu u, X\cdot \nabla u ) + \frac12 |\nabla u|^2 X\cdot \nu \right\}\, ds \right|\le C r \int_{\partial\omega_r\cap\Omega} |\nabla u|^2\, ds
\ee

\begin{align}
\left|\frac{1}{4\eps^2}\int_{\partial\omega_r} (|u|^2-1)^2\left(X\cdot\nu\right) \, ds\right|
 &= \left|\frac{1}{4\eps^2}\int_{\partial B_r\cap\Omega} 
       (|u|^2-1)^2\left(X\cdot\nu\right) \, ds\right|\\
     &  \le \frac{Cr}{\eps^2}\int_{\partial B_r\cap\Omega} 
       (|u|^2-1)^2\, ds.\label{p4}
\end{align}
For the terms on the right side of \eqref{pohoz}, we use \eqref{X2}:  $|\partial_i X_j - \delta_{ij}|\le Cr$, and for $r_0$ chosen smaller if necessary, we may assume $\div X\ge 2-Cr>1$ in $\omega_r$.  Thus, the right side of \eqref{pohoz} may be estimated as:
\begin{multline}\label{p6}
\int_{\omega_r} \left\{ \frac{1}{4\eps^2} (|u|^2-1)^2 \div X 
             + \frac12 |\nabla u|^2\div X- \sum_{i,j}\partial_i X_j (\partial_iu,\partial_ju)
             \right\}dx \\
\ge 
  \int_{\omega_r} \left\{ \frac{1}{4\eps^2} (|u|^2-1)^2
     - C r |\nabla u|^2 \right\} dx.
\end{multline}
Putting the above estimates together, we arrive at the desired bound.
\end{proof}

We are ready for the

\begin{proof}[ Proof of Proposition~\ref{eta}]
We follow \cite{Struwe}, \cite{Moser}.  If $x_0\in\Omega\setminus\Gamma$, this is proven in \cite{Struwe}, so we restrict our attention to $x_0\in\Gamma$.

Recalling the definition of $F$ \eqref{Fdef}, we note that since
\be\label{Fest}  \eta\ln\frac{1}{\eps}\ge E_\eps(u_\eps; \omega_{\eps^\beta}\setminus\omega_{\eps^\gamma}) = \int_{\eps^\gamma}^{\eps^\beta}
 \frac{F(r)}{r} \, dr,
\ee
there exists $r_\eps\in (\eps^\gamma,\eps^\beta)$ so that
$$  F(r_\eps) \le \frac{\eta}{\gamma-\beta}.  $$
By Lemma~\ref{poho} with $\Gamma_{r_\eps}:=\partial \omega_{r_\eps}\cap \Gamma$,   we deduce :
\begin{align}\nnn
\frac{1}{2\eps^2} \int_{\omega_{r_\eps}(x_0)} (|u|^2-1)^2\, dx &+ \Upsilon_\eps
\int_{\Gamma_{r_\eps}(x_0)} W(u_\eps, g)\, ds\\ \nnn
&\le C\left\{ r_\eps \int_{\omega_{r_\eps}(x_0)}  |\nabla u_\eps|^2\, dx + F(r_\eps) + r_\eps^2\Upsilon_\eps\right\}\\ \nnn
&\le C \left\{ \eps^\beta \eta |\ln \eps| + \frac{\eta}{\gamma-\beta}+\eps^{2\beta} \eps^{-s}\right\}\\ \label{XXX}
&\le C'\eta,
\end{align}
which proves \eqref{eta3}  since $r_\eps>\eps^\gamma$. Note that no conditions are required on $\eta$ at this point, and this will prove useful later on (see Corollary\ref{easycor}.)

To prove \eqref{eta1}, assume (for contradiction) that 
there is a point $x_2\in B_{\eps^\gamma}(x_0)$ with $|u_\eps(x_2)|<1-a$.  
By Lemma~\ref{apriori}, $|\nabla u_\eps|\le C_1/\eps$, and hence there is a constant $C>0$ such that $|u(x)|<1-\frac{a}2$, $\forall x\in B_{C\eps}(x_2)\subset B_{\eps^\gamma}(x_0)$.
In that case,
\begin{align*}
\frac1{4\eps^2}\int_{B_{\eps^\gamma}(x_0)}(|u|^2-1)^2\, dx&\ge 
\frac1{4\eps^2}\int_{B_{C\eps}(x_0)}(|u|^2-1)^2\, dx\\
&\ge C a^2.
\end{align*}

We then choose $\eta>0$ small enough  so this contradicts \eqref{eta3} and hence \eqref{eta1} holds for all such $\eta$ (which is independent of $x_0$).


To verify \eqref{eta2}, we return to the Pohozaev identity \eqref{pohoz}.
We recall that for $r=r_\eps$ (as in the proof of \eqref{eta3}) sufficiently small, the smoothness and compactness of $\Gamma$ ensure that $\omega_{r_\eps}=B_{r_\eps}(x_0)\cap \Omega$ is strictly starshaped around some $x_1\in\omega_{r_\eps}$, and we have $(x-x_1)\cdot\nu\ge r_\eps/4$ on $\partial\omega_{r_\eps}$.
Taking $\psi= x-x_1$ in  \eqref{pohoz}, we obtain:
\begin{align}\nnn
&\int_{\partial\omega_{r_\eps}} \left\{
(x-x_1)\cdot\nu\left[ |\partial_\tau u_\eps|^2 - |\partial_\nu u_\eps|^2\right] + (x-x_1)\cdot(\nu-\tau)\, (\partial_\nu u_\eps , \partial_\tau u_\eps)\right\} ds   \\
\label{p7}
&\qquad\le
\frac{1}{2\eps^2}\int_{\omega_{r_\eps}} (1-|u_\eps|^2)^2 dx.
\end{align}
To estimate the second term in the left hand side of this inequality we use Cauchy-Schwartz,
$$  \left| \int_{\partial\omega_{r_\eps}}  (x-x_1)\cdot(\nu-\tau)\, (\partial_\nu u_\eps , \partial_\tau u_\eps)
\right|
\le  2r_\eps \int_{\partial\omega_{r_\eps}} \left\{ \frac{1}{32} |\partial_\tau u_\eps|^2
          + 8 |\partial_\nu u_\eps|^2 \right\} ds,
$$
and hence
$$
\frac{r_\eps}{16}\int_{\partial\omega_{r_\eps}} |\partial_\tau u_\eps|^2 \, ds
\le C r_\eps \int_{\partial\omega_{r_\eps}} |\partial_\nu u_\eps|^2 \, ds + 
\frac{1}{2\eps^2}\int_{\omega_{r_\eps}} (1-|u_\eps|^2)^2 dx
$$
which yields:
\begin{align*}  \int_{\partial\omega_{r_\eps}} |\partial_\tau u_\eps|^2 \, ds
&\le C'\int_{\partial\omega_{r_\eps}} |\partial_\nu u_\eps|^2\, ds + 
\frac{8}{r_\eps\eps^2}\int_{\omega_{r_\eps}} (1-|u_\eps|^2)^2 \, dx\\
  &\le C'\Upsilon^2\int_{\Gamma_{r_\eps}} \Big|[(u_\eps,g)-\cos\alpha]g + (|u|^2-1)u\Big|^2\, ds 
     + \frac{8}{r_\eps\eps^2}\int_{\omega_{r_\eps}} (1-|u_\eps|^2)^2 \, dx+\frac{C''}{r_\eps} F(r_\eps)\\
     &\le C'''\left[\Upsilon^2\int_{\Gamma_{r_\eps}}W(u,g)\, ds + 
     \frac{8}{r_\eps\eps^2}\int_{\omega_{r_\eps}} (1-|u_\eps|^2)^2 \, dx+\frac{F(r_\eps)}{r_\eps}\right]\\
     &\le C''''\eta\left\{ \eps^{-s} +  \eps^{-\gamma}\right\}\\
     &\le C \eps^{-s},
\end{align*}
 using \eqref{eta3} in the next to last line.  By the Sobolev embedding theorem (on the one-dimensional set $\Gamma_{r_\eps}$,) there exists a constant $C>0$ independent of $x_0$ and of $\eps$ for which
\beq\label{Hol}  |u_\eps(x)-u_\eps(y)| \le C\sqrt{|x-y|}\eps^{-s/2} 
\eeq
holds for all $x,y\in\Gamma_{r_\eps}$.

The conclusion now follows as in Proposition~3.6 of \cite{Ku}.  Assume there exists $x_2\in\Gamma_{r_\eps}$ for which $W(u,g)>a$.  Using \eqref{Hol}, there would exist a radius $\rho=c\eps^{s}$, for constant $c>0$ independent of $x_0$, for which
$W(u,g)>\frac{a}2$ when $x\in \Gamma_{r_\eps}\cap B_{\rho}(x_2)$ .  In that case, we would have, by \eqref{eta3},
$$  C\eta\ge \Upsilon\int_{\Gamma_{r_\eps}\cap B_{c\eps^{s}}} W(u,g)\, ds
> \Upsilon\frac{a}{2} 2\pi c \eps^{s}=\pi a c', $$
which would lead to a contradiction for $\eta$ chosen sufficiently small.  By reducing the value of $\eta$ required for the proof of \eqref{eta1} if necessary, we obtain \eqref{eta2}.  Thus there exists $\eta>0$ for which all three statements are valid and this completes the proof of Proposition\ref{eta}.
\end{proof}

\begin{corollary}\label{easycor}
Let $(u_\eps)_{\eps>0}$ be a family of solutions with $E_\eps(u_\eps)\le K |\ln \eps|$ and $\frac34s<\gamma<s$. Then, $\forall x_0\in\bar\Omega$,
$$\frac{1}{2\eps^2}\int_{B_{\eps^\gamma}(x_0)} \left(|u_\eps|^2-1\right)^2\, dx + 
      {\Upsilon}\int_{\Gamma\cap B_{\eps^\gamma}(x_0)} W(u_\eps, g) \, ds \le C(K)
      $$
\end{corollary}

\begin{proof} As mentioned in the proof of Proposition\ref{eta}, \eqref{XXX} holds for any $\eta$ as long as 
$E(u_\eps; B_{\eps^\beta}(x_0))\le \eta|\ln\eps|$, which is clearly satisfied with $\eta=K$. Since $B_{\eps^\beta}(x_0)\subset B_{r_\eps}(x_0)$, the conclusion follows.
\end{proof}

\subsection*{Defining bad balls}

We define the family of sets
$$  S_{\eps} =\left\{ x\in\Omega: \  |u_\eps(x)|^2 <1-\sqrt2 a\ \right\},\quad
  T_{\eps} =\left\{ x\in\partial\Omega: \ W(u_\eps, g)>a^2\right\}.
$$

Following Lemmas~3.1 and 3.2 of \cite{Struwe}, we show that the sets $S_{\eps}, T_{\eps}$ which include the defects may be contained in a bounded number of vary small balls.
\begin{lemma}\label{sballs}
There exists $N_0=N_0(a,g,s,\alpha)$, $\kappa>1$, and points $p_{\eps,1},\dots,p_{\eps,I_\eps}\in S_\eps$, $y_{\eps,1},\dots,y_{\eps,J_\eps}\in T_\eps\cap\Gamma$ such that 
\begin{enumerate}
\item[(i)]  $I_\eps + J_\eps\le N_0$;
\item[(ii)] 
$\{ B_{\kappa\eps}(p_{\eps,i}), B_{\kappa\eps^s}(y_{\eps,j})\}_{1\le i\le I_\eps, 1\le j\le J_\eps}$ are mutually disjoint; more precisely

$|p_{\eps,i}-p_{\eps,j}|>8\kappa \eps$, $|p_{\eps,i}-y_{\eps,j}|>8\kappa \eps^s$ and $|y_{\eps,i}-y_{\eps,j}|>8\kappa \eps^s$
\item[(iii)]
\be\label{5cover}   S_\eps \subset \bigcup_{i=1}^{I_\eps}B_{\kappa\eps}(p_{\eps,i})\text{ and } 
   T_\eps \subset \bigcup_{j=1}^{J_\eps}B_{\kappa\eps^s}(y_{\eps,j}).  
\ee
\end{enumerate}
\end{lemma}

\begin{proof}[Proof]
The is essentially the same as in Struwe \cite{Struwe}, who considered the case of Dirichlet boundary conditions, for which all of the ``bad balls'' have the same radius $\eps$, so here we need to make some modification due to our boundary conditions.  As the existence of $\eps$-balls covering $S_\eps$ is the same as in \cite{Struwe}, we only need to treat $T_\eps$. 

By the $\eta$-compactness Proposition \ref{eta},
if $y\in T_\eps$, it follows that $E_\eps(u_\eps;B_{\eps^\beta}(y))>\eta|\ln\eps|$. Furthermore,
 applying the Vitali's covering Theorem to the collection $(B_{\eps^\beta}(y))_{y\in T_\eps}$, there is a finite choice $y_1,\dots,y_{N_1}\in\overline{T_\eps}$ for which $(B_{\eps^\beta}(y_i))_{i=1,\dots,N_1}$ are disjoint, and 
 $\left[\Omega\cap\bigcup_{y\in T_\eps} B_{\eps^\beta}(y)\right]\subset
 \left(\bigcup_{i=1}^{N_1} B_{5\eps^\beta}(y_i)\right)$. Therefore it follows:
 \begin{align}
 N_1\eta|\ln\eps|&\le\sum_{i=1}^{N_1} E_\eps(u_\eps; B_{\eps^\beta}(y_i))
   \le E_\eps(u_\eps) 
   \le K|\ln\eps|,
 \end{align}
 which means that $N_1=N_1(\eps)$ is uniformly bounded from above.

Next, using the same argument as in \eqref{Fest}, there exists $r_\eps\in (\eps^\gamma,\eps^\beta)$ such that 
$$(\gamma-\beta) |\ln\eps| F(r_\eps)\le E(u_\eps;\omega_{\eps^\beta\setminus\eps^\gamma}(y))\le K|\ln\eps|, \text{ i.e. } 
F(r_\eps)\le C_1$$
$\forall \eps, y\in T_\eps$, so by Lemma~\ref{poho} we obtain the uniform estimate
\begin{align*}  \frac{1}{2\eps^2}\int_{\omega_{r_\eps}(y)} (|u_\eps|^2-1)^2\, dx
         + \Upsilon\int_{\Gamma_{r_\eps}(y)} W(u,g)\, ds
         &\le C(F(r_\eps) + r_\eps^2\Upsilon + O(1))\\
            & \le C_2,
\end{align*}
uniformly in $\eps, y\in T_\eps$.

On the other hand, by the H\"older bound arguments employed in the proof of Proposition~\ref{eta} (see \eqref{Hol},) there exists  constants $c_1, c_2$ (independent of $\eps$) such that $\forall y\in T_\eps$, 

$$  \Upsilon_\eps\int_{B_{c_1\eps^s}(y)} W(u,g)\, dx
   \ge c_2>0,
$$ 
independently of $\eps, y\in T_\eps$.

Now following Struwe (see Lemma 3.2 in \cite{Struwe}) and using Vitali's covering Theorem again, we conclude that there exists a finite collection $y_1, \dots, y_J\in T_\eps$, with $J$ uniformly bounded in $\eps$ such that the sets { $\left\{ B_{ {\eps^s}}(y_j)\right\}_{j=1,\dots,J}$ are disjoints and $T_\eps\subset \bigcup_{j=1}^J B_{5\eps^s}(y_j)$}.
Finally, by the same argument as that of Theorem IV.1 of \cite{BBH2}, by enlarging and if necessary fusing together vortex balls which intersect, we may find $\kappa>1$ and modified centers $p_i, y_j$ for which (ii) holds.

\end{proof}

\section{Classifying the defects}

Our goal in this section is to classify defects $x_0\in\overline{\Om}$, defined as a center of one of the ``bad balls'' constructed in Lemma~\eqref{sballs}, and associate a degree to each.  For any $x_0\in \overline{\Omega}$, and $0<r<R<\infty$, denote the annulus centered at $x_0$ by
$$  A_{r,R}(x_0)=w_R(x_0)\setminus w_r(x_0).   $$
We will analyze the structure of $u_\eps$ in such annuli around defects.

We begin with a lemma which shows the energy densities are coercive near their minima:

\begin{lemma}\label{W-lemma}
For any $\alpha\in (0,\frac{\pi}2),$ we can find a constant $C_{\alpha}>0$ and $a_0=a_0(\alpha)$ such that for $u\in \C$ and $g=e^{i\gamma}\in S^1\subset\C$ with $W(u,g)<a_0^2$ we may represent  $u=fe^{i\psi}$ with

\beq
\begin{gathered}\label{a0}
|f^2-1|\le \sqrt{2W(u,g)}<\sqrt{2}a_0\text{ and }\\
\text{ either } |\psi-\gamma-\alpha|<C_{\alpha} \sqrt{W(u,g)}
\text{ or } |\psi-\gamma+\alpha|<C_{\alpha} \sqrt{W(u,g)}
\end{gathered}
\eeq
\end{lemma}

Note that by choosing $a_0$ sufficiently small, the intervals 
\be\label{intervals} \Ipm_\pm:= \{\psi \ : \ |\psi-\gamma\pm\alpha|<C_{\alpha} \sqrt{W(u,g)} \}
\ee
 will be disjoint. In particular, in places where $W$ is small, we know that $u$ must be close to either $e^{i(\gamma\pm \alpha)}$, but not both.

\begin{proof}[Proof of the Lemma~\ref{W-lemma}]
Let $a:=\sqrt{W(u,g)}<a_0$, $a_0$ to be determined.
$$\frac12(|u|^2-1)^2\le a^2\iff 1-\sqrt2 a\le |u|^2\le 1+\sqrt2 a
$$
It follows that for $a_0<\frac14$, $a\in (0,a_0)$, we may write $u=fe^{i\psi}$ with
$$1-\sqrt2 a\le f^2\le 1+\sqrt2 a\text{ and } \phi:=\psi-\gamma\in [-\pi,\pi].$$
This choice of $\phi$ is natural since we have in addition,
$$[(u,g)-\cos\alpha]^2\le a^2\iff \cos\alpha -a\le f\cos\phi\le\cos\alpha + a.$$
Therefore
$$\frac{\cos\alpha-a}{\sqrt{1+\sqrt2 a}}\le \cos\phi\le \frac{\cos\alpha+a}{\sqrt{1-\sqrt2 a}},$$
 and it follows:
\begin{align}
\cos\phi&\le (\cos\alpha + a )(1+4a)< \cos\alpha + C_1 a\\
\cos\phi&\ge(\cos\alpha - a )(1-a)> \cos\alpha - C_1 a,
\end{align}
and hence
$$|\cos\phi-\cos\alpha|<C_1 a.$$

Since $\alpha\in (0,\frac{\pi}2)$, and $\phi\in(-\pi,\pi)$, choosing $a_0=a_0(\alpha)$ sufficiently small, $\forall a\in [0,a_0)$, we have
$$\{\phi\in(-\pi,\pi):|\cos\phi-\cos\alpha|<C_1 a\}\subset (-\alpha-C_\alpha a, -\alpha+C_\alpha a)\cup (\alpha-C_\alpha a, \alpha+C_\alpha a).$$
Recalling that $a=\sqrt{W(u,g)}$,  this completes the proof.

\end{proof}

For the remainder of the paper, we fix once and for all a value $a_0=a_0(\alpha) $ such that the intervals $\Ipm$ in \eqref{intervals} where $\psi-\gamma$ is close to either $\alpha$ or $-\alpha$ are disjoint.  

\medskip

We now treat the question of classification of defects, and their associated degrees.  If $x_0\in\Om$ then the defect is an interior vortex, and its degree $d(x_0)\in\ZZ$ is defined in the usual way.  When $x_0\in\Gamma=\partial\Om$ the situation is more interesting and more subtle.

In case $x_0\in\Gamma$, for $R$ sufficiently small the piece of the boundary $\partial A_{r,R}(x_0)\cap\partial\Omega$ consists as before of exactly two arcs along $\Gamma_R=\Gamma\cap B_R(x_0)$, which we will denote by $\Gamma_{r,R}^\pm$.  (See Figure~\ref{fig:Ar}.)  We recall from the upper bound construction in Lemma~\ref{upperbound} that $\Gamma_R^\pm$ may be parametrized as
$$\Gamma_R^{\pm}=\{(\rho,\theta_{\pm}(\rho)):\ 0<\rho\le R\},  $$
for smooth $\theta_\pm(\rho)$, with $\theta_+(\rho)=O(\rho)$ and $\theta_-(\rho)=\pi+O(\rho)$.

We now apply Proposition~\ref{eta} to $u_\eps$ to conclude that for any $0<r<R$ with $A_{r,R}(x_0)$ disjoint from the bad balls covering $S_\eps\cap T_\eps$, we have $|u_\eps|^2\ge 1-\sqrt{2}a$ in $A_{r,R}(x_0)$ and, for $x_0\in\Gamma$, $W(u_\eps,g)\le a^2$ on $\Gamma_R^\pm$.  In particular, Lemma~\ref{W-lemma} applies and we obtain the representation 
\be\label{polar}   u_\eps= f(\rho,\theta) e^{i\psi(\rho,\theta)}, \quad
    \text{with} \quad |f^2-1|<\sqrt 2 a \text{ in }A_{r,R},
\ee
and the phase $\psi$ on $\Gamma_{r,R}^{\pm}(x_0)$
is chosen with  either $\psi\in\Ipm_-$ or $\psi\in\Ipm_+$, that is,
\begin{enumerate}
\item[(I)]$|\psi-\gamma-\alpha|<C_\alpha\sqrt{W(u,g)}<C_\alpha a \text{ or }$\\
\item[(II)] $|\psi-\gamma+\alpha| < C_\alpha\sqrt{W(u,g)}<C_\alpha a,$
\end{enumerate}
By the continuity of $g=e^{i\gamma}$, for $R$ small we can treat $\gamma(x)\simeq \gamma_0:=\gamma(x_0)$ on $\Gamma_R$, and in fact the complex phase difference of $u$ along each of $\Gamma_R^\pm$  is also small (on the order of $R$)  and hence the winding of the phase around a boundary vortex occurs principally around the half-circle $\partial B_R(x_0)\cap\Omega$. Introduce polar coordinates $(\rho,\theta)$ centered at $x_0$, with $\theta$ measured from the unit tangent $\tau$ to $\Gamma$ at $x_0$.  (See Figure~\ref{fig:Ar}.)

We distinguish three possibilities for each boundary defect $x_0\in\Gamma$, define the degree, and introduce a new topological index $\tau(x_0)\in \{-1,0,1\}$, the ``boojum number".


\medskip

\noindent {\sc Classification of boundary defects:}
\begin{enumerate}

\item[(i)] {\sl ``Light boojums".}

In this case, (I) holds on $\Gamma^-_{r,R}(x_0)$ while (II) holds on $\Gamma^+_{r,R}(x_0)$. This means that the phase decreases by $2\alpha$ (modulo $2\pi)$ along $\Gamma_R$. 
So let $n(x_0)\in\ZZ$ be the number of multiple of $2\pi$ by which the phase increases around $x_0$. Note that $n(x_0)$ represents the degree at $x_0$.
In particular, we may write 
$u(\rho,\theta)=f(\rho,\theta)\exp(i\psi(\rho,\theta))$ in polar coordinates centered at $x_0$,
 with phase
\be\label{phase_light}
\psi(\rho,\theta)=\gamma_0-\alpha+2\frac{\theta}{\pi}(\alpha + n(x_0)\pi )+\phi(\rho,\theta),  \quad \theta_+(\rho)<\theta<\theta_-(\rho),
\ee
with $\phi$ a smooth single-valued function in $A_{r,R}(x_0)$.  Note that for light boojums, $(iu_\eps,g)=|u_\eps|\sin(\psi-\gamma)$ changes sign, from positive to negative,  moving counter-clockwise across the boundary defect $x_0$.   We define the boojum number $\tau(x_0)=-1$ for a light boojum.

\item[(ii)]{\sl ``Heavy boojums".}

In this case, (II) holds on $\Gamma^-_{r,R}(x_0)$ while (I) holds on $\Gamma^+_{r,R}(x_0)$. This means that the phase increases by $2\alpha$ (modulo $2\pi)$ along $\Gamma_R$.
Again, let $n(x_0)\in\ZZ$ be the number of multiple of $2\pi$ by which the phase increases.  As above, $n(x_0)$ represents the degree of a heavy boojum, and using polar coordinates centered at $x_0\in\Gamma$, we may write 
$u(\rho,\theta)=f(\rho,\theta)\exp(i\psi(\rho,\theta)),$
 with phase
\be\label{phase_heavy}
\psi(\rho,\theta)=\gamma_0+\alpha+2\frac{\theta}{\pi}(-\alpha+n(x_0)\pi )+\phi(\rho,\theta), \quad \theta_+(\rho)<\theta<\theta_-(\rho),
\ee
with $\phi$ a smooth single-valued function in $A_{r,R}(x_0)$.  As for light boojums, $(iu_\eps,g)=|u_\eps|\sin(\psi-\gamma)$ changes sign across the defect $x_0$, but for heavy boojums it goes from negative to positive as we move counter clockwise.  The boojum number for a heavy boojum is $\tau(x_0)=+1$.

\item[(iii)] {\sl ``Boundary vortices".}

This occurs when either (I) or (II) holds on both $\Gamma_{r,R}^{\pm}(x_0)$.  In particular, the phase $\psi$ rotates by $2\pi n$ along $\partial B_\rho(x_0)\cap\Om$, with $n\in\ZZ$, so in polar coordinates we may write
\be\label{phase_vtx}
\psi(\rho,\theta)=\gamma_0\pm \alpha + 2n\theta + \phi(\rho, \theta),
\quad \theta_+(\rho)<\theta<\theta_-(\rho),
\ee
for smooth, single-valued $\phi(\rho, \theta)$ in $A_{r,R}(x_0)$.  The degree associated with the boundary vortex is $n=n(x_0)\in\ZZ$, and the boojum number $\tau(x_0)=0$.  Note that the sign of $(iu_\eps,g)=|u_\eps|\sin(\psi-\gamma)$ does not change across a boundary vortex.
\end{enumerate}

\begin{remark}\label{degree_rem} \rm If we extend the modulus and phase $f_\eps,\psi_\eps$ to all of $\Gamma_R$ by linearly interpolating in $\Gamma_r$ from the values in $\Gamma_R^\pm$, we may define a nonvanishing extension $\tilde u_\eps=\tilde f_\eps e^{i\tilde\psi_\eps}$ of $u_\eps$ to all of $\Gamma_R$.  Setting $\tilde u_\eps=u_\eps$ on $\partial B_R(x_0)\cap\Omega$, we obtain an $S^1$-valued map ${\tilde u\over |\tilde u|}: \ \partial\omega_R(x_0)\to  S^1$, whose degree measured on $\partial\omega_R(x_0)$ is $n(x_0)$, as defined above.
\end{remark}

\medskip

\begin{remark}
It is here that we see that the case $\alpha\in (\frac{\pi}2, \pi)$ is the same as the case $\alpha\in (0,\frac{\pi}2)$: it only exchanges the role of ``heavy" and ``light" boojums.
\end{remark}

\medskip

Now that we have defined degrees corresponding to the bad balls constructed in Lemma~\ref{sballs}, we may verify that they must always sum to the degree $\mathcal{D}=\deg(g;\partial\Om)$:

\begin{lemma}\label{degree_lemma}
Let $\{ B_{\kappa\eps}(p_{\eps,i})\}_{1\le i\le I_\eps}$, $\{B_{\kappa\eps^s}(y_{\eps,j})\}_{1\le j\le J_\eps}$ be as in Lemma~\ref{sballs}, and $d_i=n(p_{\eps,i})$, $n_j=n(y_{\eps,j})\in\ZZ$ be their degrees.  Then,
$$  \deg(g;\partial\Om)=\mathcal{D}=\sum_{i=1}^{I_\eps} d_i + \sum_{j=1}^{J_\eps} n_j.  $$
\end{lemma}

The proof of Lemma~\ref{degree_lemma} involves excising half-disks around the boundary defects, and redefining $u_\eps$ on the arcs $\Gamma_R(x_0)$ as in Remark~\ref{degree_rem}.  The details may be found in part (i) of \cite[Lemma 5.3]{ABGS}.


\section{Energy lower bound for defects}

We are ready to prove lower bounds on the energy in annular regions around the defects.

\begin{proposition}\label{elb}  Suppose $E_\eps(u_\eps)\le K|\ln\eps|$ with constant $K$ independent of $\eps$.
Assume $x_0\in\Gamma=\partial\Omega$, $R>r>\eps^s$, and that $|u|^2\ge 1-\sqrt2 a$ on the annulus $A_{r,R}(x_0)$ and $W(u,g)\le a^2$ on $\Gamma_{r,R}^\pm$. 
 If $x_0$ has degree $n(x_0)$ and boojum number $\tau(x_0)\in\{-1,0,1\}$.  Then, there exists a constant $C$ (depending on $\alpha, a$ and $\partial\Om$,) such that:
\be\label{boojum_energy}
\frac12\int_{A_{r,R}}|\nabla u_\eps|^2\, dx
  \ge 2\pi\left(n(x_0)-\tau(x_0){\alpha\over\pi}\right)^2\ln(\frac{R}{r}) -C.
\ee
\end{proposition}

It is well-known that for interior vortices $x_0\in\Omega$, the energy lower bound is given by
$$\frac12\int_{A_{r,R}}|\nabla u_\eps|^2\, dx
  \ge \pi \left(n(x_0)\right)^2\ln(\frac{R}{r}) -C. $$

\medskip

\begin{proof}[Proof of Proposition~\ref{elb}]
For simplicity, we drop the $\eps$ subscripts, and write $n:=n(x_0)$ and $\tau=\tau(x_0)$.
We may unify the polar coordinate representations \eqref{phase_light}, \eqref{phase_heavy}, and \eqref{phase_vtx} using boojum number, and write $u(\rho,\theta)=f(\rho,\theta)e^{i\psi(\rho,\theta)}$ in $A_{r,R}(x_0)$ with
\be\label{univ_phase}  \psi(\rho,\theta)=\gamma_0+ 2n\theta + \tau\alpha\left(1-{2\theta\over\pi}\right)
+\phi(\rho,\theta). \quad \theta_+(\rho)<\theta<\theta_-(\rho),
\ee
Thus, we have:
\begin{align}\nnn
|\nabla u_\eps|^2&\ge f^2|\nabla\psi|^2
\ge f^2\left|2\left(n-\tau\frac{\alpha}{\pi}\right)\nabla\theta+\nabla\phi\right|^2\\ \nnn
&=f^2\left[4\left(n-\tau\frac{\alpha}{\pi}\right)^2 |\nabla\theta|^2 +
4\left(n-\tau\frac{\alpha}{\pi}\right)\nabla\theta\cdot\nabla \phi+|\nabla\phi|^2\right]\\ \label{grad}
&=4\left(n-\tau\frac{\alpha}{\pi}\right)^2\frac1{\rho^2}+ \mathcal{E},
\end{align}
with remainder term,
$$  \mathcal{E}:= (f^2-1)4\left(n-\tau\frac{\alpha}{\pi}\right)^2\frac1{\rho^2}+ 4f^2\left(n-\tau\frac{\alpha}{\pi}\right)\frac1{\rho^2}\frac{\partial\phi}{\partial\theta} + f^2|\nabla\phi|^2.
$$

We claim that $\displaystyle \int_{A_{r,R}(x_0)} \mathcal{E}\, dx \ge C$, with constant $C$ independent of $\eps$ as long as $r>\eps^s$.  Assuming the claim for the moment, we obtain the desired lower bound, since then,
\begin{align*}\nnn
\frac12\int_{A_{r,R}}|\nabla u_\eps|^2\, dx&\ge
2\left(n-\tau\frac{\alpha}{\pi}\right)^2\int_{A_{r,R}}\frac1{\rho^2}\, dx + C\\ \nnn
&=2\left(n-\tau\frac{\alpha}{\pi}\right)^2\int_r^R\int_{\theta^+(\rho)}^{\theta^-(\rho)}\frac1{\rho}\, d\theta d\rho + C\\ \nnn
&=2\left(n-\tau\frac{\alpha}{\pi}\right)^2\int_r^R\frac1{\rho}[\theta^-(\rho)-\theta^+(\rho)]\, d\rho + C\\ \nnn
&=2\pi\left(n-\tau\frac{\alpha}{\pi}\right)^2\int_r^R\frac1{\rho}\, d\rho +C\\ 
&=2\pi\left(n-\tau\frac{\alpha}{\pi}\right)^2\ln(\frac{R}{r}) + C.
\end{align*}
It remains to verify the claim. We will start by showing that the first term in $\mathcal{E}$ has small integral. Using the upper bound on the energy from Lemma~\ref{upperbound}, and recalling  (see Section 4) $r>\eps^s$ with $\frac34 s<\gamma<s$, we have:
\begin{align}
\left|\int_{A_{r,R}}(1-f^2)\frac1{\rho^2}\, dx\right|&\le 
\left[   (\int_{B_R} (1-f^2)^2\, dx)(\int_{A_{r,R}}\frac1{\rho^4}\, dx) \right]^{\frac12}\\
&\le \left[C\eps^2|\ln\eps| (\frac1{r^2}-\frac1{R^2})\right]^{\frac12}\to 0
\end{align}
as $\eps\to 0$.

Next we show that we can bound the second term in $\mathcal{E}$ by the (positive) third term.  We write,
\beq\label{todo}\int_{A_{r,R}}f^2\frac1{\rho^2}\frac{\partial\phi}{\partial\theta}\, dx =
\int_{A_{r,R}}\frac{\partial\phi}{\partial\theta}\frac1{\rho^2}\, dx +
\int_{A_{r,R}}\frac{f^2-1}{\rho^2}\frac{\partial\phi}{\partial\theta}\, dx
\eeq
and estimate each term separately. For the first term on the right hand side,
we write:
\begin{align*}
\int_{A_{r,R}}\frac{\partial\phi}{\partial\theta}\frac1{\rho^2}\, dx&=
\int_r^R\int_{\theta^-(\rho)}^{\theta^+(\rho)}\frac{\partial\phi}{\partial\theta}\frac1{\rho}\, d\theta\ d\rho\\
&=\int_r^R[\phi(\rho,\theta^-(\rho))-\phi(\rho,\theta^+(\rho))]\, \frac{d\rho}{\rho}
\end{align*}
and therefore
$$\left|\int_{A_{r,R}}\frac{\partial\phi}{\partial\theta}\frac1{\rho^2}\, dx\right|\le C
\left(\int_{\Gamma_{r,R}^{+}(x_0)}|\phi|\frac{d\rho}{\rho}+
\int_{\Gamma_{r,R}^{-}(x_0)}|\phi|\frac{d\rho}{\rho}\right).$$

To continue we require the estimates in Lemma~\ref{W-lemma}.  Note that the intervals $\Ipm_\pm$ for the phase $\psi$ may be defined modulo $2\pi$, and the fact that $n=n(x_0)$ is the degree of the defect implies that $\psi(x)-2\pi n\in \Ipm_-$ or $\Ipm_+$ for all $x\in \Gamma^\pm_{r,R}$.  For concreteness, let's assume 
$\psi-2\pi n\in \Ipm_-$ on $\Gamma^-_{r,R}$;  all other cases may be handled in the same way.  From  Lemma~\ref{W-lemma} (with the above observation) we may then conclude that on $\Gamma_{r,R}^-(x_0)$  it holds:
\begin{align*}
C_\alpha\sqrt{W(u,\rho)}>&|\psi-\gamma-\alpha-2n\pi|\\
&=|\gamma_0-\alpha+2(\frac{\alpha}{\pi} + n)\theta^-(\rho)+\phi(\rho,\theta)-\gamma-\alpha-2n\pi|\\
&=|(\gamma_0-\gamma) + 2\alpha\left(\frac{\theta^-(\rho)}{\pi}-1\right) + 2n(\theta^-(\rho)-\pi)+\phi(\rho,\theta)|\\
&\ge |\phi(\rho,\theta)|-|\gamma_0-\gamma|-\frac{2\alpha}{\pi}|\theta^-(\rho)-\pi|-2n|\theta^-(\rho)-\pi|\\
&\ge |\phi|-C\rho.
\end{align*}
Therefore on $\Gamma_{r,R}^-(x_0)$
$$|\phi|\le C_\alpha\sqrt{W(u,g)}+O(\rho),$$
and similarly
on $\Gamma_{r,R}^+(x_0)$.

In consequence we have
\be\label{control}
\int_{\Gamma_{r,R}^{+}(x_0)\cup \Gamma_{r,R}^{-}(x_0)}|\phi|\frac{d\rho}{\rho}\le \int_{\Gamma_{r,R}^{+}(x_0)\cup \Gamma_{r,R}^{-}(x_0)} C_\alpha\frac1{\rho}\sqrt{W(u,g)}\, d\rho + O(1)
\ee

We split $\Gamma_{r,R}^{\pm}(x_0)$ in two parts: 
$$\Gamma_{r,R}^{\pm}(x_0)=\Gamma_{r,\eps^\gamma}^{\pm}(x_0)\cup \Gamma_{\eps^\gamma,R}^{\pm}(x_0),$$
and using the Corollary~\ref{easycor} we estimate:
\begin{align*}
& \int_{\Gamma_{r,\eps^\gamma}^{+}(x_0)\cup \Gamma_{r,\eps^\gamma}^{-}(x_0)} \frac1{\rho}\sqrt{W(u,g)}\, d\rho\\
&\qquad \le
\left[\left(\int_{\Gamma_{r,\eps^\gamma}^{+}(x_0)\cup \Gamma_{r,\eps^\gamma}^{-}(x_0)} 
W(u,g)\, d\rho\right)\left(\int_{\Gamma_{r,\eps^\gamma}^{+}(x_0)\cup \Gamma_{r,\eps^\gamma}^{-}(x_0)} \frac{d\rho}{\rho^2}\right)\right]^{\frac12}\\
&\qquad\le \frac{c}{r^{\frac12}\Upsilon} =o(1),
\end{align*}
since $r>\eps^s=\Upsilon^{-1}$.
Furthermore, by the global upper bound on the energy $E_\eps(u_\eps)\le K|\ln\eps|$,
\begin{align*}
&\int_{\Gamma_{\eps^\gamma,R}^{+}(x_0)\cup \Gamma_{\eps^\gamma,R}^{-}(x_0)} \frac1{\rho}\sqrt{W(u,g)}\, d\rho \\
&\qquad\le
\left[\left(\int_{\Gamma_{\eps^\gamma,R}^{+}(x_0)\cup \Gamma_{\eps^\gamma,R}^{-}(x_0)} 
W(u,g)\, d\rho\right)\left(\int_{\Gamma_{\eps^\gamma,R}^{+}(x_0)\cup \Gamma_{\eps^\gamma,R}^{-}(x_0)} \frac{d\rho}{\rho^2}\right)\right]^{\frac12}\\
&\qquad
\le C\Upsilon^{-1}|\ln\eps| \eps^{-\gamma}\to 0,
\end{align*}
since $s-\gamma>0$.

We are left with estimating the second term in \eqref{todo}:
\begin{align}\nnn \allowdisplaybreaks
\left|\int_{A_{r,R}}\frac{f^2-1}{\rho^2}\frac{\partial\phi}{\partial\theta}\, dx\right|
&\le \frac1{\eps^s}\int_{A_{r,R}}|f^2-1| |\frac1{\rho}\frac{\partial\phi}{\partial\theta}|\, dx\\ \nnn
&\le \eps^{-s}\left[ \int_{A_{r,R}}|f^2-1|^2 \int_{A_{r,R}}|\nabla\phi|^2\right]^{\frac12}\\ \nnn
&\le \eps^{-s}\left[ K\eps^2|\ln\eps|\int_{A_{r,R}}|\nabla\phi|^2\right]^{\frac12}\\ \nnn
&\le C\eps^{1-s}\sqrt{|\ln\eps|}\left[\int_{A_{r,R}}|\nabla\phi|^2\right]^{\frac12}\\
\allowdisplaybreaks
\label{end}
&=o(1)(\int_{A_{r,R}}|\nabla\phi|^2)^{\frac12} + o(1)
\end{align}
Finally, the last term in $\mathcal{E}$ is bounded below,
$$  \int_{A_{r,R}} f^2|\nabla\phi|^2\, dx \ge (1-\sqrt{2}a)^2\int_{A_{r,R}} |\nabla\phi|^2\, dx,  $$
and hence this positive term controls \eqref{end}.  Putting all of these estimates together, we obtained the desired lower bound on the residual term, $\int_{A_{r,R}} \mathcal{E}\, dx \ge C$, and the desired lower bound is established.
\end{proof} 

\begin{remark}\rm \label{extra_term}
We note that it is in deriving the estimate \eqref{control}, on the energy contribution of the  ``excess phase'' $\phi$ to the energy of a boundary defect, where we need to introduce the boundary penalization of $(|u|^2-1)^2$ in $W(u,g)$.  (See \eqref{eq:RP2} and the following remarks there.)  In particular, while we know that $|u_\eps|^2\ge 1-\sqrt{2}a>0$ away from the bad balls, this is not a strong enough estimate to control the error term in \eqref{control} without introducing additional logarithmically growing terms.
\end{remark}

\medskip

Next we compare the energies of boundary boojums and boundary vortices.  First, we remark that, since the phase $\psi\simeq \gamma\pm\alpha$ away from the bad balls, and $0<\alpha<{\pi\over 2}$, light and heavy boojums must be paired and in fact must alternate as we trace out $\Gamma=\partial\Om$.  In addition, given the lower bound \eqref{boojum_energy} for annuli, we observe that the lower bound for the energy of a ``ground state'' boojum pair, with $(n,\tau)=(0,-1), (1,1)$, is smaller than that of a boundary vortex since
$$  C_\alpha:= \left({\alpha\over\pi}\right)^2 + \left( 1-{\alpha\over\pi}\right)^2 < 1.  $$
This suggests that boojum pairs will always be energetically preferred over boundary vortices.  We will indeed show this in the course of proving the main theorem, but the following fundamental lemma is suggestive of this fact (and will be instrumental in proving it). 

In the following lemma we will denote by $n_j^0$ the degree of a boundary vortex, and by $n_i^+$ the degree of a heavy boojum while $n_i^-$ will be that of a light boojum:

\begin{lemma}\label{newdeglem}
Assume $ \exists \, N^v, N^b\in\{0,1,2,\dots\}$ and integers $\left\{\no\right\}_{j=1, \dots, N^v}, \left\{n^+_i,  n^{-}_i\right\}_{i=1, \dots, N^b}$ with $\sum_{i=1}^{N^b} (n^+_i+ n^{-}_i) +\sum_{j=1}^{N^v} \no=D.$ Then

\beq\label{inf}  
\sum_{j=1}^{N^v} ({\no})^2 +
\sum_{i=1}^{N^b} \left[ (\nm+\frac{\alpha}{\pi})^2+ (\np-\frac{\alpha}{\pi})^2\right] \ge |D|  \, C_{\alpha}.
\eeq

Furthermore, in case that $D>0$, we have equality if and only if 
$$\no=0\ \forall j,\quad \np=1\ \forall i, \quad \text{and } \ \nm=0, i=1, \dots, D=N^b,$$ while if $D<0$, we have equality if and only if 
$$\no=0\ \forall j, \quad \np=0\ \forall i, \quad \text{and }\ \nm=-1, i=1, \dots, D=N^b$$
\end{lemma}

\begin{remark}
We note that the minimum value is therefore obtain by $N^b$ boojum pairs of light and heavy boojums and with no boundary vortex.
\end{remark}

{\color{red} 
\begin{remark}  In proving Lemma~\ref{newdeglem} it is advantageous to list the different types of defects (light and heavy boojums and boundary vortices) separately, but for later purposes it will be more convenient to make a single list of the defects $y_{\eps,\ell}$, using the integer degree $n_\ell$ and boojum number $\tau_\ell\in \{0,-1,+1\}$ to distinguish their topological type.  In the latter case, the lower bound expressed in equation \eqref{inf} is reformulated as:
\beq\label{concis}\sum_{\ell=1}^{N^v+N^b}  (n_\ell-\tau_\ell \frac{\alpha}{\pi})^2 \ge |D|  \, C_{\alpha},
\eeq
 Recall that $\tau=0$ for a boundary vortex, $\tau=\pm1$ for a heavy, respectively light, boojum and that $N^v+N^b$ is the total number of boundary defects.
\end{remark}
}

\begin{proof}  Assume $D>0$.
We use induction on $D\in \NN$ and assume that the lower bound is false, i.e. that the inequality \eqref{inf} is reversed; for $D=1$ we then have:

$$\sum_{j=1}^{N^v} (n_j^0)^2 +
\sum_{i=1}^{N^b} \left[ (\nm +\frac{\alpha}{\pi})^2+(\np -\frac{\alpha}{\pi})^2\right] \le   C_{\alpha}.
$$

As we have a sum of positive terms,  this means:
$$\sum_{j=1}^{N^v} (n_j^0)^2\le  C_{\alpha}<1\implies n_j^0=0\ \forall j.$$
Further,
$$\sum_{i=1}^{N^b} \left[ (\nm+\frac{\alpha}{\pi})^2+( \np-\frac{\alpha}{\pi})^2\right] \le C_{\alpha},$$
which means that for each $i$ we have
$$\left[ (\nm+\frac{\alpha}{\pi})^2+( \np-\frac{\alpha}{\pi})^2\right] \le C_{\alpha}=(\frac{\alpha}{\pi})^2 + (1-\frac{\alpha}{\pi})^2<1,$$
and hence either
$$(\nm, \np)=(0,1) \text{ or } (0,0) \text{ or } (-1,0).$$
Since $D=1>0$, for at least one $i_0$ we have 
$$(n_{i_0}^-,  n_{i_0}^+)=(0,1),
\text{ and } (n_{i_0}^-+\frac{\alpha}{\pi})^2+( n_{i_0}^+-\frac{\alpha}{\pi})^2=C_{\alpha},$$
so 
$$\sum_{i\neq i_0} \left[ (n_i^++\frac{\alpha}{\pi})^2+( n_i^--\frac{\alpha}{\pi})^2\right] \le 0\implies N^b=1, i_0=1, (n_1^-, n_1^+)=(0,1).$$
This corresponds to the case of equality in \eqref{inf} with $D=1$, and hence we conclude that \eqref{inf} must hold in the case $D=1$.

Next, consider the case $D>1$ and assume 
\beq\label{assump}\sum_{j=1}^{N^v} (n_j^0)^2 +
\sum_{i=1}^{N^b} \left[ (\nm+\frac{\alpha}{\pi})^2+(\np-\frac{\alpha}{\pi})^2\right] \le  D\, C_{\alpha}.
\eeq
holds with D replaced by $D-m$, $m=1,2, \dots, D-1$, with equality if and only if $n_j^0=0\, \forall j, \nm=0\, \forall i, \text{and }\np=1, i=1, \dots, D-m=N^b$, and we will prove that this still holds for $D$, leading to the conclusion of the Lemma in the case that $D>0$.

Since $D>0$, there must be a positively charged defect somewhere. If $\exists \ n_j^0\ge 1$, we eliminate that boundary vortex to obtain a new configuration.
Indeed, assuming without loss of generality that $j=1$, we have a configuration
$$\left(  \left\{n_j^0\right\}_{j=2, \dots, N^v}, \left\{ \nm,\np\right\}_{i=1,\dots, N^b}\right),$$
with degree $D-n_1^0\le D-1$, and
\begin{align}\label{config}
 \sum_{j=2}^{N^v} (n_j^0)^2 +
\sum_{i=1}^{N^b} \left[ (\nm+\frac{\alpha}{\pi})^2+(\np-\frac{\alpha}{\pi})^2\right] & \le D  \, C_{\alpha}- (n_1^0)^2\\ \nonumber
&=(D-n_1^0)  C_\alpha + n_1^0  C_{\alpha} - (n_1^0)^2\\ \nonumber
&<(D-n_1^0)  C_{\alpha},
\end{align}
since $n_1^0\ge 1$ by assumption.

By induction hypothesis, we conclude that 
$$n_j^0=0\ \forall j=2, 3,\dots, N^v, N^b=D-n_1^0, \nm=0\ \forall i, \np=1\ \forall i=1, \dots, D-n_1^0.$$
This means that the left hand side in \eqref{config} becomes:
$$\sum_{j=2}^{N^v} (n_j^0)^2 +
\sum_{i=1}^{N^b} \left[ (\nm+\frac{\alpha}{\pi})^2+(\np-\frac{\alpha}{\pi})^2\right]=(D-n_1^0)  C_\alpha,$$
which is a contradiction, so this case cannot occur.

Thus, it must be that $\exists i$, which without loss of generality we will take to be $1$, with $n_1^-+ n_1^+\ge 1$. Moreover, since for $x<y$,
$$(x+\frac{\alpha}{\pi})^2 + (y-\frac{\alpha}{\pi})^2<(y+\frac{\alpha}{\pi})^2 + (x-\frac{\alpha}{\pi})^2,$$
we may assume that $ n^+_1\ge n^-_1, n^+_1\ge 1$  since otherwise we would switch $ n^+_1$ and $n^-_1$ and obtain a smaller value in \eqref{assump}.

Note that 
\begin{align}\nnn
(\nm+\frac{\alpha}{\pi})^2+(\np-\frac{\alpha}{\pi})^2&=
(\nm+\frac{\alpha}{\pi})^2+(\np-1+(1-\frac{\alpha}{\pi}))^2\\
&\ge (\np+\frac{\alpha^2}{\pi^2})+(\nm-1+(1-\frac{\alpha}{\pi})^2)
=\nm+\np -1 +C_{\alpha} \label{useful}
\end{align}
so  eliminating this boojum pair to create a new configuration with degree
$\tilde D=D-(\nm+\np)<D$, we have from \eqref{assump}
\begin{align*}
\sum_{j=1}^{N^v} (n_j^0)^2 +
\sum_{i=2}^{N^b} \left[ (\nm+\frac{\alpha}{\pi})^2+(\np-\frac{\alpha}{\pi})^2\right] & \le  C_{\alpha} D -[n^-_1+n^+_1 -1 +C_{\alpha}]\\
&=\left\{C_{\alpha}(D-1)-[n^-_1+ n^+_1-1]\right\}\\
&\le C_{\alpha}  [(D-1)-(n^-_1+ n^+_1-1)]\\
&= \tilde D  C_{\alpha}
\end{align*}
with equality if and only if $n^-_1=0,  n^+_1=1$. By the induction hypothesis, $\tilde D<D$, and
$$n_j^0=0\ \forall j, N^b=\tilde D + n^-_1 + n^+_1=D,n^-_i=0,  n^+_1=1, i=1, \dots, D.$$
This completes the proof for $D>0$.

When $D<0$, we note that this reduces to the postive case if we replace $n_j^0\to -n_j^0$, $n_j^+\to -n_j^-$, and $n_j^-\to -n_j^+$, that is, negative degree heavy boojums are counted the same way as positive degree light boojums, and vice-versa.  The case $D=0$ is trivially true.

\end{proof}

\section{Proof of the Main Theorem}\label{TheoremProof}

Returning to the bad balls constructed in Lemma~\ref{sballs}, we may pass to a subsequence $\eps_n\to 0$ for which there exist points $\xi_k\in\Om$, $k=1,\dots,N^v$, $\zeta_k\in \Gamma$, $k=1,\dots,N^b$, for which 
\be\label{limit_points}
\begin{gathered}  y_{\eps_n,j}\longrightarrow \zeta_k \quad \text{for some $k\in\{1,\dots,N^b\}$, and } \\
p_{\eps_n,i}\longrightarrow \xi_k \ \text{for some $k\in\{1,\dots,N^v\}$, or }\quad
p_{\eps_n,i}\longrightarrow \zeta_k \ \text{for some $k\in\{1,\dots,N^b\}$. }
\end{gathered}
\ee
That is, certain interior vortices $y_i$ for $u_\eps$ might accumulate at a boundary point $\zeta_k\in\Gamma$.  Our task is to provide a global lower bound on the energy to match the upper bounds from Lemma~\ref{upperbound}.  To do this we adapt techniques of vortex ball analysis introduced by Jerrard \cite{Jerrard} and Sandier \cite{Sandier}, but we must treat the various types of defect (boojums, boundary vortices, interior vortices, and interior vortices approaching the boundary) with care, as each leads to a different contribution to the energy.

For $\sigma>0$, define 
\be\label{OmSig} \mathcal{B}_\sigma:=\left(\bigcup_{k=1}^{N^v}B_\sigma(\xi_k)\right)\cup\left(\bigcup_{k=1}^{N^b}B_{\sigma^s}(\zeta_k)\right), \quad\text{and}\quad \Om_\sigma:=\Om\setminus\mathcal{B}_\sigma.
\ee

\begin{lemma}\label{Etienne}  Let $1>\sigma>0$ be fixed.  Then there exists $C=C(g,\alpha,\Om)$ such that:
$$  E_{\eps_n}(u_{\eps_n};\mathcal{B}_\sigma) \ge  2\pi\, sC_\alpha \, |\mathcal{D}|\, \ln\left({\sigma\over\eps}\right) +C,
\quad\text{if \ $sC_\alpha\le\frac12$},  $$
and
$$  E_{\eps_n}(u_{\eps_n};\mathcal{B}_\sigma) \ge  \pi \, |\mathcal{D}|\, \ln\left({\sigma\over\eps}\right) +C,
\quad\text{if \ $sC_\alpha>\frac12$},  $$
\end{lemma}

Combined with Lemma~\ref{upperbound} we may then conclude that the energy is bounded above away from any neighborhood of the defects:

\begin{corollary}\label{EnergyBA}  For any $\sigma>0$, there exists $C$ such that:
$$  E_{\eps_n}(u_{\eps_n};\Om_\sigma) \le \begin{cases}  2\pi\, sC_\alpha \, |\mathcal{D}|\, |\ln\sigma| +C, &\text{if $sC_\alpha\le\frac12$}, \\
\pi \, |\mathcal{D}|\, |\, |\ln\sigma| +C, &\text{if \ $sC_\alpha>\frac12$}.
\end{cases}
$$
\end{corollary}

\begin{proof}[Proof of Lemma~\ref{Etienne}]  We let $1>\sigma>0$ be given, but such that 
$$\sigma^s < 4\min\left\{ |\xi_i-\xi_j|, |\zeta_i-\zeta_j|, \text{dist}(\xi_i,\Gamma) \ 
   | \ i\neq j, \ i=1,\dots,N^v, \ j=1,\dots,N^b\right\},
   $$
so that the balls $B_\sigma(\xi_i),B_{\sigma^s}(\zeta_j)$ be well separated in $\overline{\Om}$.  We take $\eps\to 0$ along the subsequence employed in \eqref{limit_points} above.  
By taking a further subsequence we may assume that each of the centers of the bad balls constructed in Lemma~\ref{sballs} lies withing $\sigma$ of its limiting $\xi_i$ or $\zeta_j$.  We assume $\eps\to 0$ along this subsequence, but without explicit notation by subscripts.

First, we separate out the ``bad balls'' defined in Lemma~\ref{sballs} whose centers converge to interior points $\xi_j\in\Om$ (see \eqref{limit_points} above).  Along the subsequence, the total degree due to these, $D^{int}$ is constant, and applying the result of Sandier or Jerrard in a slightly smaller domain $\Om'\Subset\Om$, we have the  
lower bound for the energy in this collection of interior balls,
\be\label{LB_int}  E_\eps\left( u_\eps\, ; \, \bigcup_{k=1}^{N^v} B_\sigma(\xi_k)\right)
\ge \pi |D^{int}|\, \ln \left({\sigma\over\eps}\right) + c, 
\ee
for any fixed $\sigma>0$.

For the bad balls which accumulate on the boundary we employ the same procedure introduced in \cite{Sandier} to obtain lower bounds for the classical Ginzburg-Landau functional with Dirichlet boundary condition, adapted to deal with boundary defects.  We construct families of balls, $\mathcal{B}(t)$, $t\ge 1$, containing the bad balls of Lemma~\ref{sballs}, and growing in time.  Each ball $B^i(t)\in\mathcal{B}(t)$ carries a degree, a radius $R_i(t)$, and a ``seed size'' $r_i(t)$, which in some sense remembers the scale of the original ball ($O(\eps)$ or $O(\eps^s)$.)  The lower bound is derived through a two-step evolution. The first step is ``expansion'': to continuously grow the balls' radii, and use Proposition~\ref{elb} to estimate the energy in the annuli contained between the expanded balls and the original bad balls, in terms of the logarithm of the ratio of the radii.  When two or more balls come into contact, the second ``merger'' step combines balls together, and uses Lemma~\ref{newdeglem} to bound the energy in the resulting larger balls from below.  These two steps are repeated until the radii of the growing balls exceeds $\sigma$.

To do this we need to further separate the remaining balls into two classes, those whose centers lie on the boundary versus interior balls converging to the boundary.  Assume that $sC_\alpha\le\frac12$; the modifications required for the opposite case will be described at the end of the proof.

%
%
%

\noindent\underline{\sl Step 0: Initialization.} \
Set the initial time $t_0=1$ (for convenience), and begin the process with the remaining bad balls as defined in Lemma~\ref{sballs}, numbered in a single list, $B^i(t_0)$, $i=1,\dots,N_0$. 
Define two index sets $\mathcal{I}_b, \mathcal{I}'$ as follows: 
\begin{itemize}
\item
For $i\in  \mathcal{I}_b$, $B^i(t_0)$ are centered on the boundary $\Gamma$. Boundary balls have initial radii $r_i(t_0)=\kappa\eps^s$, and carry both a degree $n_i$ and a Boojum number $\tau_i$.  We let $D^0:=\sum_{i\in\mathcal{I}_b} n_i$, the total degree of defect balls centered on $\partial\Om$.
\item 
For $i\in  \mathcal{I}'$,  $B^i(t_0)$ lie in the interior (but accumulate on the boundary as $\eps\to 0$.)  These balls have degree $d_i$ and initial radii $r_i(t_0)=\kappa\eps$.  The total degree of defect balls approaching $\partial\Om$ will be denoted $D'$. 
\end{itemize}
We note that by Lemma~\ref{degree_lemma} we must have $D^{int}+D^0+D'=\mathcal{D}$, the total degree associated to the boundary function $g$.
We also define
$$   D^b:= D^0+D' =\mathcal{D}-D^{int},  $$
and note that $D^b$ is also constant in $\eps$.

The values $r_i(t_0)$ are initially chosen to be the actual radii of vortex balls, but following Sandier \cite{Sandier} we think of them as a ``seed size'', which will change in the process of merging the expanding balls but retain the order of magnitude of the original radii.

\medskip

\noindent \underline{\sl Step 1: Initial expansion.} \  We grow the radii of each ball continuously in time $t$, maintaining a uniform ratio of the radius of each ball to its initial radius.  We do this in a different way depending on the two classes of bad balls near the boundary.  We recall that the initial time $t_0=1$.
 If $B_1(t_0)$ is centered on the boundary $\Gamma$ and $B_2(t_0)$ is an interior ball converging to the boundary, we require that their radii $R_1(t), R_2(t)$ satisfy:
\be\label{ratio_radii}  {R_1(t)\over r_1(t_0)} =  {t\over t_0} =\left({R_2(t)\over r_2(t_0)}\right)^s.  
\ee
By (ii) of Lemma~\ref{sballs} we can increase $t$, expanding each ball for some positive time $t>t_0=1$, in which the balls will remain disjoint.  
During this expansion phase the seed sizes $r_i(t)=r_i(t_0)$ remain constant.
Call $t_1>t_0=1$ the first time at which two or more expanding balls touch.   Applying Proposition~\ref{elb} to each ball at time $1=t_0\le t\le t_1$, we obtain a lower bound in each annulus, of outer radius $R_i(t)$ and inner radius $r_i(t_0)$ around each center:
\begin{align}  \nnn
 E_\eps\left(u_\eps \, ; \ \bigcup_{i=1}^{N_0}B^i(t)\right)  
    &\ge E_\eps\left(u_\eps \, ; \  \bigcup_{i=1}^{N_0}(B^i(t)\setminus B^i(t_0))\right)
    \\  \nnn
    &\ge \sum_{i\in\mathcal{I}_b} 2\pi \left( n_i -\tau_i{\alpha\over\pi}\right)^2
    \ln{R_i(t)\over r_i(t_0)}
      + \sum_{i\in\mathcal{I}'} \pi d_i^2 \ln{R_i(t)\over r_i(t_0)}
  + c_1  \\
 &= \left\{\sum_{i\in\mathcal{I}_b} 2\pi \left( n_i -\tau_i{\alpha\over\pi}\right)^2  + \sum_{i\in\mathcal{I}'} {\pi\over s} \, d_i^2 
 \right\}\ln {t\over t_0} + c_1  \label{annulusbound}
\end{align}
Using Lemma~\ref{newdeglem}, (see also \eqref{concis}) we then obtain a lower bound on the energy in the expanded balls near $\partial\Om$,
\begin{align}
\nnn
E_\eps\left(u_\eps \, ; \ \bigcup_{i=1}^{N_0}B^i(t)\right)
   &\ge \left[ 2\pi C_\alpha \, |D^0| + {\pi\over s} |D'| \right]\ln {t\over t_0} + c_1 \\
   \label{LB_postexp}
   &\ge \pi\mu |D^b| \, \ln {t\over t_0} + c_1,
\end{align}
for $1=t_0<t\le t_1$, where we recall $D^b=D^0+D'$ and denote
\be\label{mu_def}   \mu=\min\left\{ 2C_\alpha \ , \ {1\over s}
\right\}.
\ee

\bigskip

\noindent \underline{\sl Step 2: Merging.} \  At time $t_1$, some of the expanding balls will come into contact with each other, in the sense that their closures will intersect.  The merging process is based on the observation:
\be\label{observation}
\text{if ${R_1\over r_1}=t={R_2\over r_2}$ then ${R_1+R_2\over r_1+r_2}=t$.}
\ee
Thus, we can combine balls whose closures touch into new balls by summing the radii, and the lower bound will be preserved if we adjust the ``seed size'', which remembers the radii of the initial balls, accordingly.  That is, the new denominator $\tilde r=r_1+r_2$ will no longer be the initial radius but will be a quantity of the same order of magnitude. 

If the closures of two or more interior balls $B^1(t_1),\dots,B^k(t_1)$ touch, (but remain disjoint from boundary balls,) they are enclosed within a new interior ball, $\tilde B^j(t_1)$, of radius $\tilde R_j(t_1):=R_1(t_1)+\cdots+R_k(t_1)$.  The degree of this new ball will be $\tilde d_j=\sum_{i=1}^k d_i$, and we will choose a new ``seed size'' $\tilde r_j(t_1):=r_1(t_0)+\cdots+ r_k(t_0)=O(\eps)$.  In this way, we are maintaining the ratio,
$$    \left[{\tilde R_j(t_1)\over \tilde r_j(t_1)}\right]^s = \left[{R_i(t_1)\over r_i(t_0)}\right]^s={t_1\over t_0}, \quad i=1,\dots, k.  $$
The energy contained in the new ball at $t=t_1$ may be bounded below, 
\begin{align} \nnn E_\eps\left(u_\eps \, ; \, \tilde B^j(t_1)\right)
   &\ge \sum_{i=1}^k  E_\eps\left(u_\eps \, ; \, B^i(t_1)\right)  \\
   \nnn
   &\ge \sum_{i=1}^k \pi d_i^2 \ln \left[{t_1\over t_0}\right]^{\frac1s} + O(1)   \\
   &\ge {\pi\over s} |\tilde d_j|\, \ln {t_1\over t_0} + O(1). \label{merge1}
   \end{align}

The case of two or more boundary balls $B^1(t_1),\dots,B^k(t_1)$ merging is only slightly more complicated.  As above, the new merged ball $\tilde B^j(t_1)$ will have radius $\tilde R_j(t_1)=\sum_{i=1}^k R_i(t_1)$, and new ``seed size'' $\tilde r_j(t_1):=r_1(t_0)+\cdots+ r_k(t_0)=O(\eps^s)$.
We recall that light and heavy boojums must alternate along $\Gamma$, and thus if we enclose two or more boundary balls in a larger $\tilde B^j(t_1)$, the boojum number of the merged ball $\tilde\tau_j=\sum_{i=1}^k \tau_i\in \{-1,0,1\}$.  Likewise, the degree also sums, $\tilde n_j=\sum_{i=1}^k n_i\in\ZZ$.  Thus, the new boundary ball's energy may be bounded below by:
\begin{align} \nnn E_\eps\left(u_\eps \, ; \, \tilde B^j(t_1)\right)
   &\ge \sum_{i=1}^k  E_\eps\left(u_\eps \, ; \, B^i(t_1)\right)  \\
   \label{merge2}
   &\ge \sum_{i=1}^k 2\pi \left(n_i-\tau_i{\alpha\over\pi}\right)^2 \ln \left[{t_1\over t_0}\right] + O(1). 
      \end{align}

The most delicate case is when interior defect balls collide with boundary balls.  (If interior balls contact $\partial\Om$ itself, we can think of this as the merger of interior balls with an empty boundary ball, of radius, degree, and boojum number all zero.)  Assume $k$ boundary balls and $\ell$ interior balls (with radii $R_i(t_1)$ and seed size $r_i(t_0)$) meet at $t=t_1$.  We create a new boundary ball $\tilde B^j(t_1)$ with radius and new seed size,
$$   \tilde R_j(t_1)=\sum_{i=1}^k R_i(t_1) + \sum_{i=k+1}^{k+\ell} [R_i(t_1)]^s, \qquad
\tilde r_j(t_1)=\sum_{i=1}^k r_i(t_0) + \sum_{i=k+1}^{k+\ell} [r_i(t_0)]^s =O(\eps^s).  $$
As $0<s<1$, $\tilde R_j(t_1)$ is larger than the sum of the radii of the old balls, and so $\tilde B^j(t_1)$ encloses each of the merging balls inside.
Employing the key observation \eqref{observation}, we obtain the same lower bound on the energy in the new boundary ball \eqref{merge2} as in the previous case.

Putting each case together, we have created a new family of defect balls $\{\tilde B^j(t_1)\}_{j=1,\dots,N_1}$, whose union contains the expanded balls $\bigcup_{i=1}^{N_0} B^i(t_1)$.  By the merging process, the closure of these balls is disjoint. We divide the balls into two classes via the index sets 
$\tilde{\mathcal{I}}_b$ and 
$\tilde{\mathcal{I}'}$ (separating the family into balls centered on the boundary versus those in the interior but approaching the boundary,) and denote by 
$$ \tilde D^0=\sum_{j\in \tilde{\mathcal{I}}_b} \tilde n_j\quad\text{ and }\quad
\tilde D'=\sum_{j\in \tilde{\mathcal{I}}'} \tilde d_j, $$ 
the total degrees.  Then, applying Proposition~\ref{newdeglem}
we have a lower bound:
\begin{align}  \nnn
E_\eps\left( u_\eps \, ; \, \bigcup_{j=1}^{N_1}\tilde B^j(t_1)\right)
&\ge \left[\sum_{j\in\tilde{\mathcal{I}}_b} 2\pi \left(n_i-\tau_i{\alpha\over\pi}\right)^2
   + \sum_{j\in\tilde{\mathcal{I}}'}{\pi\over s} |\tilde d_j|\right] \ln {t_1\over t_0} + O(1) \\  \nnn
  &\ge \left[ 2\pi C_\alpha\, |\tilde D^0| + {\pi\over s} |\tilde D'|\right] 
      \ln {t_1\over t_0} + O(1)
      \\  \label{merge_LB}
      &\ge \pi\mu |D^b| \,\ln {t_1\over t_0} + O(1),
\end{align}
in terms of the new merged defect balls.  (We note that, while $\tilde D^0, \tilde D'$ may be different from $D^0,D'$ because of merging, $D^b=D^0+D'=\tilde D^0+\tilde D'$.)

\medskip

\noindent \underline{\sl Step 3:  Repeat as necessary.} \ 
Restart the expansion process in Step 1, but starting now with the merged balls $\{\tilde B^j(t_1)\}_{j=1,\dots,N_1}$, whose closures are disjoint. Dropping the tildas, we expand the balls according to \eqref{ratio_radii} but for $t\ge t_1$, with the new seed sizes $r_j(t_1)$.  Again, expansion may continue until two or more expanded balls touch, at some $t_2>t_1$.  Applying Proposition~\ref{elb} in each annular region 
$B^j(t_2)\setminus\overline{B^j(t_1)}$, we obtain a lower bound analogous to \eqref{annulusbound}
\begin{align*}  
 E_\eps\left(u_\eps \, ; \ \bigcup_{i=1}^{N_1}\left[B^j(t)\setminus B^j(t_1)\right]\right)   
    &\ge \sum_{i\in\mathcal{I}_b} 2\pi \left( n_j -\tau_j{\alpha\over\pi}\right)^2
    \ln{R_j(t)\over r_j(t_1)}
      + \sum_{j\in\mathcal{I}'} \pi d_j^2 \ln{R_j(t)\over r_j(t_1)}
  + c_1  \\
 &= \left\{\sum_{j\in\mathcal{I}_b} 2\pi \left( n_j -\tau_j{\alpha\over\pi}\right)^2  + \sum_{i\in\mathcal{I}'} {\pi\over s} \, d_j^2 
 \right\}\ln {t\over t_1} + O(1) \\
 &\ge \left[ 2\pi C_\alpha \, |D^0| + {\pi\over s} |D'| \right]\ln {t\over t_1} + O(1)
 \\
 &\ge \pi\mu\, |D^b|\, \ln {t\over t_1} + O(1),
\end{align*}
with $\mu$ defined in \eqref{mu_def}.  Combining this with  \eqref{merge_LB} we have improved our lower bound to:
\begin{align*}  E_\eps\left(u_\eps \, ; \ \bigcup_{i=1}^{N_1}B^j(t)\right)
&\ge E_\eps\left(u_\eps \, ; \ \bigcup_{i=1}^{N_1}\left[B^j(t)\setminus B^j(t_1)\right]\right) + E_\eps\left(u_\eps \, ; \ \bigcup_{i=1}^{N_1} B^j(t_1)\right)  \\
& \ge \pi\mu |D^b| \, \ln {t\over t_0} + O(1),
\end{align*}
for all $t\in (t_1, t_2)$.  

This process must terminate after a bounded finite number of steps, as by Lemma~\ref{sballs} the number of bad balls is uniformly bounded in $\eps$.  After all the mergers are finished, there are only $N^b$ boundary balls remaining, each centered on $\Gamma$, converging to the points $\zeta_k\in\Gamma$, and the expansion step may continue without interruption until the sum of the radii $\sum_{j=1}^{N^b}R_j(t_*)= \sigma^s/2$.  Since the seed size $r_j(t)=O(\eps^s)$ for boundary centered balls, we obtain (for all sufficiently small $\eps$ in the subsequence),
\begin{align*}
E_\eps\left(u_\eps \, ; \ \bigcup_{j=1}^{N^b}B_{\sigma^s}(\zeta_j)\right)
&\ge E_\eps\left(u_\eps \, ; \ \bigcup_{i=1}^{N^b}B^j(t_*)\right) \\
&\ge \pi\mu |D^b| \, \ln {t_*\over t_0} + O(1) \\
&\ge  \pi\mu |D^b| \, \ln {\sigma^s\over \eps^s} + O(1) \\
&= s\pi\mu |D^b| \, \ln {\sigma\over \eps} + O(1)
\end{align*} 
For a lower bound on the energy contained in all of the balls, we include the lower bound \eqref{LB_int} on the energy of defect balls contained in the interior of $\Om$.  Thus, we have
$$  E_{\eps}(u_{\eps};\mathcal{B}_\sigma) \ge  \pi\left[\mu s \, |D^b| + |D^{int}|\right] \ln\left({\sigma\over\eps}\right) +C,
$$
where $\mathcal{B}_\sigma$ is defined in \eqref{OmSig}.
In case $\mu=1/s\le 2C_\alpha$, since $\mathcal{D}=|D^b+D^{int}|\le |D^b|+|D^{int}|$, we obtain the desired lower bound and the proof of the lemma is complete.  In case $\mu=2C_\alpha< 1/s$, we have $1>2sC_\alpha$ and a similar argument leads to the desired lower bound,
\begin{align} \nnn   E_\eps(u_{\eps};\mathcal{B}_\sigma) &\ge  \pi\left[2sC_\alpha \, |D^b| + |D^{int}|\right] \ln\left({\sigma\over\eps}\right) +C \\
\nnn
&\ge 2\pi s C_\alpha |\mathcal{D}|\, \ln{\sigma\over\eps} + c_2.  
\end{align}
\end{proof}

\begin{remark}
\rm   We note that when defining the collection of bad balls $\mathcal{B}_\sigma$ we may delete any balls (interior or boundary) for which the degree $\deg(u_\eps;\partial B_\eps(\xi))=0$ (or $\deg(u_\eps;\partial B_{\eps^s}(\xi))=0$ for boundary balls.)  Doing so does not change the lower bound on the energy contained in the bad set $\mathcal{B}_\sigma$, and thus any bad balls with net degree zero form part of the ``regular'' set where $u_\eps$ converges.
\end{remark}

\begin{proof}[Proof of Theorem~\ref{main}]
From Corollary~\ref{EnergyBA} we can choose a subsequence $u_\eps$ of minimizers which is bounded in $H^1\left(\Om\setminus\mathcal{B}_\sigma;\mathbb{C}\right)$, for any small $\sigma>0$, and for which the corresponding bad balls (from Lemma~\ref{sballs}) converge to the defect sites $\xi_i\in\Om$ or $\zeta_j\in\Gamma$.  By the upper bound in Corollary~\ref{EnergyBA} (extracting another subsequence, if necessary,) $u_\eps\wto u_*$ in $H^1_{loc}$ away from the defects $\xi_i\in \Om$, $\zeta_j\in\Gamma$, with $|u_*|=1$.  As in \cite{Struwe} the limiting $u_*$ is an $\mathbb{S}^1$-valued harmonic map with defects on the finite point set $\Sigma:=\{\xi_i,\zeta_j \ | \ i=1,\dots, N^v, \ j=1,\dots,N^b\}$.  Following \cite{BBH2, Riviere} the convergence may be improved to 
$H^1_{loc}(\Om\setminus\Sigma)$ and $C^{1,\beta}_{loc}(\bar\Om\setminus\Sigma)$.  
In particular, passing to the limit in \eqref{EL} on $\Gamma\setminus \{\zeta_j\}_{j=1,\dots,N^b}$, we may conclude that $W(u_*,g)=0$ away from the boundary defects. That is, $u_*=ge^{\pm i\alpha}$ on the boundary arcs determined by the defects on $\Gamma$.  We may also conclude that the degrees $d_i$ (corresponding to interior limiting defects $\xi_i$) and $n_j$ (corresponding to boundary limiting vortices $\zeta_j$, with boojum number $\tau_j$) are preserved in the limit.

We now fix 
$$R\le \frac14\min\left\{ |\xi_i-\xi_j|, |\zeta_i-\zeta_j|, \text{dist}(\xi_i,\Gamma) \ 
   | \ i\neq j, \ i=1,\dots,N^v, \ j=1,\dots,N^b\right\},
   $$
so that the balls of radius $R$ around each defect are pairwise disjoint.  For any $\sigma>0$ with  $\sigma^s<R$, we apply Proposition~\ref{elb} to $u_\eps$ in each annular region $A_{\sigma,R}(\xi_i)$, $A_{\sigma^s,R}(\zeta_j)$, to obtain a lower bound,
\begin{align*}  E_\eps(u_\eps; \Om_\sigma) &\ge
       E_\eps\left(u_\eps; \Om\setminus\left[\bigcup_i A_{\sigma,R}(\xi_i)\cup
         \bigcup_j A_{\sigma^s,R}(\zeta_j)\right]\right)  \\
         &\ge  \left[ \pi \sum_i  d_i^2 + 2 \pi \, s \sum_j \left(n_j -\tau_j{\alpha\over\pi}\right)^2\right] \ln{R\over \sigma} + O(1).
\end{align*}
From the upper bound in Corollary~\ref{EnergyBA} and Lemma~\ref{newdeglem} we then have:
\begin{align*}
\pi\mu \; s \mathcal{D}\, |\ln \sigma| &\ge E_\eps(u_\eps; \Om_\sigma) \\
   &\ge \left[ \pi \sum_i  d_i^2 + 2 \pi \, s \sum_j \left(n_j -\tau_j{\alpha\over\pi}\right)^2\right] \ln{R\over \sigma} + O(1)  \\
   &\ge \left[ \pi \sum_i |d_i| + 2 \pi C_\alpha s |D^b| \right] \ln{R\over \sigma} + O(1) \\
   &\ge \left[ \pi |D^{int}| + 2 \pi C_\alpha s |D^b| \right] \ln{R\over \sigma} + O(1),
\end{align*}
with equality if and only if $d_i=\text{sgn}\,D^{int}$, and $n_j=0$ for $\tau_j=0,1$ while $n_j=1$ for $\tau_j=-1$.   As this inequality is true for all $\sigma>0$, we must have 
\be\label{equality}  s\mu |\mathcal{D}| \ge \sum_i |d_i| + 2 C_\alpha s |D^b| \ge 
s\mu \left( |D^{int}| +  |D^b|\right) \ge
s\mu |\mathcal{D}|, 
\ee
and thus each term is equal.  

If we assume $\mu=2C_\alpha<1/s$ we must then conclude that $d_i=0$ $\forall i$, and $\mathcal{D}=D^b$, so all of the topologically nontrivial defects occur on the boundary, with degrees  $n_j=0$ for $\tau_j=0,1$ while $n_j=1$ for $\tau_j=-1$.  Consequently, there are exactly $2\mathcal{D}$ defects, alternating between light and heavy boojums.  On the other hand, if $\mu=1/s< 2C_\alpha$ then the equality of each term in \eqref{equality} forces $D^b=0$, and hence $\mathcal{D}=D^{int}$ and $d_i=1$ for all $i$.  In this case there are exactly $\mathcal{D}$ interior defects.  When $2C_\alpha s=1$ each defect (boundary or interior) has the same energy cost at highest order, and the question of characterizing the defects is more subtle.
\end{proof}

\section{Renormalized Energies}\label{RNE}

After proving Theorem~\ref{main}, which details the nature of defects under weak anchoring to oblique angle condition at the boundary, it is natural to ask whether we may determine the location of boojums in the $\eps\to 0$ limit.  This involves verifying (as in \cite{BBH2}) that the defect locations minimize a Renormalized Energy, determined by a more precise asymptotic expansion of the energy which identifies the order-one term.  Rather than carry out the necessary estimates as in \cite{BBH2}, we argue formally in this section in order to give the form of the Renormalized Energy and come to some heuristic conclusion in special geometries relevant to physical cases.  

As the case $2sC_\alpha>1$ is essentially the same as the Dirichlet case studied in \cite{BBH2}, we restrict our attention to the more novel situation $2sC_\alpha<1$ in which minimizers exhibit boojum pairs, which (as in the statement of Theorem~\ref{main},) we denote  by $y_j, \tilde y_j$, $j=1,\dots,\mathcal{D}$, with $y_j$ a light, and $\tilde y_j$ a heavy, boojum.  
We recall that (along a subsequence) $u_\eps\to u_*=\exp(i\varphi_*)$, in $C_{loc}^{1,\beta}(\overline{\Om}\setminus\{y_1,\tilde y_1,\dots,y_{\mathcal{D}}, \tilde y_{\mathcal{D}})$.  The limit  $u_*$ is an $\mathbb{S}^1$-valued harmonic map in $\Om$ with defects on the given boundary points.  In addition, $W(u_*,g)=0$ on $\Gamma \setminus\{y_1,\tilde y_1,\dots,y_{\mathcal{D}}\}$, and thus on $\Gamma$,
$u_*=g\exp(\pm i\alpha)$, with jumps of $-2\alpha$ or $-(2\pi-2\alpha)$ at the light and heavy boojums, $y_j, \tilde y_j$ respectively.
 These must alternate along $\Gamma$, and the heavy boojum carries a degree $n_j=1$ for each $j=1,\dots,\mathcal{D}$.  As described in \cite{BBH2, Riviere}, the energy of minimizers is then expanded as:
\be\label{asympexp}
E_\eps(u_\eps) =  \mathcal{D}(\pi \, s\, C_\alpha|\ln\eps|+Q_b)   + W(y_1,\tilde y_1,\dots,y_{\mathcal{D}}, \tilde y_{\mathcal{D}})  +o(1),
\ee 
where $Q_b$ is a constant, representing the energy of boojum cores at the length scale $\eps^s$.  This defines the Renormalized Energy,
 $W: \ \Gamma^{2\mathcal{D}}\to \R$, where $\Gamma^{2\mathcal{D}}$ is the set of all ${2\mathcal{D}}$ points on $\Gamma$.  To connect $W$ to the limit $u_*=\exp(i\varphi_*)$ of the minimizers $u_\eps$, we define the conjugate harmonic function to the phase $\varphi_*$: 
$\Phi(x)=\Phi(x; \{y_j,\tilde y_j\})$ with  $\nabla\Phi=-\nabla^\perp\varphi_*$.  
The conjugate solves
\be\label{Phi}
\left.
\begin{gathered}
\Delta\Phi=0, \quad\text{in $\Omega$,}\\
\frac{\partial\Phi}{\partial\nu}= g\wedge g_\tau - \sum_{j=1}^{\mathcal{D}} 
\left[2\alpha\,\delta_{y_j}(x)+2(\pi-\alpha)\delta_{\tilde y_j}(x) \right], \quad\text{on $\Gamma$}.
\end{gathered}
\right\}
\ee
 the boundary condition reflecting the jump in the harmonic phase $\varphi_*$ at light and heavy boojums.
 Then, it may be shown \cite{BBH2} that the Renormalized Energy is given by
\be\label{RNen}  W(y_1,\tilde y_1,\dots,y_{\mathcal{D}}, \tilde y_{\mathcal{D}}):=
   \lim_{\rho\to 0}\left(
     \frac12\int_{\Omega\setminus\mathcal{B}_\rho} 
     |\nabla\Phi(x; \{y_j,\tilde y_j\})|^2\, dx
        -\pi \, s\, C_\alpha\mathcal{D} \ln\frac{1}{\rho}\right),
\ee
where $\mathcal{B}_\rho=\bigcup_{j=1}^{\mathcal{D}} [B_\rho(y_j)\cup B_\rho(\tilde y_j)]$. 

In the special case where $\Om=B_1$, the unit disk, with equivariant $g=g(\theta)=e^{i\mathcal{D}\theta}$ of degree $\mathcal{D}>0$, the Renormalized Energy may be expressed simply.  In that case, $g\wedge g_\tau=\mathcal{D}$ is constant, so $\Phi$ is a linear combination of Green's functions $G(x,p)$ with pole $p\in\Gamma$: 
$$  -\Delta_x G(x,p) = 0, \ \text{in $\Omega$,} \quad
     \frac{\partial G}{\partial\nu_x}(x,p)=1-2\pi\delta_p(x), \ \text{for $x\in\Gamma=\partial B_1$,}
$$
whose solution is  $G(x,p)=2\ln |x-p|$.  Then, we have:
$$  \Phi(x;  \{y_j,\tilde y_j\})= \sum_{j=1}^{\mathcal{D}} 
  \left[ {\alpha\over\pi} G(x,y_j) 
      + \left(1-{\alpha\over\pi}\right) G(x, \tilde y_j)\right]. $$
Substituting into \eqref{RNen} and integrating by parts (see \cite[I.4]{BBH2} for interior vortices and \cite[Section 6]{ABGS} for boundary defects,) we obtain an explicit formula for the Renormalized Energy,
\begin{align*}  W(y_1,\tilde y_1, \dots, y_{\mathcal{D}},\tilde y_{\mathcal{D}})
&= \sum_{i,j=1\atop i\neq j}^{\mathcal{D}} 
  \left[ {\alpha\over\pi} G(y_i,y_j) 
      + \left(1-{\alpha\over\pi}\right) G(\tilde y_i, \tilde y_j)\right]
      - \sum_{i,j=1}^{\mathcal{D}} G(y_i,\tilde y_j) \\
  &=  -2\sum_{i,j=1\atop i\neq j}^{\mathcal{D}} 
  \left[ {\alpha\over\pi} \ln |y_i-y_j|
      + \left(1-{\alpha\over\pi}\right) \ln| \tilde y_i- \tilde y_j|\right]
      - 2\sum_{i,j=1}^{\mathcal{D}} \ln |y_i-\tilde y_j|.
\end{align*}
In the case $\mathcal{D}=1$ there is only one boojum pair and  
$W(y_1,\tilde y_1)=-4\ln |y_1-\tilde y_1|$, so it is clear that the light and heavy boojums must be antipodally placed on the circle $\Gamma$.  For the Landau-de Gennes case $\mathcal{D}=2$, different weights appear in the sum,
\begin{align*}  W(y_1, \tilde y_1, y_2, \tilde y_2)
&= -4\left[{\alpha\over\pi} \ln |y_1-y_2|
      + \left(1-{\alpha\over\pi}\right) \ln| \tilde y_1- \tilde y_2|\right] \\
      &\qquad
 - 2 \ln\left[|y_1-\tilde y_1|\,|y_1-\tilde y_2|\,|y_2-\tilde y_1|\,|y_2-\tilde y_2|\right],
\end{align*}

\section{Numerical Examples}
In this section we present several examples of configurations with vortices that can be observed in numerically computed critical points of the energy \eqref{eq:E}. These are obtained by simulating gradient flow for $E_{\eps}^{g,\alpha}$ using the finite elements software package COMSOL \cite{COMSOL}. 

Note that we do not claim that solutions that we obtain are minimizers of $E_{\eps}^{g,\alpha}$ or prove that these solutions converge to critical points of the limiting energy. Rather, we use numerical simulations as a useful tool that demonstrates that the behavior of computed solutions is similar to what is predicted by rigorous analysis discussed in the previous sections. 

In each of the examples below, we consider a circular domain $\Omega$ of radius $1$ centered at the origin and set $\varepsilon=0.02.$ We assume $\alpha=\frac{\pi}{3}$ so that $C_\alpha=\frac{5}{9}$ and consider two cases: $s=1$ and $s=0.72$ corresponding to a situation described in part (a) and (b) of Theorem \ref{main}, respectively.

\subsection{Boundary Data of Degree One}

Here we suppose that $g(x)=\frac{x}{|x|}$ on $\partial\Omega$ so that $\deg{g}=1$.   

First, let $s=1$. According to Theorem \ref{main}, the minimizers of $E_{\eps}^{g,\alpha}$ must converge to an $\mathbb{S}^1-$valued harmonic map with a single vortex in the interior of the domain $\Omega$. The numerically computed critical point of $E_{\eps}^{g,\alpha}$ exhibits this feature as is shown in Figure~\ref{fig:n1}.
\begin{figure}
\begin{center}
\includegraphics[scale=.56]{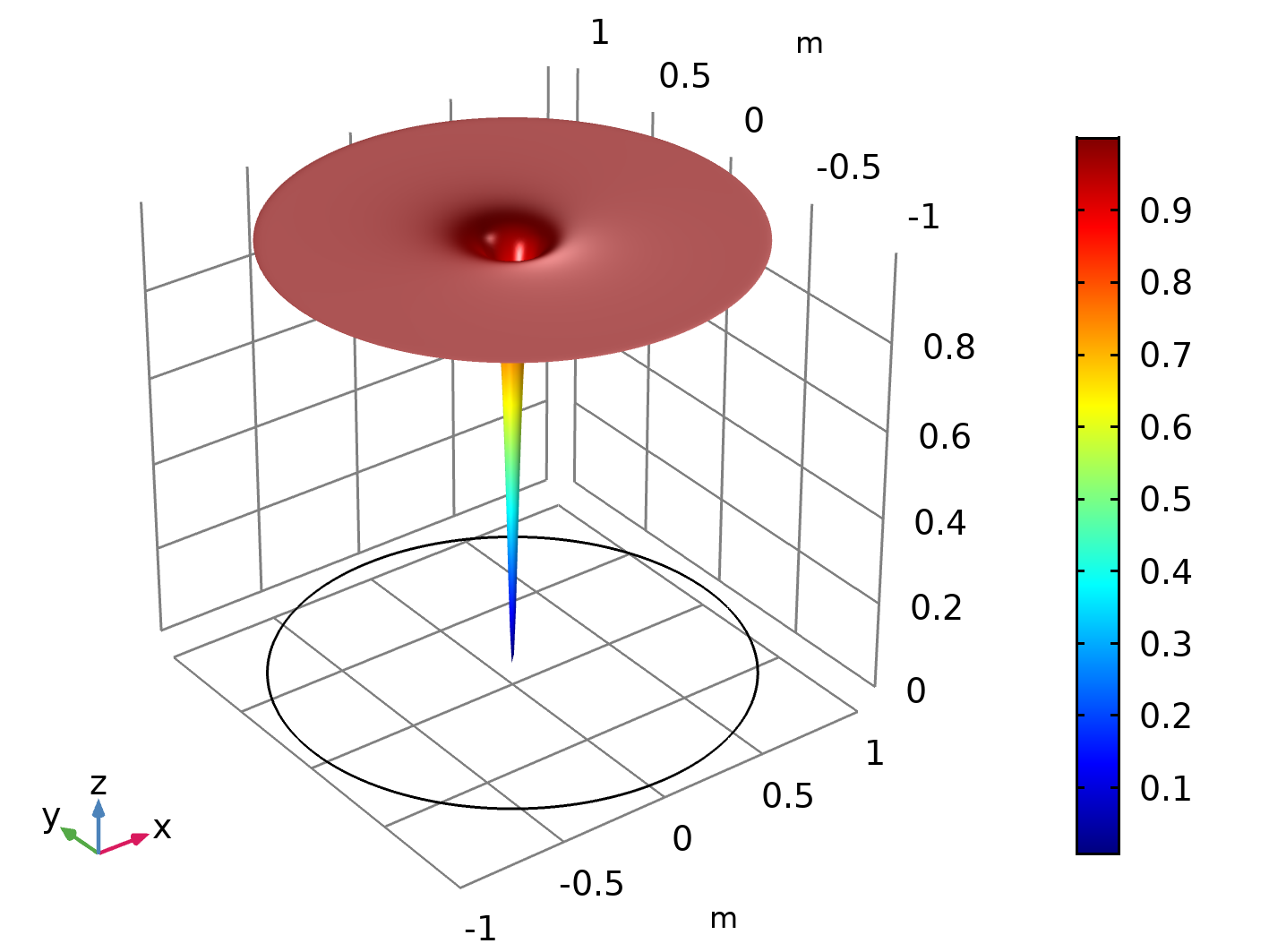}\qquad
\includegraphics[scale=.56]{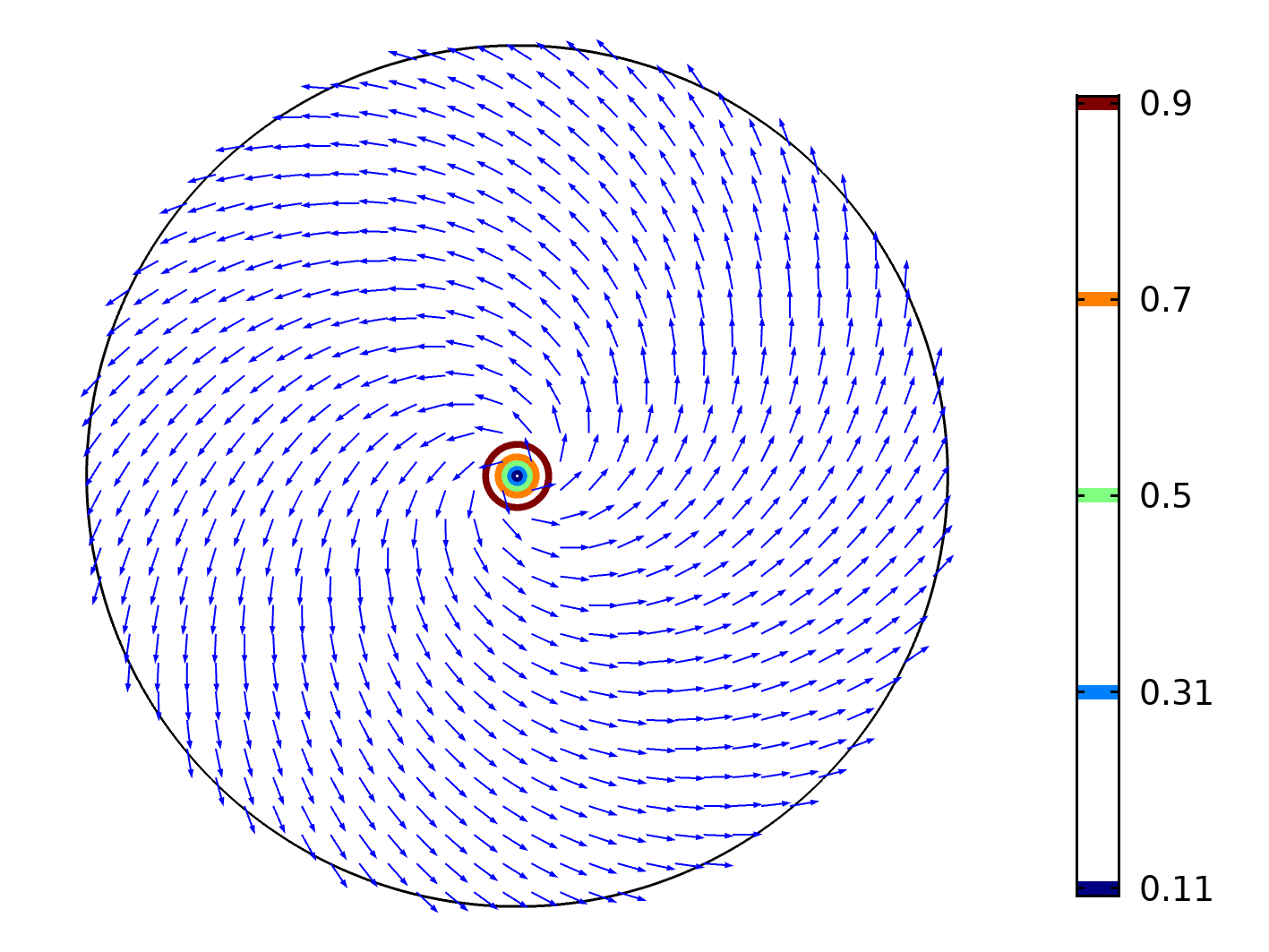} \\ \bigskip \includegraphics[scale=.56]{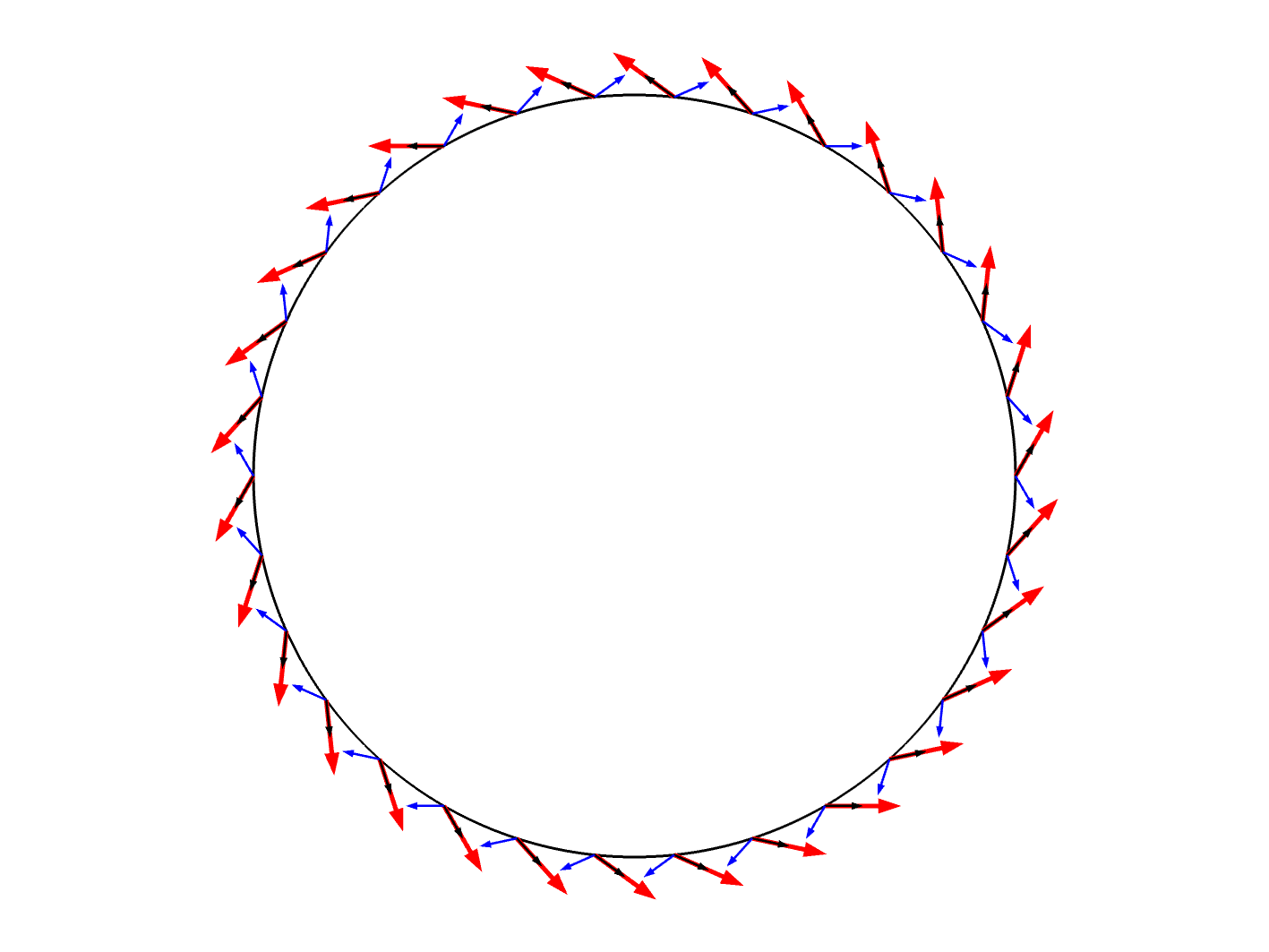}
\caption{A critical point of $E_{\eps}^{g,\alpha}$ for $g(x)=\frac{x}{|x|}$, $\alpha=\frac{\pi}{3}$, and $s=1$. Top left: The plot of $|u|$; Top right: The vector field $u$. Contour lines of $|u|$ are depicted to indicate the location of the singularity; Bottom: The restriction of the vector field $u$ to $\partial\Omega$. Here $u$ is shown in red, $e^{-i\pi/3}g$ and $e^{i\pi/3}g$ are shown in blue and black, respectively.}\label{fig:n1}
\end{center}
\end{figure}
Note that the absence of boundary singularities corresponds to $u$ being continuous on the boundary then $u$ must essentially coincide everywhere on $\partial\Omega$ with $e^{i\pi/3}g$ (or $e^{-i\pi/3}g$) in order to minimize the surface energy. 

Recall that in Section \ref{sec:nema} we established the relationship between $u$ and the nematic director $n$.  We can now use \eqref{eq:pn} to find the distribution of the director in $\Omega$. This distribution is depicted in Figure~\ref{fig:n2} and is characterized by the presence of a single disclination of degree $1/2$ at the origin.
\begin{figure}
\begin{center}
\includegraphics[scale=.56]{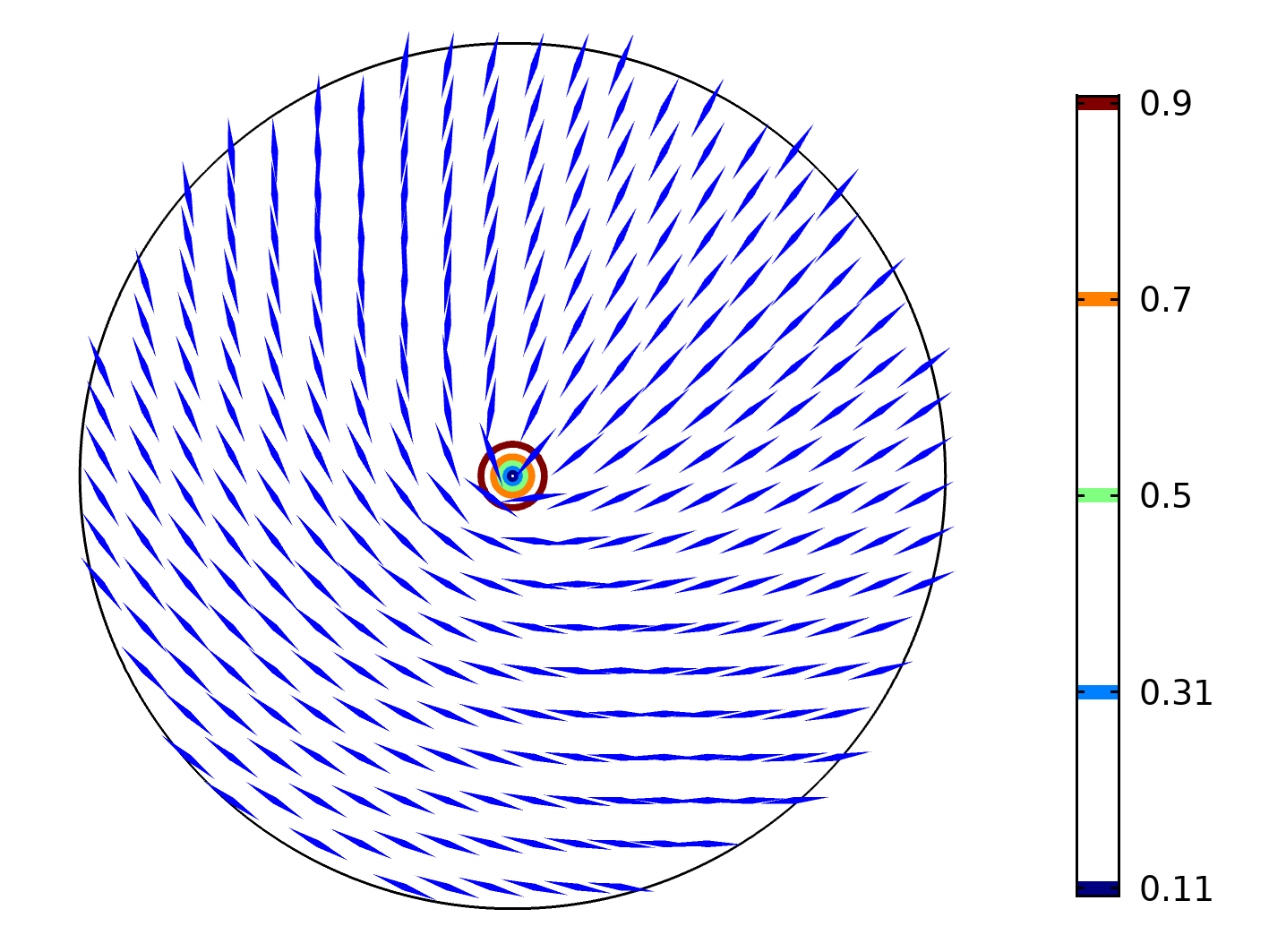}
\caption{The director field $n$ for $g(x)=\frac{x}{|x|}$, $\alpha=\frac{\pi}{3}$, and $s=1$. Contour lines of $|u|$ are shown to indicate the location of the singularity.}\label{fig:n2}
\end{center}
\end{figure}

Now, let $s=0.72$. Since $0.72<0.9=\frac{1}{2C_\alpha}$, from Theorem \ref{main} we expect that a numerically computed critical point of $E_{\eps}^{g,\alpha}$ should have one light and one heavy boojum on $\partial\Omega$. This is indeed the case for a critical point in Figure~\ref{fig:n3}.
\begin{figure}
\begin{center}
\includegraphics[scale=.56]{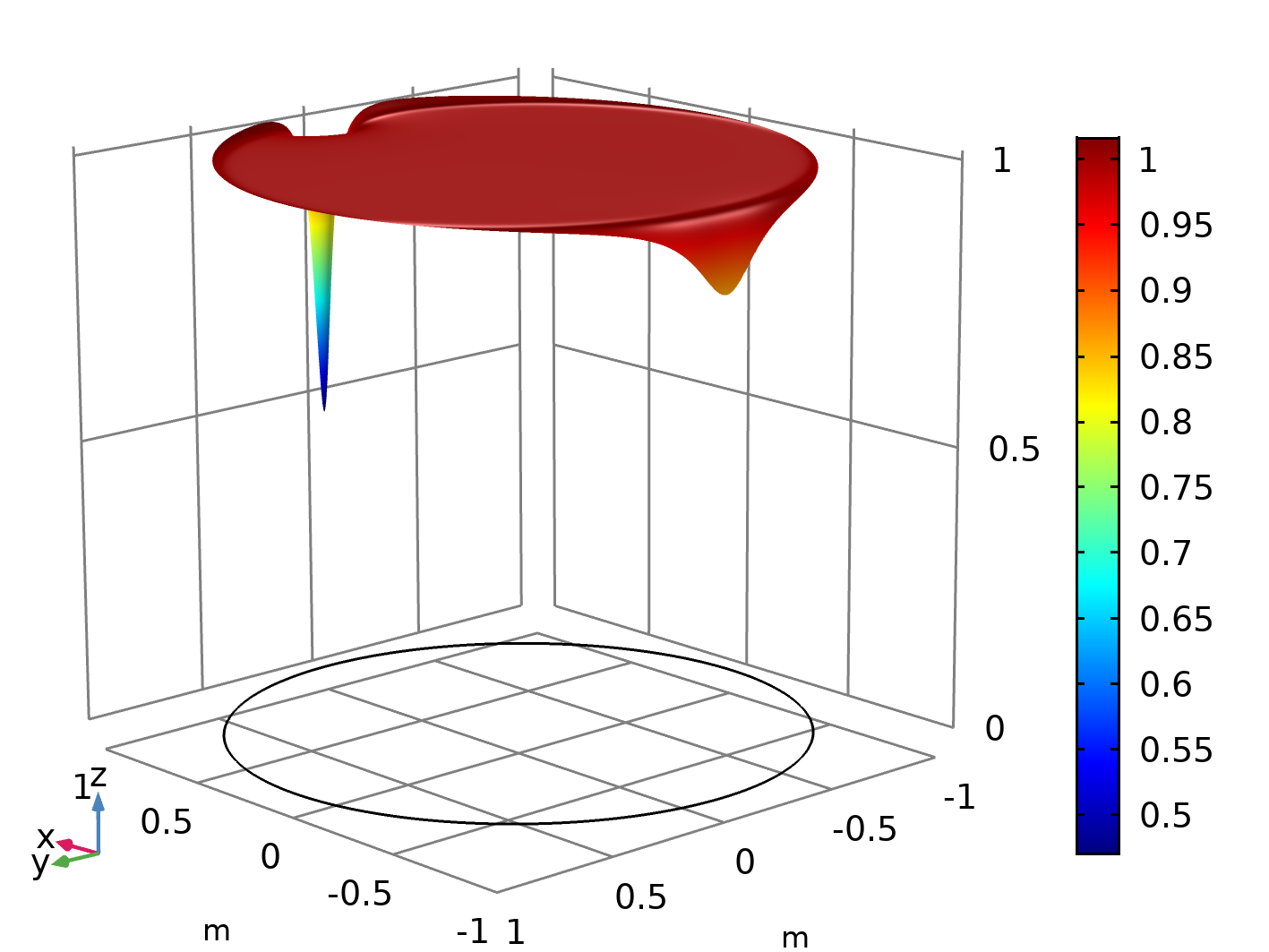}\qquad
\includegraphics[scale=.56]{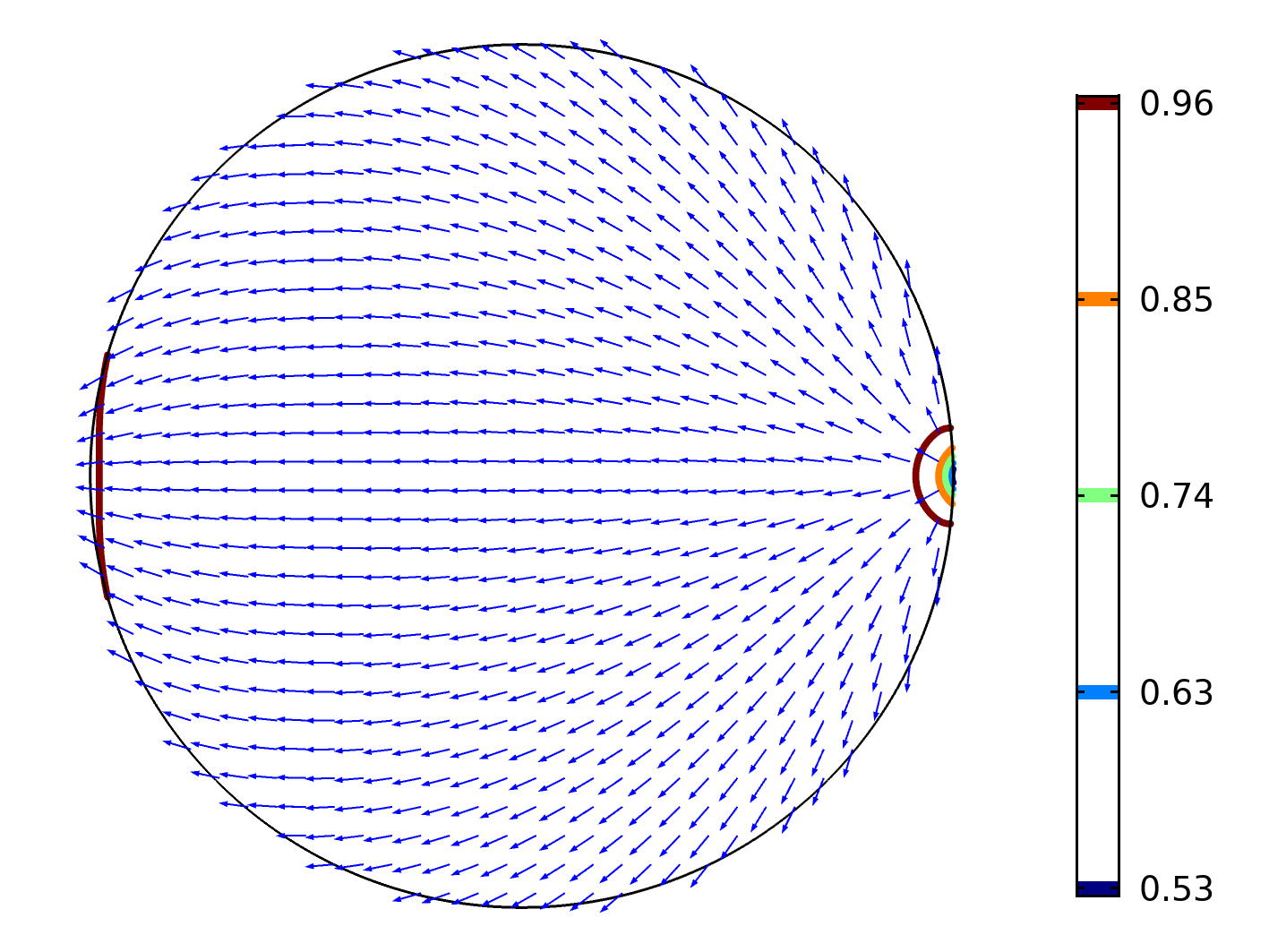} \\ \bigskip \includegraphics[scale=.56]{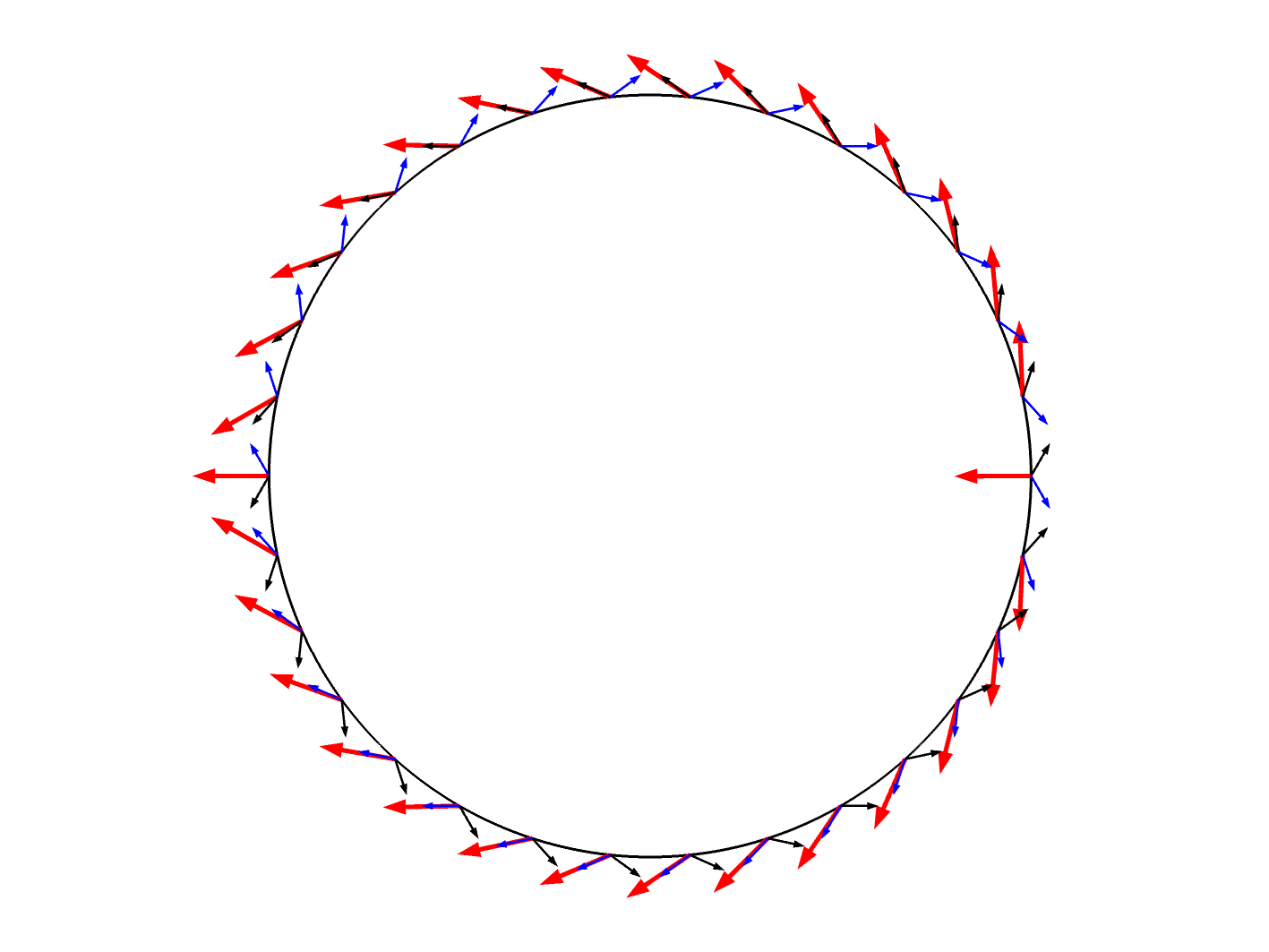}
\caption{A critical point of $E_{\eps}^{g,\alpha}$ for $g(x)=\frac{x}{|x|}$, $\alpha=\frac{\pi}{3}$, and $s=0.72$. Top left: The plot of $|u|$; Top right: The vector field $u$. Contour lines of $|u|$ are depicted to indicate the locations of the singularities; Bottom: The restriction of the vector field $u$ to $\partial\Omega$. Here $u$ is shown in red, $e^{-i\pi/3}g$ and $e^{i\pi/3}g$ are shown in blue and black, respectively.}\label{fig:n3}
\end{center}
\end{figure}
The light boojum corresponds to the shallower depression of $|u|$ on the boundary in the top left inset in Figure~\ref{fig:n3} and it is placed antipodally from the heavy boojum. Further, as the vector field is forced to switch its orientation with respect to $g$ on $\partial\Omega$ at boundary singularities, the bottom inset in Figure~\ref{fig:n3} demonstrates that $u$ "jumps" between $e^{-i\pi/3}g$ and $e^{i\pi/3}g$ as one traverses $\partial\Omega$. 

The distribution of the nematic director $n$ when $s=0.72$ is shown in Figure~\ref{fig:n4}.
\begin{figure}
\begin{center}
\includegraphics[scale=.56]{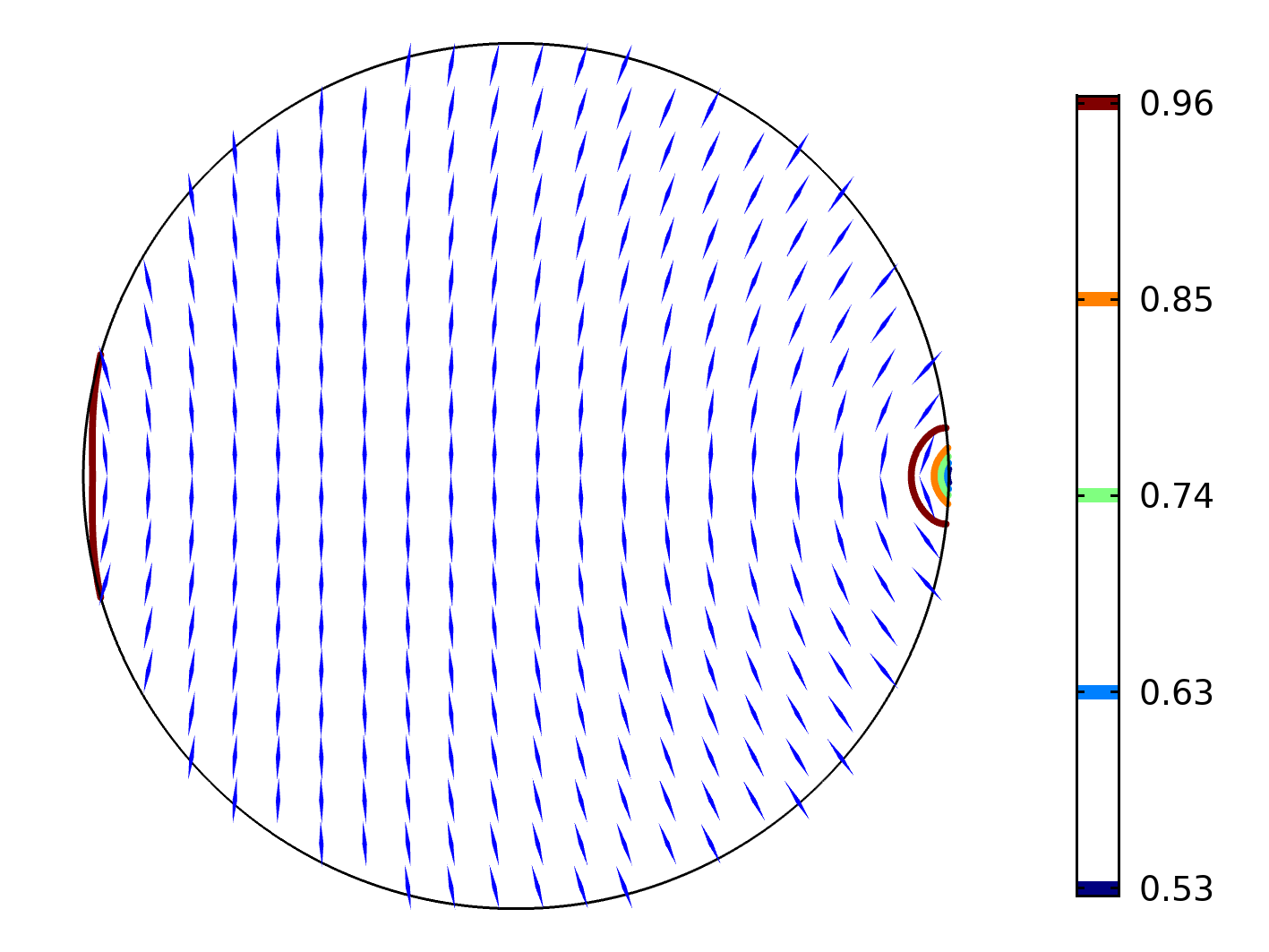}
\caption{The director field $n$ for $g(x)=\frac{x}{|x|}$, $\alpha=\frac{\pi}{3}$, and $s=0.72$. Contour lines of $|u|$ are shown to indicate the locations of the singularities.}\label{fig:n4}
\end{center}
\end{figure}

\subsection{Boundary Data of Degree Two}

Here we let $g(x)=\left(x_1^2-x_2^2,\,2x_1x_2\right)/|x|^2$ on $\partial\Omega$ so that $\deg{g}=2$.   

As expected, for $s=1$, the numerically computed critical point of $E_{\eps}^{g,\alpha}$ has two interior, degree one singularities as is shown in Figure~\ref{fig:n5}.
\begin{figure}
\begin{center}
\includegraphics[scale=.56]{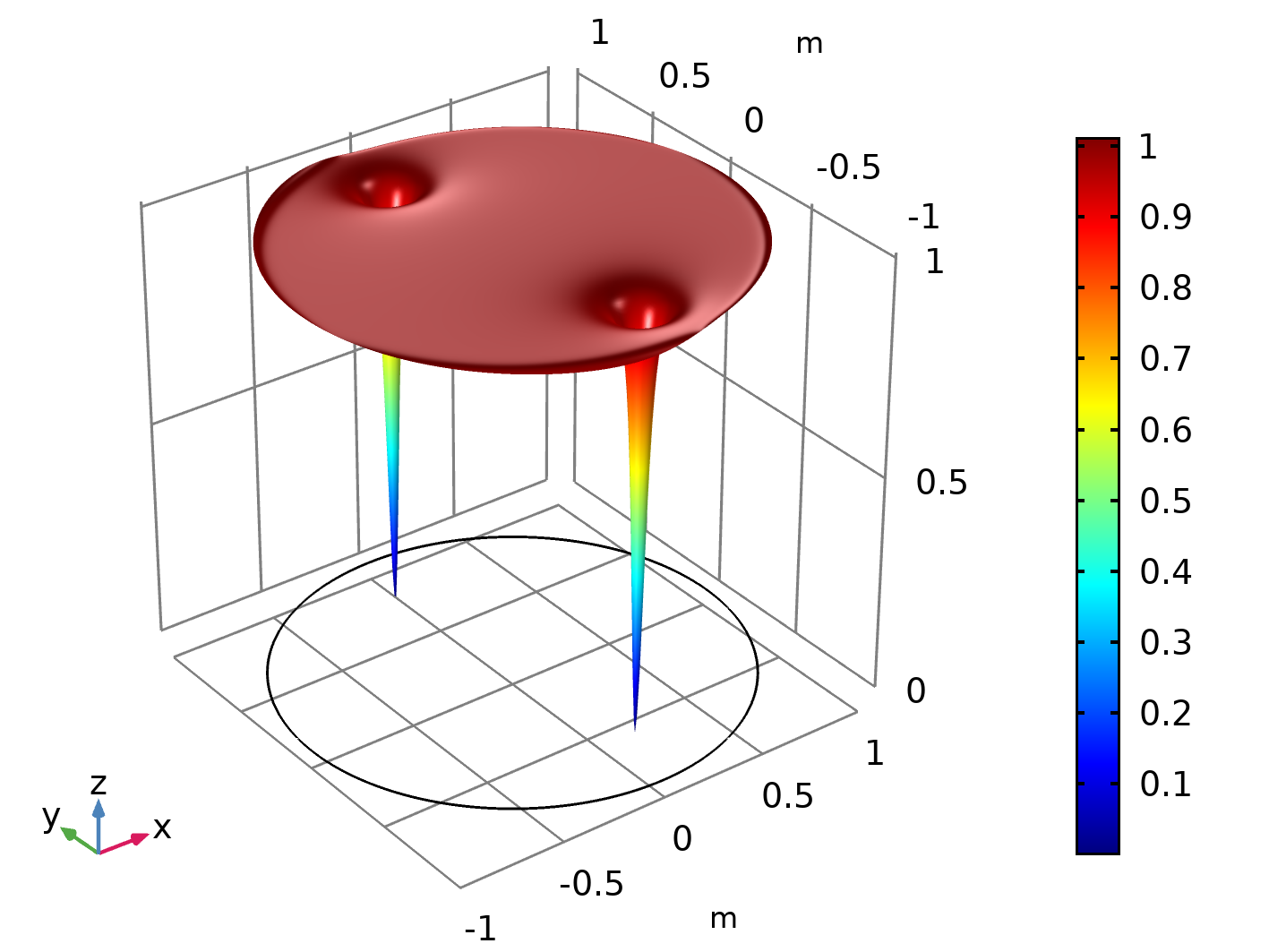}\qquad
\includegraphics[scale=.56]{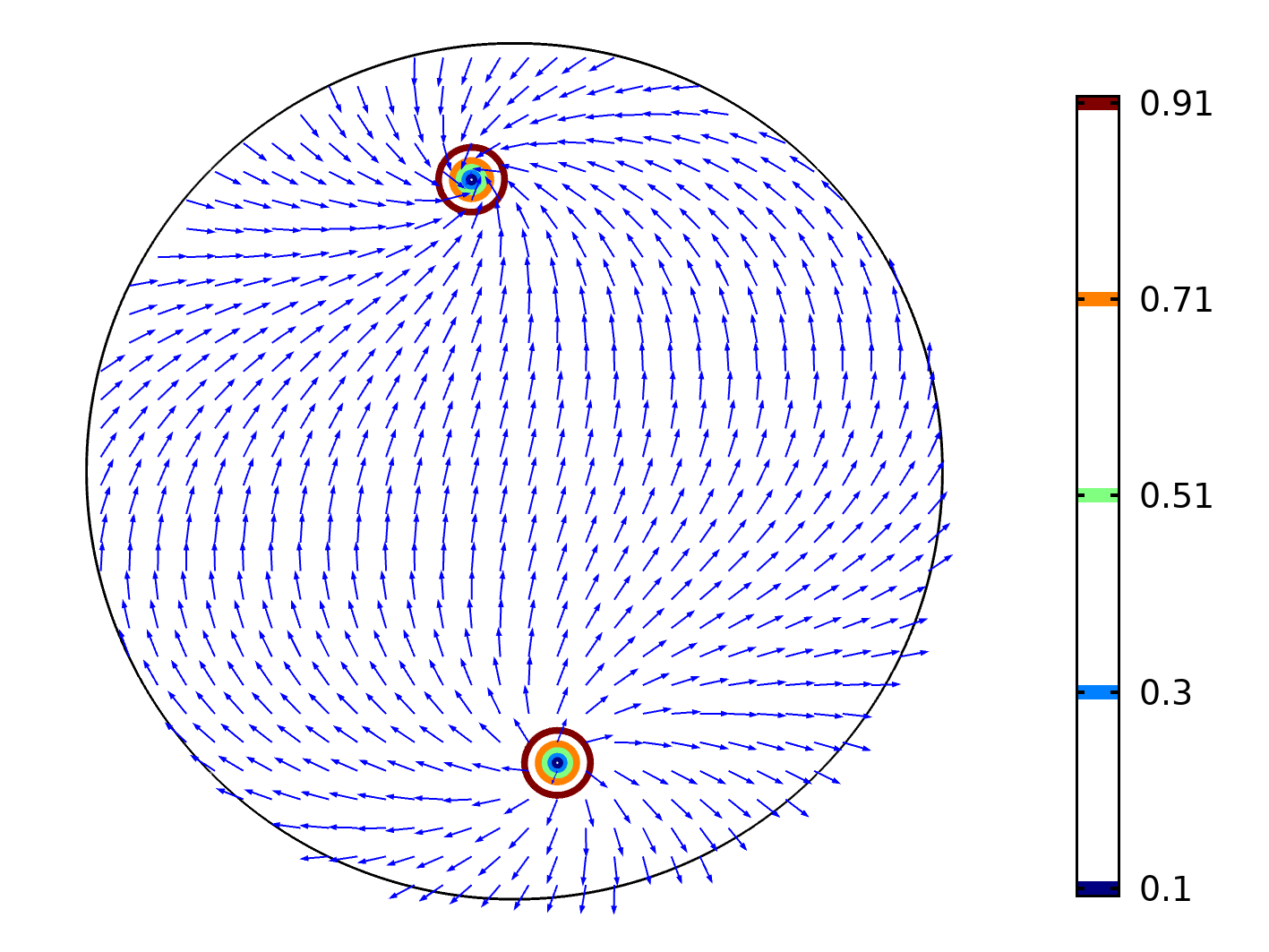} \\ \bigskip \includegraphics[scale=.56]{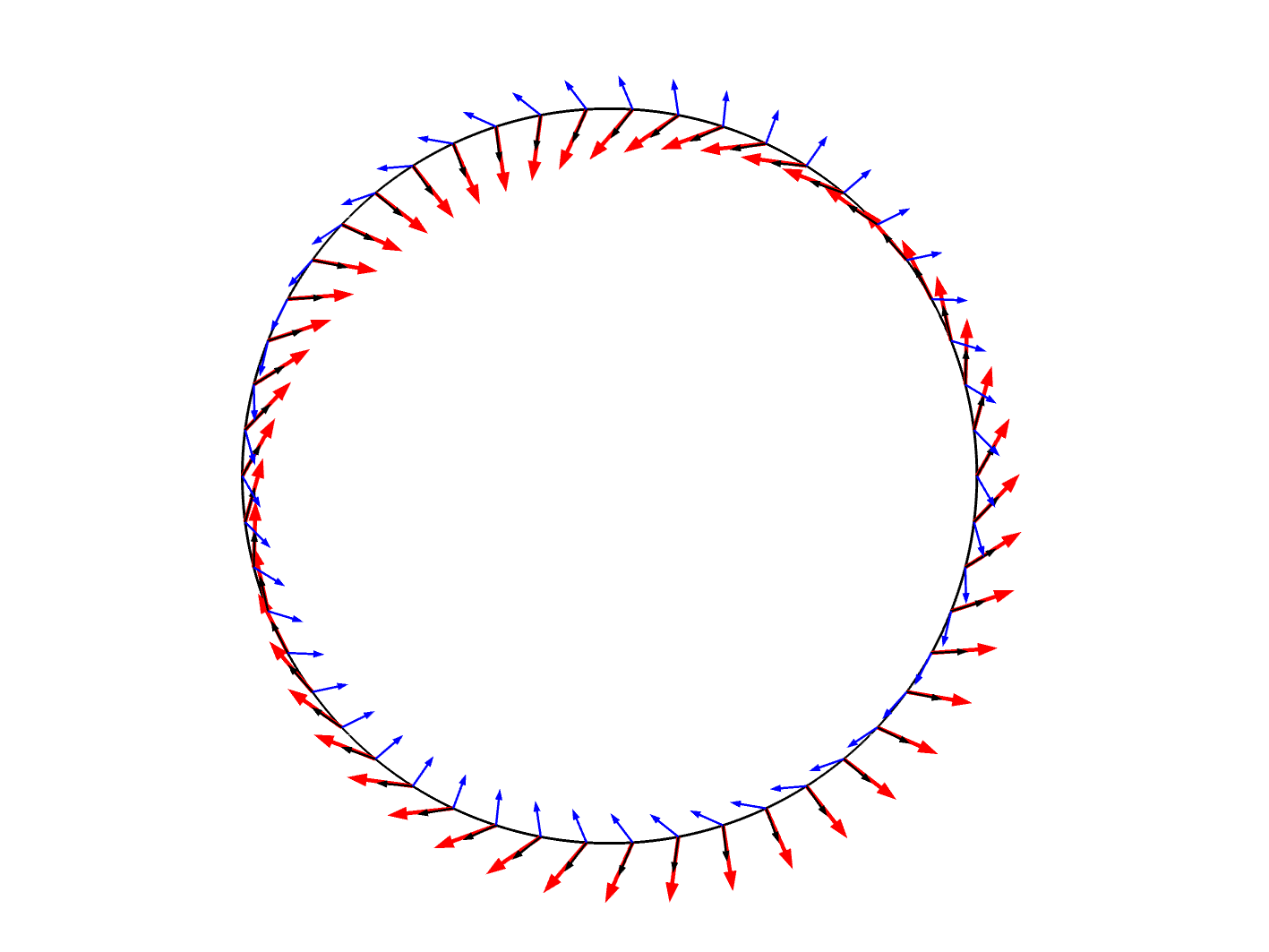}
\caption{A critical point of $E_{\eps}^{g,\alpha}$ for $g(x)=\left(x_1^2-x_2^2,\,2x_1x_2\right)/|x|^2$, $\alpha=\frac{\pi}{3}$, and $s=1$. Top left: The plot of $|u|$; Top right: The vector field $u$. Contour lines of $|u|$ are depicted to indicate the location of the singularity; Bottom: The restriction of the vector field $u$ to $\partial\Omega$. Here $u$ is shown in red, $e^{-i\pi/3}g$ and $e^{i\pi/3}g$ are shown in blue and black, respectively.}\label{fig:n5}
\end{center}
\end{figure}
The same is true for the nematic director (Figure~\ref{fig:n6}) that now has two degree $1/2$ singularities in the interior of $\Omega$.
\begin{figure}
\begin{center}
\includegraphics[scale=.56]{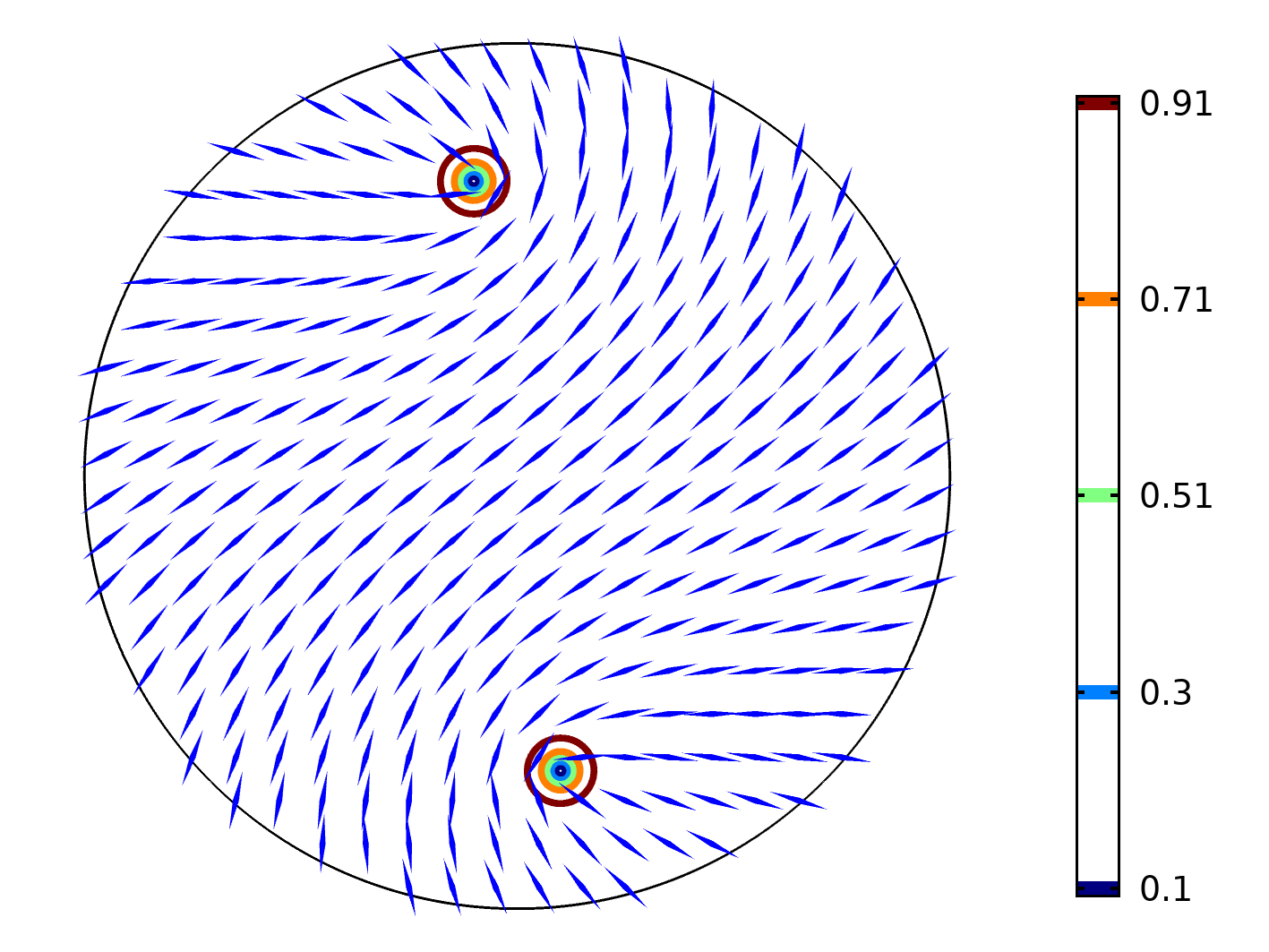}
\caption{The director field $n$ for $g(x)=\left(x_1^2-x_2^2,\,2x_1x_2\right)/|x|^2$, $\alpha=\frac{\pi}{3}$, and $s=1$. Contour lines of $|u|$ are shown to indicate the location of the singularity.}\label{fig:n6}
\end{center}
\end{figure}

For $s=0.72$, a numerically computed critical point of $E_{\eps}^{g,\alpha}$ has two light and two heavy boojums on $\partial\Omega$ as demonstrated in Figure~\ref{fig:n7}. The boojums types are interleaved as one traverses the boundary and the boojums are equidistant from each other.  
\begin{figure}
\begin{center}
\includegraphics[scale=.56]{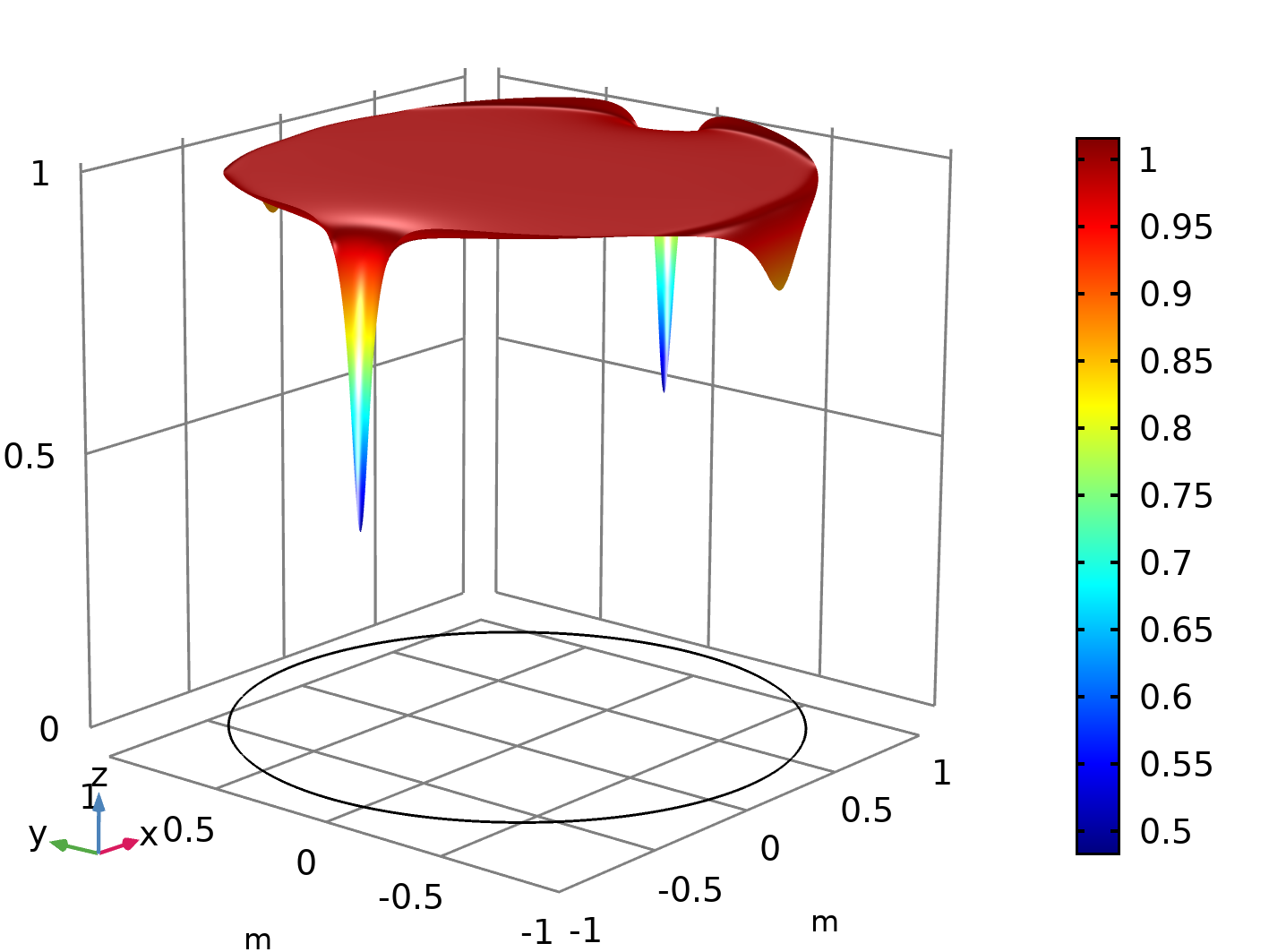}\qquad
\includegraphics[scale=.56]{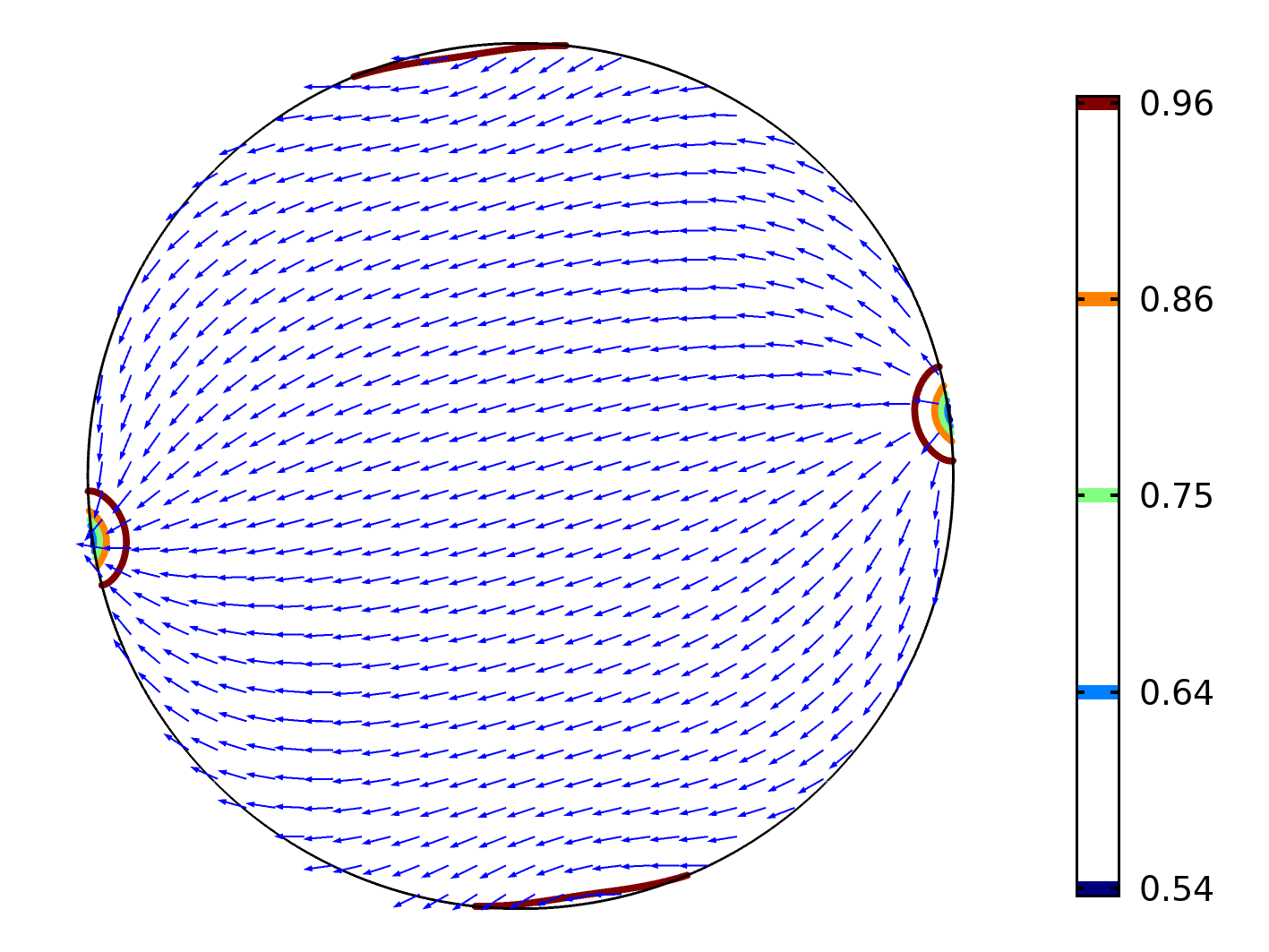} \\ \bigskip \includegraphics[scale=.56]{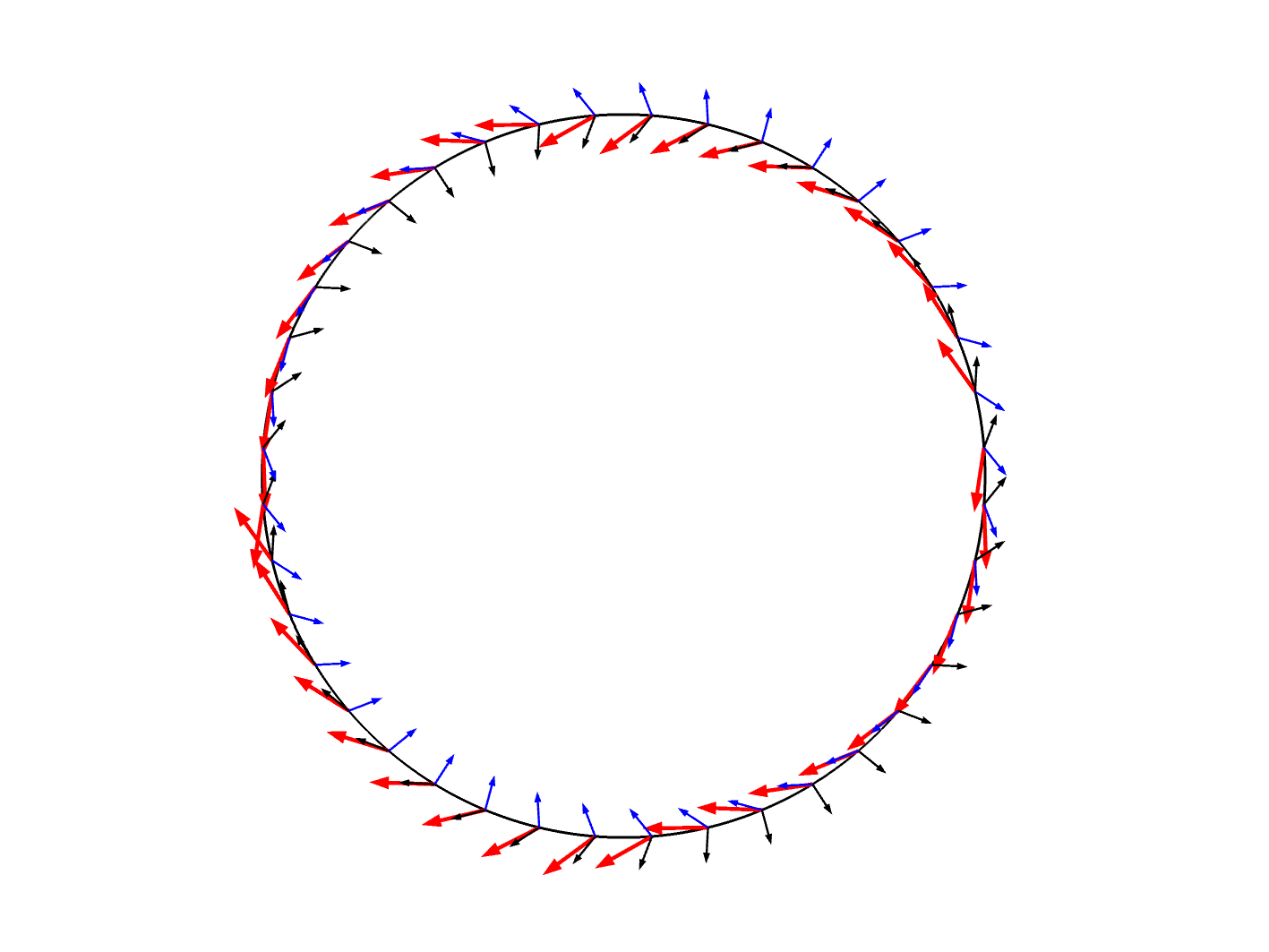}
\caption{A critical point of $E_{\eps}^{g,\alpha}$ for $g(x)=\left(x_1^2-x_2^2,\,2x_1x_2\right)/|x|^2$, $\alpha=\frac{\pi}{3}$, and $s=0.72$. Top left: The plot of $|u|$; Top right: The vector field $u$. Contour lines of $|u|$ are depicted to indicate the locations of the singularities; Bottom: The restriction of the vector field $u$ to $\partial\Omega$. Here $u$ is shown in red, $e^{-i\pi/3}g$ and $e^{i\pi/3}g$ are shown in blue and black, respectively.}\label{fig:n7}
\end{center}
\end{figure}

The distribution of the nematic director $n$ when $s=0.72$ is shown in Figure~\ref{fig:n8}. The director has four boundary singularities, two light and two heavy boojums.

The original motivation for this study stems from experimental results of Volovik and Lavrentovich \cite{LV}, for the case of a (three-dimensional) nematic ball.  Although our treatment in this paper is restricted to a two-dimensional, thin film geometry, we may still observe the resemblance of the configuration in Figure~\ref{fig:n8} with that in Fig. 7b in \cite{LV}, which shows two polar point defects and an equatorial disclination ring on the surface of a spherical particle. Because the 3D configuration in \cite{LV} is invariant with respect to both the axial and mirror symmetries in each cross-section of the particle that contains the axis of symmetry, it displays four surface point singularities separated by 90 degrees angles. Topologically this is the same situation as in Figure~\ref{fig:n8}: there are two heavy boojums that are the traces of the singularities at the poles, and two light boojums corresponding to the intersection between the circular cross-section and the equatorial disclination ring. Here both the equatorial ring and the light boojums serve the same purpose of unwinding the extra phase gained at the poles due to the heavy boojums.

Furthermore, the experimental work in \cite{LV} also indicates that when the angle of inclination between the director and the normal on the surface of the particle is close to 90 degrees , it might be reasonable to seek minimizers of the 3D problem in the class of functions that possess both axial and inversion symmetries. This is the approach that we undertook recently in a separate work \cite{ABGL}, and we conjecture that these techniques can be extended to the present problem in 3D.

\begin{figure}
\begin{center}
\includegraphics[scale=.56]{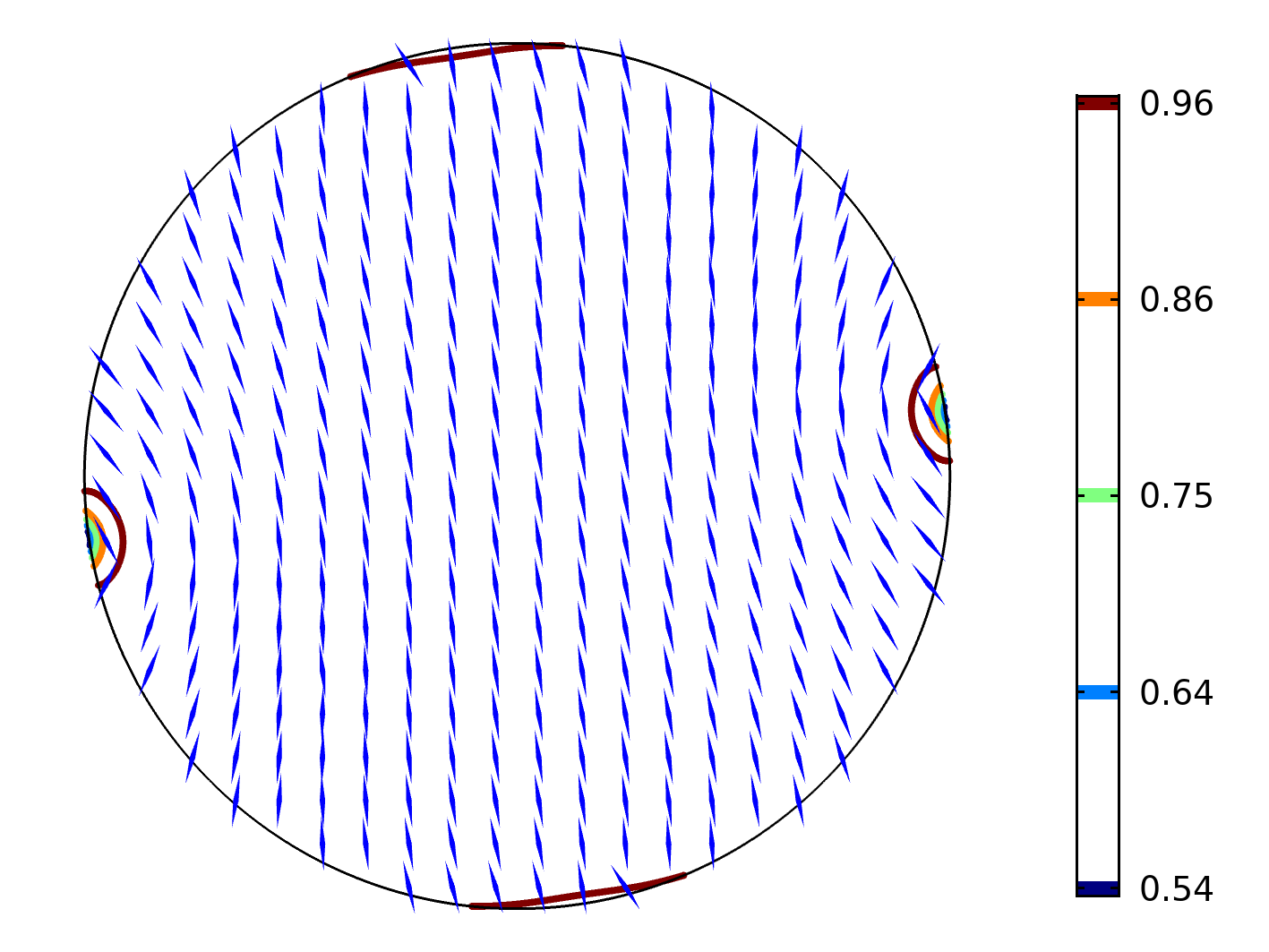}
\caption{The director field $n$ for $g(x)=\left(x_1^2-x_2^2,\,2x_1x_2\right)/|x|^2$, $\alpha=\frac{\pi}{3}$, and $s=0.72$. Contour lines of $|u|$ are shown to indicate the locations of the singularities.}\label{fig:n8}
\end{center}
\end{figure}

\clearpage

\bibliographystyle{siam}

\begin{thebibliography}{10}


\bibitem{ABGS}
{\sc S.~Alama, L.~Bronsard, and B.~Galv\~{a}o Sousa}, {\em Weak anchoring for a
  two-dimensional liquid crystal}, Nonlinear Anal., 119 (2015), pp.~74--97.
  
\bibitem{ABGL} S.~Alama, L.~Bronsard, D.~Golovaty, X.~Lamy, {\em Saturn ring defect around a spherical particle immersed in nematic liquid crystal}, preprint 2020.

\bibitem{BPP}
{\sc P.~Bauman, J.~Park, and D.~Phillips}, {\em Analysis of nematic liquid
  crystals with disclination lines}, Arch. Ration. Mech. Anal., 205 (2012),
  pp.~795--826.

\bibitem{BBH2}
{\sc F.~Bethuel, H.~Brezis, and F.~H{\'e}lein}, {\em Ginzburg-{L}andau
  vortices}, Progress in Nonlinear Differential Equations and their
  Applications, 13, Birkh\"auser Boston Inc., Boston, MA, 1994.

\bibitem{DKMO}
{\sc A.~Desimone, R.~V. Kohn, S.~M{\"u}ller, and F.~Otto}, {\em A reduced
  theory for thin-film micromagnetics}, Comm. Pure Appl. Math., 55 (2002),
  pp.~1408--1460.

\bibitem{GSM}
{\sc D.~Golovaty, J.~A. Montero, and P.~Sternberg}, {\em Dimension reduction
  for the {L}andau-de {G}ennes model in planar nematic thin films}, J.
  Nonlinear Sci., 25 (2015), pp.~1431--1451.

\bibitem{GSM2}
\leavevmode\vrule height 2pt depth -1.6pt width 23pt, {\em Dimension reduction
  for the {L}andau--de {G}ennes model on curved nematic thin films}, J.
  Nonlinear Sci., 27 (2017), pp.~1905--1932.

\bibitem{Jerrard}
{\sc R.~L. Jerrard}, {\em Lower bounds for generalized {G}inzburg-{L}andau
  functionals}, SIAM J. Math. Anal., 30 (1999), pp.~721--746.

\bibitem{Ku}
{\sc M.~Kurzke}, {\em Boundary vortices in thin magnetic films}, Calc. Var.
  Partial Differential Equations, 26 (2006), pp.~1--28.

\bibitem{apala_zarnescu_01}
{\sc A.~Majumdar and A.~Zarnescu}, {\em Landau-{D}e {G}ennes theory of nematic
  liquid crystals: the {O}seen-{F}rank limit and beyond}, Arch. Ration. Mech.
  Anal., 196 (2010), pp.~227--280.

\bibitem{Moser}
{\sc R.~Moser}, {\em On the energy of domain walls in ferromagnetism},
  Interfaces Free Bound., 11 (2009), pp.~399--419.

\bibitem{GNSV}
{\sc D.~Golovaty, M.~Novack, P.~Sternberg, and R.~Venkatraman}, {\em A model problem for nematic-isotropic transitions with highly disparate elastic constants},
 arXiv 1811.12586 (2018).
  
\bibitem{GKLNS}
{\sc D.~Golovaty, Y.-K.~Kim, O.~D.~Lavrentovich, M.~Novack, and P.~Sternberg}, {\em A model problem for nematic-isotropic transitions with highly disparate elastic constants},
 arXiv 1801.04477 (2019).
 
 \bibitem{LV} G.~E.~Volovik, O.~D.~ Lavrentovih, {\em Topological dynamics of defects:  boojums in nematic drops,} Zh. Eksp. Teor. Fiz. 85, 1997-2010 (December 1983).

\bibitem{MN}
{\sc N.~Mottram and C.~Newton}, {\em Introduction to Q-tensor theory},
  University of Strathclyde, Department of Mathematics Research Reports,
  2004:10 (2004).

\bibitem{Riviere}
{\sc T.~Rivi{\`e}re}, {\em Asymptotic analysis for the {G}inzburg-{L}andau
  equations}, Boll. Unione Mat. Ital. Sez. B Artic. Ric. Mat. (8), 2 (1999),
  pp.~537--575.

\bibitem{Sandier}
{\sc E.~Sandier}, {\em Lower bounds for the energy of unit vector fields and
  applications}, J. Funct. Anal., 152 (1998), pp.~379--403.

\bibitem{Struwe}
{\sc M.~Struwe}, {\em On the asymptotic behavior of minimizers of the
  {G}inzburg-{L}andau model in {$2$} dimensions}, Differential Integral
  Equations, 7 (1994), pp.~1613--1624.

\bibitem{TAS12}
{\sc M.~Tasinkevych, N.~Silvestre, and M.~T. Da~Gama}, {\em Liquid crystal
  boojum-colloids}, New Journal of Physics, 14 (2012), p.~073030.

\bibitem{Sluckin}
{\sc P.~I.~C. Teixeira, T.~J. Sluckin, and D.~E. Sullivan}, {\em Landau–de
  gennes theory of anchoring transitions at a nematic liquid
  crystal–substrate interface}, Liquid Crystals, 14 (1993), pp.~1243--1253.

\bibitem{VL83}
{\sc G.~Volovik and O.~Lavrentovich}, {\em Topological dynamics of defects:
  boojums in nematic drops}, Zh Eksp Teor Fiz, 85 (1983), pp.~1997--2010.
  
  \bibitem{COMSOL}
{\em {COMSOL} {Multiphysics\textregistered} {v. 5.3}}.
\newblock http://www.comsol.com/.
\newblock {COMSOL} {AB}, {Stockholm}, {Sweden}.


\end{thebibliography}

\end{document}